  \RequirePackage{fix-cm}
  \documentclass[smallcondensed,hyperref]{svjour3}     % onecolumn (ditto)
  \smartqed  % flush right qed marks, e.g. at end of proof
  \usepackage{graphicx}
  \usepackage{mathptmx}      % use Times fonts if available on your TeX system
  \DeclareMathAlphabet{\mathcal}{OMS}{cmsy}{m}{n} % because I want mathcal!

\usepackage{amsmath}
\usepackage{amsfonts}
\usepackage{amssymb}
\usepackage{graphicx} % for includegraphics
\graphicspath{{results/}}
\usepackage{xspace} % needed for \eg, \ie, \etc
\usepackage{threeparttable} % for tables with footnotes
\usepackage{subfigure}
\usepackage[dvipsnames]{xcolor}
\usepackage{hyperref}
\hypersetup{
  colorlinks=true,
  linkcolor=red,
  citecolor=blue,
  urlcolor=blue
}

\newcommand{\bm}[1]{\vec{#1}}
\newcommand{\ds}[0]{\displaystyle}
\newcommand{\poly}[1]{\mathbb{P}_{#1}}

\newcommand{\fnc}[1]{\ensuremath{\mathcal{#1}}}
\newcommand{\bfnc}[1]{\ensuremath{\bm{\mathcal{#1}}}}
\newcommand{\ufnc}[0]{\fnc{U}}
\newcommand{\vfnc}[0]{\fnc{V}}
\newcommand{\ffnc}[0]{\fnc{F}}

\newcommand{\Tmesh}[0]{\ensuremath{\mathcal{T}_{h}}}
\newcommand{\nodes}[0]{\ensuremath{\Xi_{\kappa}}}

\newcommand{\mat}[1]{\ensuremath{\tens{#1}}}
\newcommand{\Rs}[0]{\ensuremath{\mat{R}}}

\newcommand{\Qxi}[0]{\ensuremath{\mat{Q}_{\xi}}}
\newcommand{\Qeta}[0]{\ensuremath{\mat{Q}_{\eta}}}
\newcommand{\Qx}[0]{\ensuremath{\mat{Q}_{x}}}
\newcommand{\Qy}[0]{\ensuremath{\mat{Q}_{y}}}
\newcommand{\Sx}[0]{\ensuremath{\mat{S}_{x}}}
\newcommand{\Sy}[0]{\ensuremath{\mat{S}_{y}}}
\newcommand{\Dxi}[0]{\ensuremath{\mat{D}_{\xi}}}
\newcommand{\Deta}[0]{\ensuremath{\mat{D}_{\eta}}}
\newcommand{\Dx}[0]{\ensuremath{\mat{D}_{x}}}
\newcommand{\Dy}[0]{\ensuremath{\mat{D}_{y}}}
\newcommand{\Ex}[0]{\ensuremath{\mat{E}_{x}}}

\newcommand{\overbar}[1]{\mkern 1.5mu\overline{\mkern-1.5mu#1\mkern-1.5mu}\mkern 1.5mu}
\newcommand{\barDx}[0]{\overbar{\mat{D}}_{x}}
\newcommand{\barDy}[0]{\overbar{\mat{D}}_{y}}

\newcommand{\barSx}[0]{\overbar{\mat{S}}_{x}}
\newcommand{\barSy}[0]{\overbar{\mat{S}}_{y}}
\newcommand{\barH}[0]{\overbar{\mat{H}}}
\newcommand{\barHk}[0]{\ensuremath{\overbar{\mat{H}}_{\kappa}}}
\newcommand{\barRsk}[0]{\ensuremath{\overbar{\mat{R}}_\kappa}}
\newcommand{\barPrjk}[0]{\ensuremath{\overbar{\mat{P}}_{\kappa}}}
\newcommand{\barSxk}[0]{\overbar{\mat{S}}_{x,\kappa}}
\newcommand{\barSyk}[0]{\overbar{\mat{S}}_{y,\kappa}}

\renewcommand{\H}[0]{\ensuremath{\mat{H}}}

\newcommand{\nk}[0]{\ensuremath{n_{\kappa}}}
\newcommand{\Rsk}[0]{\ensuremath{\mat{R}_\kappa}}
\newcommand{\Qxk}[0]{\ensuremath{\mat{Q}_{\xi,\kappa}}}
\newcommand{\Exk}[0]{\ensuremath{\mat{E}_{\xi,\kappa}}}
\newcommand{\Sxk}[0]{\ensuremath{\mat{S}_{\xi,\kappa}}}
\newcommand{\Hk}[0]{\ensuremath{\mat{H}_{\kappa}}}
\newcommand{\Dxk}[0]{\ensuremath{\mat{D}_{\xi,\kappa}}}
\newcommand{\Disk}[0]{\ensuremath{\mat{M}_{\kappa}^{\mat{D}}}}
\newcommand{\Prjk}[0]{\ensuremath{\mat{P}_{\kappa}}}
\newcommand{\Mlps}[0]{\ensuremath{\mat{M}_{\kappa}^{\mat{P}}}}
\newcommand{\Rgk}[0]{\ensuremath{\mat{R}_{\gamma,\kappa}}}
\newcommand{\Hg}[0]{\ensuremath{\mat{H}_{\gamma}}}

\newcommand{\Id}[0]{\mat{I}}
\newcommand{\Fx}[0]{\mat{F}_{x}}
\newcommand{\Fy}[0]{\mat{F}_{y}}

\newcommand{\uh}[0]{\ensuremath{\bm{u}_{h}}}
\newcommand{\vh}[0]{\ensuremath{\bm{v}_{h}}}
\newcommand{\wh}[0]{\ensuremath{\bm{w}_{h}}}
\newcommand{\fh}[0]{\ensuremath{\bm{f}_{h}}}
\newcommand{\uk}[0]{\ensuremath{\bm{u}_{\kappa}}}
\newcommand{\vk}[0]{\ensuremath{\bm{v}_{\kappa}}}
\newcommand{\wk}[0]{\ensuremath{\bm{w}_{\kappa}}}
\newcommand{\fk}[0]{\ensuremath{\bm{f}_{\kappa}}}
\newcommand{\utildek}[0]{\ensuremath{\tilde{\bm{u}}_{\kappa}}}
\newcommand{\unode}[1]{\bm{u}_{#1}}  
  
\newcommand{\flux}[1]{\bfnc{F}^{\star}_{#1}}

\DeclareMathOperator{\mydiag}{diag}

\newcommand{\etal}[0]{{\em et~al.\@}\xspace}
\newcommand{\eg}[0]{{e.g.\@}\xspace}
\newcommand{\ie}[0]{{i.e.\@}\xspace}

\newcommand{\ignore}[1]{} % comment out large sections of code
\newcommand{\phm}[0]{\phantom{-}}

\definecolor{green}{RGB}{51,153,51}

\begin{document}
  
  \title{Entropy-stable, high-order summation-by-parts discretizations without interface penalties\thanks{This is a pre-print of an article published in The Journal of Scientific Computing. The final authenticated version is available online at: https://doi.org/10.1007/s10915-020-01154-8}} 
  
  
  \author{Jason E. Hicken}
  
  
  \institute{J. Hicken \at
                110 8th St Troy NY 12180 US\\
                Tel.: +1-518-276-4893\\
                Fax: +1-518-276-6025\\
                \email{hickej2@rpi.edu}
  }
  
  \date{Received: date / Accepted: date}

  \maketitle
  
  \begin{abstract}
    The paper presents high-order accurate, energy-, and entropy-stable
    discretizations constructed from summation-by-parts (SBP) operators.  Notably,
    the discretizations assemble global SBP operators and use continuous
    solutions, unlike previous efforts that use discontinuous SBP discretizations.
    Derivative-based dissipation and local-projection stabilization (LPS) are
    investigated as options for stabilizing the baseline discretization.  These
    stabilizations are equal up to a multiplicative constant in one dimension, but
    only LPS remains well conditioned for general, multidimensional SBP operators.
    Furthermore, LPS is able to take advantage of the additional nodes required by
    degree $2p$ diagonal-norms, resulting in an element-local stabilization with a
    bounded spectral radius.  An entropy-stable version of LPS is easily obtained
    by applying the projection on the entropy variables.  Numerical experiments
    with the linear-advection and Euler equations demonstrate the accuracy,
    efficiency, and robustness of the stabilized discretizations, and the
    continuous approach compares favorably with the more common discontinuous SBP
    methods.
  
  \end{abstract}

\section{Introduction}\label{sec:intro}

High-order discretizations have been put forward as a possible means of improving the efficiency of computational fluid dynamics (CFD) simulations.  The arguments in favor of high-order discretizations include both improved accuracy-per-degree-of-freedom as well as better cache usage on current and future architectures.  Despite these potential advantages, the use of high-order CFD remains uncommon in industry.  Mesh generation of curved elements is one bottleneck facing high-order methods, but the issue I focus on here is robustness: high-order discretizations have inherently less numerical dissipation, which makes them prone to instabilities, particularly for under-resolved flows.

Entropy stability offers one promising avenue for constructing robust, high-order CFD methods.  This is not a new idea.  For example, over thirty years ago, Hughes~\etal~\cite{Hughes1986new} presented a finite-element discretization of the compressible Navier-Stokes equations that satisfied the second-law of thermodynamics.  And in 1999, Barth~\cite{Barth1999numerical} extended this work to cover Galerkin-least-squares stabilizations and discontinuous Galerkin (DG) schemes.  However, these early examples make the assumption that the integrations present in the finite-element semi-linear forms are exact.  Exact integration is not possible, in general, for the Euler and Navier-Stokes equations, so these schemes must rely on potentially costly ``over-integration'' in practice.  Even then, the discrete schemes are not provably stable and may fail.

In light of the above, there has been growing interest in semi-discrete and fully-discrete high-order schemes that are provably entropy stable.  Fisher's thesis~\cite{Fisher2012thesis} represented a seminal contribution in this direction  --- see also \cite{Fisher2013discretely} and \cite{Fisher2013high}.  He showed that summation-by-parts (SBP) finite difference methods could be combined with entropy-conservative flux functions~\cite{Tadmor1987entropy,Tadmor2003entropy,Ismail2009affordable,Chandrashekar2015kinetic} to produce high-order entropy-stable schemes.  This was later extended to tensor-product spectral-element methods~\cite{Carpenter2014entropy,Parsani2016entropy} that also possess the summation-by-parts property~\cite{Gassner2013skew}.  Subsequently, SBP operators were generalized to simplex elements in~\cite{multiSBP} and later used to construct entropy-stable discretizations on triangular and tetrahedral grids~\cite{Chen2017entropy,Crean2018entropy}.

My objective in this paper is to extend the entropy-stable SBP-framework to continuous-Galerkin type discretizations.  Previous entropy-stable SBP discretizations have focused on discontinuous-Galerkin (DG)-type methods; even the finite-difference methods in~\cite{Fisher2012thesis}, \cite{Fisher2013discretely}, and \cite{Fisher2013high} used numerical flux functions embedded in penalty terms to couple blocks in multi-block grids.  My motivation for considering continuous SBP (C-SBP) discretizations, which were first proposed in~\cite{multiSBP}, is to reduce the computational cost of the residual evaluations by eliminating the interface penalties.

Stabilization is arguably the principal challenge in adapting the entropy-stable framework to C-SBP discretizations.  It is well-known that continuous Galkerin finite-element methods produce oscillatory solutions for hyperbolic partial-differential equations (PDEs) and require stabilization; examples of stabilizations include stream-line upwind Petrov-Galerkin (SUPG)~\cite{Brooks1982streamline}, Galerkin-least-squares (GLS)~\cite{Hughes1989gls}, variational multiscale~\cite{Hughes1995multiscale}, edge stabilization~\cite{Douglas1976interior}, and local-projection stabilization (LPS)~\cite{Becker2001finite}, to name a few.

Each stabilization that has been proposed has advantages and disadvantages.  Therefore, in order to identify a suitable choice for an entropy-stable C-SBP scheme, a list of desired properties is useful.
\begin{description} 
\item[Entropy stable:] It must be possible to make the stabilization provably entropy stable, at least at the semi-discrete level.  This rules out SUPG as a potential candidate, as well as other non-symmetric stabilizations.
\item[Well conditioned:] The stabilization should have a spectral radius that is comparable to D-SBP discretizations.  For example, edge stabilization based on jumps in the gradient~\cite{Douglas1976interior,Burman2004edge,Burman2006continuous} can be made entropy stable, but numerical experiments~\cite{Crean2016investigation} indicate that it has poor conditioning, especially as the discretization order increases.  Numerical experiments (not reported here) also suggest that GLS has poor conditioning.
\item[Element local:] I would like the stabilization to have a stencil that matches the stencil of the (unstabilized) discretization.  This requirement reduces parallel communication and memory requirements for the Jacobian.  It also greatly simplifies algorithmic differentiation, since coloring can be performed at the element level.  Edge stabilization and LPS are not element local, in general, while SUPG and GLS are.
\end{description}

To the best of my knowledge, no stabilization meets all the requirements listed above, so a compromise is necessary.  The solution that I advocate here is to sacrifice optimal approximation accuracy by increasing the number of nodes necessary for a degree $p$ basis.  The additional degrees of freedom enable the creation of an entropy-stable, well-conditioned LPS to be applied at the element level.  While sacrificing optimal approximation may be unpalatable to some, I believe this is an attractive holistic solution, since diagonal-norm SBP operators also typically require more nodes than a degree $p$ basis~\cite{multiSBP,Fernandez2017simultaneous}. 

I begin below by describing the C-SBP discretization and its stabilization in the context of the constant-coefficient, linear advection equation; see Section~\ref{sec:LPS_advec}.  In Section~\ref{sec:euler}, I review the entropy conservation of (unstabilized) C-SBP discretizations in the context of the Euler equations and show how LPS can be used to create an entropy-stable stabilization.  I detail the SBP and LPS operators and their construction in Section~\ref{sec:constructSBP}.  I verify the discretizations in Section~\ref{sec:results} by presenting some numerical experiments, and I conclude with a summary and discussion in Section~\ref{sec:conclude}. 

\section{Stabilization of continuous SBP discretizations: linear advection}\label{sec:LPS_advec}

In this section, I use the constant-coefficient linear advection equation to present the key ideas behind the proposed stabilization. As discussed in the introduction, these ideas are i) sacrificing optimal-polynomial approximation to achieve a localized stabilization, and ii) using a local-projection-based stabilization.

While most readers will be interested in more useful PDEs, such as the Euler or Navier-Stokes equations, I begin with the linear-advection equation because it avoids complications that I believe would obscure an intuitive understanding of the stabilization.  It also provides a simple context to review C-SBP discretitzations.

\subsection{The unstabilized SBP discretization in one-dimension}\label{sec:1d_advec}

%- introduce then discretize the linear period advection equation (using LGL nodes)
%- motivate the need for stablization and introduce LPS
%- explain that, while we have lost ``optimal'' convergence for the number of nodes, the number of DOF is comparable to DG (and cost of residual is better)
%- relate LPS to typical high-order derivative dissipation for element (cite Penner paper)

Consider the one-dimensional, constant-coefficient advection equation on a unit
periodic domain $\Omega = [0,1]$:
\begin{equation}\label{eq:1d_advec}
  \begin{alignedat}{2}
    \frac{\partial \ufnc}{\partial t} + \lambda \frac{\partial \ufnc}{\partial x} &= \ffnc,
    &\qquad &\forall\; x \in [0,1], \\
    \ufnc(0,t) &= \ufnc(1,t), &\qquad &\forall\; t \geq 0, \\
    \ufnc(x,0) &= \ufnc_0(x), &\qquad &\forall\; x \in [0,1],
  \end{alignedat}
\end{equation}
where $\lambda \in \mathbb{R}$ is the advection velocity and $\ffnc \in L^2(\Omega)$
is a source.

Let the domain $\Omega = [0,1]$ be divided into the mesh
\begin{equation*}
  \Tmesh \equiv \left\{ \Omega_\kappa \right\}_{\kappa=1}^{K} =  \left\{ [h(\kappa-1),h\kappa] \right\}_{\kappa=1}^{K}
\end{equation*}
of $K$ elements of uniform size $h = 1/K$.  The SBP discretization that we
consider in this section is a spectral-collocation method using Legendre-Gauss-Lobbato (LGL)
nodes, where the solution is stored at $\nk$ LGL quadrature nodes on each
element $\kappa$.  See Figure \ref{fig:1d_mesh} for an example of the mesh and
the quadrature points.

\begin{figure}[tbp]
  \begin{center}
    \includegraphics[width=\textwidth]{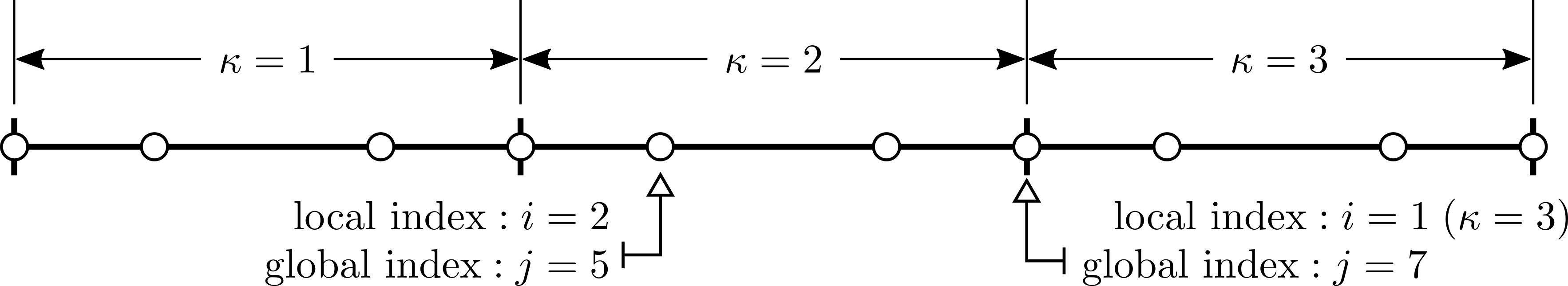}
    \caption[]{Example one-dimensional mesh illustrating node
      indexing.\label{fig:1d_mesh}}
  \end{center}
\end{figure}

\begin{remark}
  A conventional spectral-collocation method based on LGL nodes uses $\nk = p+1$
  collocation points on each element, where $p$ is the degree of the polynomial
  basis.  In this work we use $\nk > p+1$, which is more like a finite-difference
  operator.  We will motivate this decision later in Section~\ref{sec:sacrifice}.
\end{remark}

The global solution is denoted by the vector $\uh \in \mathbb{R}^{n}$, where the
number of degrees of freedom in this example is $n = (\nk-1)K$; for those more familiar with discontinuous SBP discretizations, the solution here is \emph{not} multivalued at the element interfaces.

In order to manipulate the global solution at the element-level, we need to define
restriction and prolongation operators.  To this end, consider a global node
with index $j \in \{1,2,\ldots,n\}$ that coincides with the local index $i \in
\{1,2,\ldots,\nk\}$ on element $\kappa$; again, see Figure~\ref{fig:1d_mesh}.
Then the $(i,j)$th entry of the restriction matrix is unity:
$\left[\Rs_\kappa\right]_{ij} = 1$.  For example, assuming the nodes are
ordered sequentially, both on each element and globally, the restriction operator for
element $\kappa=2$ in Figure~\ref{fig:1d_mesh} is
\begin{equation*}
  \Rs_{2} = \begin{bmatrix}
    0 & 0 & 0 & 1 & 0 & 0 & 0 & 0 & 0 \\
    0 & 0 & 0 & 0 & 1 & 0 & 0 & 0 & 0 \\
    0 & 0 & 0 & 0 & 0 & 1 & 0 & 0 & 0 \\
    0 & 0 & 0 & 0 & 0 & 0 & 1 & 0 & 0 
    \end{bmatrix}.
\end{equation*}
The prolongation operator on element $\kappa$ is simply the transpose of
$\Rs_\kappa$.

Next, I introduce a degree $p$ diagonal-norm SBP operator $\Dxk = \Hk^{-1} \Qxk$ on the $\nk$ LGL nodes in the reference space $\xi \in
[-1,1]$; see~\cite{Gassner2013} for the case $p = \nk-1$ and~\cite{DCDRF2014} for
the more general case.  Briefly, $\Dxk$ is a finite-difference operator that exactly differentiates degree $p$ polynomials at the nodes.  Furthermore, $\Hk$ is a diagonal matrix with positive entries along its diagonal, and the symmetric part of $\Qxk$ satisfies $\Qxk + \Qxk^T = \mydiag(-1,0,0,\ldots,1) \equiv \Exk$.  I will review the multidimensional SBP definition in Section~\ref{sec:constructSBP}.

Using the restriction operators and the matrices $\Hk$ and $\Qxk$, one can define an SBP operator that acts on the
global degrees of freedom:
\begin{gather}
  \Dx = \H^{-1} \Qx \notag \\
  \intertext{where}  
  \Qx \equiv \sum_{\kappa=1}^{K} \Rs_{\kappa}^T \Qxk \Rs_{\kappa},
  \qquad\text{and}\qquad
  \H  \equiv \frac{h}{2} \sum_{\kappa=1}^{K} \Rs_{\kappa}^T \Hk \Rs_{\kappa}.
  \label{eq:Qx_and_H}
\end{gather}
The fact that $\Dx$ defines a degree $p$ SBP operator at the nodes on $\Omega =
[0,1]$ was established in~\cite{multiSBP}.

\begin{remark}
  For the case $p = \nk-1$, $\Qx$ and $\H$ define the stiffness and lumped-mass matrices
  for a spectral collocation scheme.
\end{remark}

\begin{remark}
  For the periodic case under consideration, the symmetric part of $\Qx$ is zero:
  \begin{equation*}
  \Ex = \Qx + \Qx^T = \sum_{\kappa=1}^{K} \Rs_{\kappa}^T (\Qxk + \Qxk^T) \Rs_{\kappa} 
    = \sum_{\kappa=1}^{K} \Rs_{\kappa}^T \Exk \Rs_{\kappa} = \mat{0},
  \end{equation*}
  since the $-1$ and $1$ values in $\Exk$ from adjacent elements cancel at the common node.  This will be an important property that we will also need for entropy conservation in the multidimensional case in Section~\ref{sec:euler}.
\end{remark}

\subsubsection{Strong-form discretization}

Using the global operator $\Dx$, the SBP semi-discretization of
\eqref{eq:1d_advec} is given by
\begin{equation}\label{eq:1d_advec_sbp}
  \frac{d \uh}{dt} + \lambda \Dx \uh = \fh,
\end{equation}
where $\fh \in \mathbb{R}^{n}$ is the evaluation of $\ffnc$ at the nodes.  The
truncation error, energy stability, and conservative nature of
\eqref{eq:1d_advec_sbp} follow from the properties of the SBP operator $\Dx$; see, for example,
the reviews~\cite{Svard2014} and \cite{Fernandez2014}.

\subsubsection{Weak-form discretization}

I will now show that \eqref{eq:1d_advec_sbp} is also a discretization of the weak
formulation of the constant-coefficient linear advection equation.  Let
$H^{1}(\Omega)$ denote the Hilbert space of periodic functions on $\Omega =
[0,1]$ with bounded derivatives.  Then the weak formulation
of~\eqref{eq:1d_advec} is obtained by multiplying the PDE by an arbitrary $\vfnc \in
H^{1}(\Omega)$, integrating over the domain, and applying integration by parts;
that is, find $\ufnc \in H^{1}(\Omega)$ such that
\begin{equation}\label{eq:1d_advec_weak}
  \int_{\Omega} \vfnc \frac{\partial \ufnc}{\partial t} \,dx - \int_{\Omega} \frac{\partial
    \vfnc}{\partial x} \lambda \ufnc \, dx = \int_{\Omega} \vfnc \ffnc \, dx,
  \qquad \forall\; \vfnc \in H^{1}(\Omega).
\end{equation}

To mimic \eqref{eq:1d_advec_weak} in the discrete case, let $\vh \in \mathbb{R}^{n}$ denote the value of a test function $\vfnc \in
H^{1}(\Omega)$ evaluated at the nodes of $\Tmesh$.  Then, left-multiplying
\eqref{eq:1d_advec_sbp} by $\vh^T \H$ we find
\begin{equation}
  \vh^T \H \frac{d \uh}{dt} - (\Dx \vh)^T \H (\lambda \uh) = \vh^T \H \fh, \qquad \forall \vh \in \mathbb{R}^{n},
  \label{eq:1d_advec_weak_sbp}
\end{equation}
where I have used the SBP property $\H \Dx = -\Dx^T \H + \H \Ex$, and the
fact that $\Ex = \mat{0}$ for the periodic problem under consideration.

There is an obvious structural similarity between \eqref{eq:1d_advec_weak} and \eqref{eq:1d_advec_weak_sbp}; however, the similarity between the two weak forms is more than qualitative.   Since $\Dx$ is a diagonal-norm SBP first-derivative operator, it follows that $\H$ and the nodes define a degree $q \geq 2p-1$ quadrature rule, and $\Qx^T$ approximates the weak
derivative~\cite{Hicken2013quad,multiSBP}:
\begin{align*}
  \vh^T \H \uh &= \int_{\Omega} \vfnc \ufnc \,dx + \text{O}(h^{2p}) \\
  (\Dx \vh)^T \H \uh &= \vh^T \Qx^T \uh = \int_{\Omega} \frac{\partial \vfnc}{\partial x} \ufnc \, dx + \text{O}(h^{2p}).
\end{align*}
Thus, beyond mimicking the structure of \eqref{eq:1d_advec_weak}, each term in \eqref{eq:1d_advec_weak_sbp} is actually a high-order approximation of the corresponding term in the continuous equation.

\begin{remark}
While many discretizations may approximate \eqref{eq:1d_advec_weak} to high-order, the strength of SBP discretizations is that they also mimic integration by parts exactly; that is, few schemes also satisfy the SBP property, $\Qx + \Qx^T = \Ex$, which is valuable in proving stability. 
\end{remark}

\subsubsection{Element-level discretization}

The global discretizations \eqref{eq:1d_advec_sbp} and \eqref{eq:1d_advec_weak_sbp} are concise and useful for high-level analyses; however, in practice, a C-SBP discretization is likely to be constructed similar to a finite-element method, \ie, at the element level.  Furthermore, this element-level perspective can be helpful for definitions, as we have already seen for $\Dx$ and $\H$, and for some detailed analyses.

In the case of the constant-coefficient advection equation, the element-level discretization is obtained by inserting the definitions \eqref{eq:Qx_and_H} into the weak form \eqref{eq:1d_advec_weak_sbp}:
\begin{multline*}
\sum_{\kappa=1}^{K} \vk^T \Hk \frac{d \uk}{dt} - \sum_{\kappa=1}^{K} (\Dxk \vk)^T \Hk (\lambda \uk) = \sum_{\kappa=1}^{K} \vk^T \Hk \fk,
\\ \forall \vk \in \mathbb{R}^{\nk}, \kappa = 1,2,\ldots,K,
\end{multline*}
where we have introduced the restricted solution, test function, and source:
\begin{equation*}
\uk \equiv \Rsk \uh, \qquad 
\vk \equiv \Rsk \vh, \qquad \text{and},\qquad
\fk \equiv \Rsk \fh,
\end{equation*}
respectively.

\subsection{Stabilization in one dimension}\label{sec:1d_stab}

The linear advection equation conserves ``energy'' when $\ffnc=0$ in the periodic IBVP~\eqref{eq:1d_advec}.  That is, the time-rate-of-change of $\|\ufnc\|_{\Omega}^{2}$ is zero.  Likewise, the SBP semi-discretization~\eqref{eq:1d_advec_sbp} conserves the discrete ``energy'' $\|\uh\|_{\H}^2 \equiv \uh^T \H \uh$ when $\fh = \bm{0}$.  This is easy to see by replacing $\vh$ with $\uh$ and setting $\fh = \bm{0}$ in the equivalent weak form~\eqref{eq:1d_advec_weak_sbp}:
\begin{gather*}
 \uh^T \H \frac{d \uh}{dt} - (\Dx \uh)^T \H (\lambda \uh) = 0, \\
 \Rightarrow\qquad
 \frac{d}{dt} \uh^T \H \uh = \lambda\left(\uh^T \Qx^T \uh + \uh^T \Qx \uh\right) = \lambda \uh^T \Ex \uh = 0,
\end{gather*}
where, again, $\Ex = \mat{0}$ due to periodicity.  Thus we have $\frac{d}{dt} \| \uh\|_{\H}^2 = 0$, showing that the energy is conserved.

The SBP semi-discretization \eqref{eq:1d_advec_sbp} is non-dissipative.  This is a property shared by all semi-discretizations of the linear-advection equation that use a skew-symmetric spatial operator.  This does not imply that skew-symmetric discretizations, including \eqref{eq:1d_advec_sbp}, produce error-free solutions.  Dispersion errors are still present, as are aliasing errors for non-constant coefficient problems.  These errors, which are often characterized by high-frequency modes, can lead to suboptimal convergence rates, such as those observed in \cite{multiSBP}.

\ignore{dispersion errors are still present.  Indeed, since the highest frequency modes generally have the largest dispersion errors, the discrete solution will eventually become contaminated with even-odd oscillatory modes.  These high-frequency modes can lead to suboptimal convergence rates; see, for example, the CSBP results in \cite{multiSBP}.
}

This problem with skew-symmetric, or nearly skew-symmetric, discretizations of advection-dominated PDEs is well known.  It is addressed by so-called stabilization methods, such as SUPG~\cite{Brooks1982streamline}, in the case of finite-element methods, or artificial dissipation based on undivided-differences~\cite{Jameson1981numerical,Pulliam1986artificial,Mattsson2004stable}, in the case of traditional finite-difference methods.

\begin{remark}  The term stabilization is retained for historical reasons, but it should not be confused with energy or entropy stability.  As shown at the beginning of this section, a skew-symmetric discretization like \eqref{eq:1d_advec_sbp} can be \emph{energy stable} in the sense that its solution has a bounded norm; this same solution can possess spurious, high-frequency modes that need to be ``stabilized.''
\end{remark} 

\subsubsection{Derivative-based stabilization in one dimension}\label{sec:derivbased}

Generically speaking, stabilization methods introduce dissipative terms that target undesired, high-frequency modes.  A common construction for these dissipative terms is based on the inner product between (sufficiently) high-order derivatives of the trial and test functions~\cite{Hesthaven2008nodal,Mattsson2004stable,Ranocha2018stability,Penner2018high}.  For a single element on the reference domain $\Omega_\kappa = [-1,1]$, the continuous derivative-based dissipation operator, and its corresponding SBP discretization, take the form
\begin{equation} 
  \int_{-1}^{1} \frac{\partial^s \vfnc}{\partial \xi^s} \fnc{A}(\xi) \frac{\partial^s \ufnc}{\partial \xi^s} \,d\xi
  \approx \vk^T \underbrace{\left(\Dxk^s\right)^T \Hk \mat{A}_{\kappa} \Dxk^s}_{\displaystyle \equiv \Disk} \uk,
  \label{eq:disk}
\end{equation}
where $\Dxk^{s}$ is either a direct discretization of $\partial^s/\partial \xi^s$ or the product of $s$ first derivative operators $\Dxk$.  I will discuss the choice of $s$ shortly.  The scaling function $\fnc{A}(\xi) > 0$ is used to ensure the stabilization is dimensionally consistent and its magnitude is commensurate with the wave speed.  The matrix $\mat{A}_{\kappa} = \mydiag(\fnc{A}(\xi_1),\fnc{A}(\xi_2),\ldots,\fnc{A}(\xi_{\nk}))$ is a diagonal matrix whose entries are the function $\fnc{A}(\xi)$ evaluated at the quadrature nodes of element $\kappa$.

The dissipation operator $\Disk$ is symmetric positive semi-definite by construction, so $\uk^T \Disk \uk \geq 0$.  Consequently, adding this dissipation to each element of the weak-form C-SBP discretization (on the left-hand side) produces an energy-stable discretization of the constant-coefficient advection equation.  That is, the time rate-of-change of the solution norm is non-positive: $\frac{d}{dt} ( \uh^T \H \uh ) \leq 0$.

The order of the derivatives, $s$, that appears in $\Disk$ is chosen based on accuracy considerations.  Let $\mathbb{P}^{p}[-1,1]$ denote the space of polynomials of degree $p$ on the reference domain $[-1,1]$.  For a given polynomial $\fnc{P} \in \mathbb{P}^{p}[-1,1]$, I use $\bm{p}$ to denote $\fnc{P}$ evaluated at the quadrature nodes: $(\bm{p})_i = \fnc{P}(\xi_i), \forall i = 1,2,\ldots,\nk$.  To ensure optimal $p+1$ convergence rates, I select $s = p+1$ so that $\Dxk^s \bm{p} = 0$ for all polynomials $\fnc{P} \in \mathbb{P}^{p}[-1,1]$.  The choice $s=p+1$ preserves the accuracy of the SBP first-derivative operator when the dissipation is added to $\Dxk$:
\begin{align*}
\left( \Dxk + \Hk^{-1} \Disk \right) \bm{p} &= \Dxk \bm{p} + \Hk^{-1} \left(\Dxk^s\right)^T \Hk \mat{A}_{\kappa} \underbrace{\Dxk^s \bm{p}}_{= \bm{0}} \\
&= \Dxk \bm{p} 
= \begin{bmatrix}
\frac{\partial \fnc{P}}{\partial \xi}(\xi_1), & \frac{\partial \fnc{P}}{\partial \xi}(\xi_2), & \ldots &
\frac{\partial \fnc{P}}{\partial \xi}(\xi_q) \end{bmatrix}^{T},
\end{align*}
for all $\fnc{P} \in \mathbb{P}^{p}$.  Note that stability requires $\Hk^{-1}$ to left multiply $\Disk$ in the first line --- it disappears when the strong form is contracted with $\vk^T \Hk$, and it does not impact accuracy.

\subsubsection{Trading element-local stabilization for optimal approximation}\label{sec:sacrifice}

One consequence of $s = p+1$ is that we need at least $\nk = s+1 = p+2$ quadrature points on an element to construct an operator such that $\Dxk^s \bm{p} = \bm{0}$ for all $\fnc{P} \in \mathbb{P}^{p}[-1,1]$; if $\nk = p+1$ we will obtain the trivial operator $\Dxk^s = \mat{0}$.  The implication is that this form of \emph{element-local dissipation} is not ``optimal'' for C-SBP operators, since more than $p+1$ nodes are required for a degree $p$ operator; however, this is a narrow definition of optimal inherited from approximation theory.  For example, classical finite-difference discretizations on uniform grids are not ``optimal'' according to this definition, yet few would argue that they are inefficient in practice.  Furthermore, if efficiency --- accuracy per unit cost --- is a primary objective for an element-based discretization, then our results suggest that one should consider non-optimal operators.

\begin{remark}
In addition to classical finite-difference methods, finite-element methods that use so-called bubble functions~\cite{Baiocchi1993virtual} also have more degrees of freedom than necessary for a target polynomial degree $p$ and, thus, enable element-local stabilizations; however, to the best of my knowledge, the existing stabilizations are not entropy stable and high-order.
\end{remark}

\subsubsection{Projection-based stabilization in one dimension}

In this section, I describe a projection-based dissipation that is a SBP generalization of local-projection stabilization, or LPS~\cite{Becker2001finite}.  The form of LPS that I consider amounts to penalizing, or damping, polynomial degrees higher than $p$ in the solution; however, it has traditionally been used in incompressible flows to damp high-degree modes in the gradient~\cite{Braack2009finite}.

In order to define LPS, I will use the following continuous projection operator on the reference element: for a given function $\ufnc \in L^2[-1,1]$, find a polynomial $\tilde{\ufnc} \in \poly{p}[-1,1]$ such that
\begin{equation}
  \int_{\xi=-1}^{1} \tilde{\vfnc}(\ufnc - \tilde{\ufnc}) \, d\xi = 0, \qquad \forall\, \tilde{\vfnc} \in \poly{p}[-1,1].
  \label{eq:L2proj}
\end{equation}
This is a simple $L^2$ projection that removes any ``high-frequency'' content in $\ufnc$ that cannot be represented in the polynomial space $\poly{p}[-1,1]$.

There are several ways to obtain an appropriate SBP discretization of the projection \eqref{eq:L2proj}.  The approach I describe here is well-suited for element-based SBP operators defined on LGL nodes with degree $2\nk-3$ cubatures; I will discuss some alternative constructions in Section~\ref{sec:construct_lps}.

First, we represent the projected quantity at the nodes using the orthogonal Legendre polynomials\footnote{Assume the Legendre polynomials are normalized to have unit $L^2$ norm} up to degree $p$, which I denote here by $\{\fnc{L}_j(\xi)\}_{j=0}^{p}$:
\begin{equation*}
\utildek = \mat{L} \bm{y},
\end{equation*}
where $(\utildek)_i = \tilde{\ufnc}(\xi_i)$, and $\mat{L}_{ij} = \fnc{L}_j(\xi_i)$ is a generalized, $\nk \times(p+1)$ Vandermonde matrix.  The vector $\bm{y} \in \mathbb{R}^{p+1}$ holds the to-be-determined basis coefficients.

Next, we can also use the Legendre polynomials for the test functions $\tilde{\vfnc}$, because the $\fnc{L}_j$ form a basis for $\poly{p}[-1,1]$.  Finally, we use the SBP matrix $\Hk$ to perform integration, so that \eqref{eq:L2proj} discretizes as
\begin{equation*}
\mat{L}^T\Hk \left(\uk - \mat{L} \bm{y}\right) = \bm{0} \qquad \Rightarrow \qquad
\bm{y} = \mat{L}^T \Hk \uk,
\end{equation*}
where $\mat{L}^T \Hk \mat{L} = \mat{I} \in \mathbb{R}^{(p+1)\times(p+1)}$ follows from the orthonormality of the Legendre polynomials and the accuracy of the quadrature.  To clarify this last point, for $\nk > p+1$ nodes an LGL quadrature is exact for degree $2\nk - 3 > 2p - 1$ polynomials, and the Legendre polynomials in $\mat{L}$ are at most degree $p$.

We can now isolate the high-order modes in an arbitrary $\uk$ by subtracting the projected nodal values of $\utildek$:
\begin{equation*}
\uk - \utildek = \uk - \mat{L} \mat{L}^T \Hk \uk = \Prjk \uk,
\end{equation*}
where $\Prjk = \mat{I} - \mat{L} \mat{L}^T \Hk$.  If we wish to penalize the high-order modes, \ie, $\Prjk \uk$, in a symmetric manner similar to the derivative-based dissipation \eqref{eq:disk}, then we arrive at the following local-projection stabilization on element $\kappa$:
\begin{equation} 
  \int_{\xi=-1}^{1} (\vfnc - \tilde{\vfnc}) \fnc{A}(\xi) (\ufnc - \tilde{\ufnc}) \,d\xi
  \approx \vk^T \underbrace{\left(\Prjk \right)^T \Hk \mat{A}_{\kappa} \Prjk}_{\ds \equiv \Mlps} \uk, 
  \label{eq:LPS}
\end{equation}

\subsubsection{Equivalence between derivative-based and projection-based stabilization}

The theorem below shows that, \emph{for one-dimensional discretizations}, the dissipation operators $\Mlps$ and $\Disk$ are closely related. 

\begin{theorem}\label{thm:equiv}
Consider the derivative-based and projection-based dissipation operators, $\Disk$ and $\Mlps$, on the reference element with $\nk$
 Legendre-Gauss-Lobatto nodes, where $\nk = p+2 \geq 2$.  Then $\Disk = \alpha \Mlps$, for some constant $\alpha > 0$.
\end{theorem}

\begin{proof}
 $\Disk$ and $\Mlps$ are non-trivial, symmetric matrices that satisfy $\Disk \mat{L} = \Mlps \mat{L} = \mat{0}$, where $\mat{L} \in \mathbb{R}^{\nk \times (p+1)}$ holds the Legendre polynomials of degree $p$ evaluated at the nodes; that is, both operators annihilate polynomials of degree $p$ or less.  Therefore, since $\mat{L}$ has full column rank with one more row than column, it follows that both $\Disk$ and $\Mlps$ are symmetric rank-one operators.  Thus, we can express both operators as $\Disk = \lambda^{\mat{D}} \bm{m} \bm{m}^T$ and $\Mlps = \lambda^{\mat{P}} \bm{m} \bm{m}^T$, where $\bm{m} \in \mathbb{R}^{q}$ is the unique (up to sign) eigenvector that satisfies $\| \bm{m} \| = 1$ and $\mat{m}^T\mat{L} = 0$.   The result follows with $\alpha =  \lambda^{\mat{D}}/\lambda^{\mat{P}}$.
\end{proof}

Theorem~\ref{thm:equiv} shows that $\Disk$ and $\Mlps$ are equivalent up to a multiplicative constant when using $\nk = p+2$ LGL nodes.  To make this equivalence more concrete, Table~\ref{tab:diss1d} lists the rank-one decompositions, as described in the proof of Theorem~\ref{thm:equiv}, for degrees $p=0,1,2,$ and 3.  For simplicity, the table assumes $\mat{A}_{\kappa} = \mat{I}$.

We see from Table~\ref{tab:diss1d} that the equivalence factor $\alpha$ grows rapidly with $p$ and that this growth is due to $\Disk$ and not $\Mlps$.  The derivative-based dissipation operator approximates $\partial^{2s}/\partial \xi^{2s}$, so Fourier modes with frequency $p+1$ are amplified by $(p+1)^{2s} = (p+1)^{2(p+1)}$.  While the rapid growth of $\alpha$ is of no concern for one-dimension or tensor-product CSBP discretizations --- we can always scale $\Disk$ to get $\Mlps$ --- we will see that it can have significant consequences for derivative-based dissipation based on multidimensional SBP operators.

\begin{table}[t]
\begin{center}
\caption[]{Rank-one decompositions of the dissipation operators $\Disk$ and $\Mlps$ for degrees 0 to 3 when $\mat{A} = \mat{I}$.  In the one-dimensional C-SBP framework under consideration, a degree $p$ operator uses LGL elements with $\nk=p+2$ nodes. \label{tab:diss1d}}
\begin{tabular}{lcccc}
\rule{0ex}{4ex}%
\textbf{degree} & $\bm{m}^T$ & $\Disk$ & $\Mlps$ & $\alpha$ \\\hline
\rule{0ex}{5ex}$p=0$ & 
$\displaystyle \frac{1}{\sqrt{2}} \begin{bmatrix} \phm1 & -1 \end{bmatrix}$ & 
$\bm{m}\bm{m}^T$ & $\bm{m}\bm{m}^T$ & 1 \\[3ex]
$p=1$ & 
$\displaystyle \frac{1}{\sqrt{6}} \begin{bmatrix} \phm1 & -2 & \phm1 \end{bmatrix}$ & 
$2\, \bm{m} \bm{m}^T$ & $\frac{2}{3}\, \bm{m}\bm{m}^T$ & 3 \\[3ex]
$p=2$ & 
$\displaystyle \frac{1}{2\sqrt{3}}\begin{bmatrix}
\phm1 & -\sqrt{5} & \phm\sqrt{5} & -1 \end{bmatrix}$ &
$\frac{675}{2}\, \bm{m} \bm{m}^T$ & $\frac{1}{2}\, \bm{m} \bm{m}^T$ & 675  \\[3ex]
$p=3$ &
$\displaystyle \frac{1}{6\sqrt{5}} \begin{bmatrix}
\phm3 & -7 & \phm8 & -7 & \phm3 \end{bmatrix}$ &
$17\,640\, \bm{m} \bm{m}^T$ & $\frac{2}{5}\, \bm{m}\bm{m}^T$ & 44\,100
\end{tabular}
\end{center}
\end{table}

\subsection{Stabilization of SBP discretizations that use multidimensional operators}\label{sec:2d_stab}

I now consider C-SBP discretizations that use multidimensional (non-tensor product) operators and, in particular, the problem with using derivative-based stabilization with such operators.  The problem can be understood by studying the spectral properties of the dissipation operators on a single element in reference space.  Thus, I will begin by generalizing the derivative- and projection-based operators to a generic, multidimensional reference element.

In multiple dimensions, the derivative-based dissipation operator becomes a sum of one-dimensional dissipation operators.  For example, in two dimensions, the derivative-based dissipation operator is 
\begin{multline*} 
  \int_{\Omega_\kappa} \left[ \frac{\partial^s \vfnc}{\partial \xi^s} \fnc{A}_{\xi}(\xi,\eta) \frac{\partial^s \ufnc}{\partial \xi^s} 
  + \frac{\partial^s \vfnc}{\partial \eta^s} \fnc{A}_{\eta}(\xi,\eta) \frac{\partial^s \ufnc}{\partial \eta^s} \right]
  \,d\Omega \\[2ex]
  \approx \vk^T \underbrace{\left[ \left(\Dxi^s\right)^T \Hk \mat{A}_{\xi} \Dxi^s + \left(\Deta^s\right)^T\Hk \mat{A}_{\eta} \Deta^s\right]}_{\ds \equiv \Disk} \uk,
\end{multline*}
where $\Dxi^s \approx \partial^s/\partial \xi^s$ and $\Deta^s \approx \partial^s/\partial \eta^s$.  As in the one-dimensional case, $\fnc{A}_{\xi}(\xi,\eta) > 0$ and $\fnc{A}_{\eta}(\xi,\eta) > 0$ are positive functions, and the diagonal matrices $\mat{A}_\xi$ and $\mat{A}_{\eta}$ hold the values of $\fnc{A}_\xi$ and $\fnc{A}_{\eta}$, respectively, evaluated at the nodes of the element.  Furthermore, as explained in Section~\ref{sec:derivbased}, we take $s=p+1$ when using degree $p$ operators.

The projection-based dissipation operator is essentially unchanged in multiple dimensions.  Indeed, in two dimensions, the only noticeable difference is the domain of integration in the continuous operator:
\begin{equation*} 
  \int_{\Omega_\kappa} (\vfnc - \tilde{\vfnc}) \fnc{A}(\xi,\eta) (\ufnc - \tilde{\ufnc}) \,d\xi d\eta
  \approx \vk^T \underbrace{\left(\Prjk \right)^T \Hk \mat{A}_{\kappa} \Prjk}_{\ds \equiv \Mlps} \uk,
\end{equation*}
where $\tilde{\ufnc} \in \poly{p}(\Omega_\kappa)$ is the $L^2$ projection of $\ufnc \in L^2(\Omega_\kappa)$ onto the space of (total) degree $p$ polynomials and is the solution to
\begin{equation*}
  \int_{\Omega_\kappa} \tilde{\vfnc}(\ufnc - \tilde{\ufnc}) \, d\Omega = 0, \qquad \forall\, \tilde{\vfnc} \in \poly{p}(\Omega_\kappa).
\end{equation*}
The discrete projection operator, $\Prjk$, is also defined as it was in one dimension: $\Prjk = \mat{I} - \mat{L} \mat{L}^T \Hk$.  Here, the entries in $\mat{L}$ correspond to the nodal values of appropriate orthogonal polynomials for the reference element under consideration.  I provide additional details on the construction of $\Prjk$ and $\Mlps$ in Section~\ref{sec:constructSBP}.

Unlike the one-dimensional case, $\Disk$ and $\Mlps$ are not scalar multiples of one another, in general.  To illustrate this, Figure \ref{fig:diss_spect} plots the eigenvalues of $\Hk^{-1}\Disk$ and $\Hk^{-1}\Mlps$ for SBP operators defined on the standard reference triangle.  The specific operators correspond to the degree $p=1$ and $p=4$ operators defined later in Section~\ref{sec:constructSBP}, but the qualitative trends in Figure~\ref{fig:diss_spect} hold more generally.  The dissipation operators are unscaled for simplicity, that is $\mat{A}_{\xi} = \mat{A}_{\eta} = \mat{A}_{\kappa} = \mat{I}$.

\begin{remark}
The eigenvalues of $\Hk^{-1}\Disk$ and $\Hk^{-1}\Mlps$ are strictly real and non-negative, since $\Disk$ and $\Mlps$ are symmetric positive semi-definite matrices and $\Hk$ is positive definite.
\end{remark}

The spectra in Figures~\ref{fig:eigval_operators_deg1} and \ref{fig:eigval_operators_deg4} are clearly partitioned into zero and non-zero eigenvalues.  The zero eigenvalues correspond to non-dissipated modes: polynomials of total degree $p$ or less.  The non-zero eigenvalues correspond to dissipated modes.  Note that I have ``normalized'' the eigenvalues of the derivative-based operators such that their smallest non-zero eigenvalue is equal to one; this is done to highlight the magnitude between the smallest and largest non-zero eigenvalues.

The projection-based dissipation operators damp high-frequency modes uniformly, since their non-zero eigenvalues are all equal to one.  In contrast, the derivative-based dissipation operators are characterized by a range of non-zero eigenvalues and damp the high-frequency modes non-uniformly.

\begin{figure}[t]
  \subfigure[$p=1$ \label{fig:eigval_operators_deg1}]{%
    \includegraphics[width=0.49\textwidth]{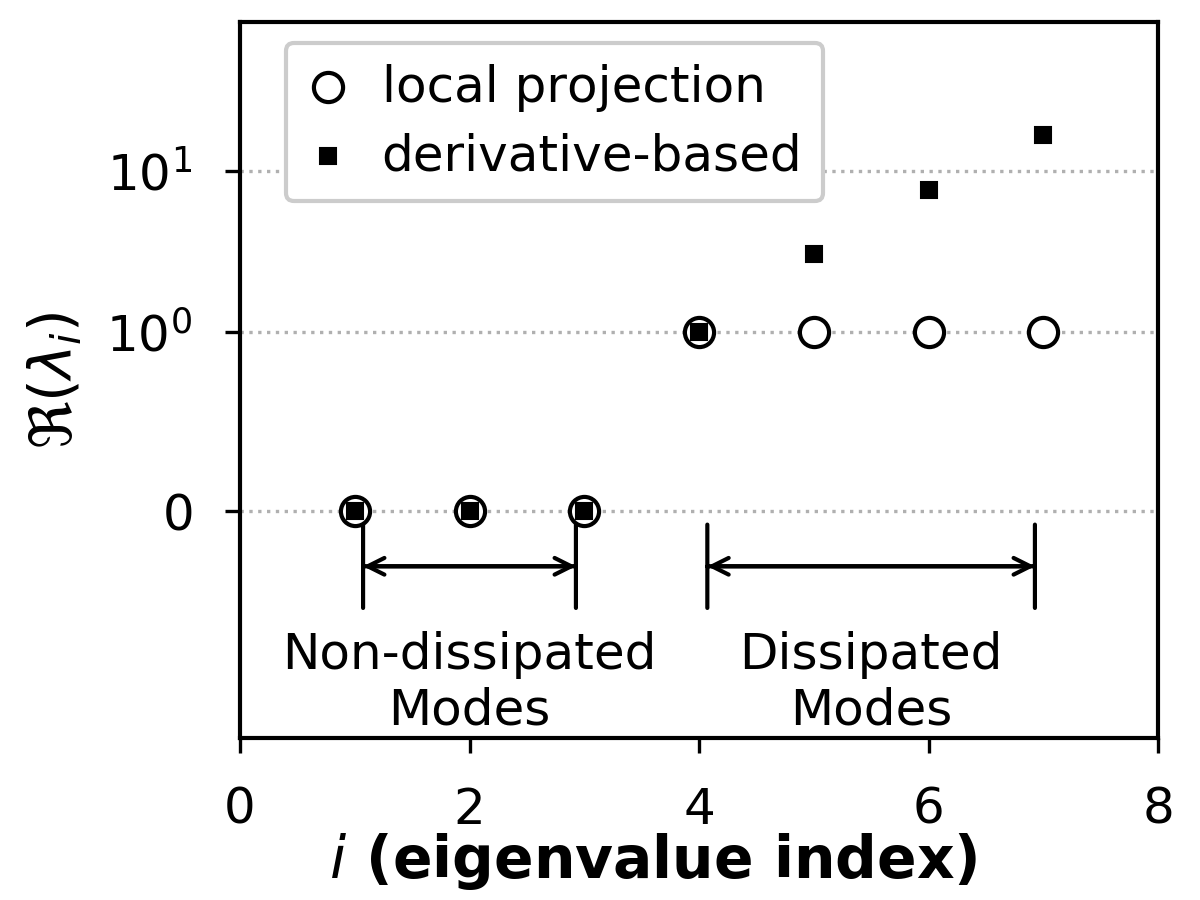}}
  \subfigure[$p=4$ \label{fig:eigval_operators_deg4}]{%
        \includegraphics[width=0.49\textwidth]{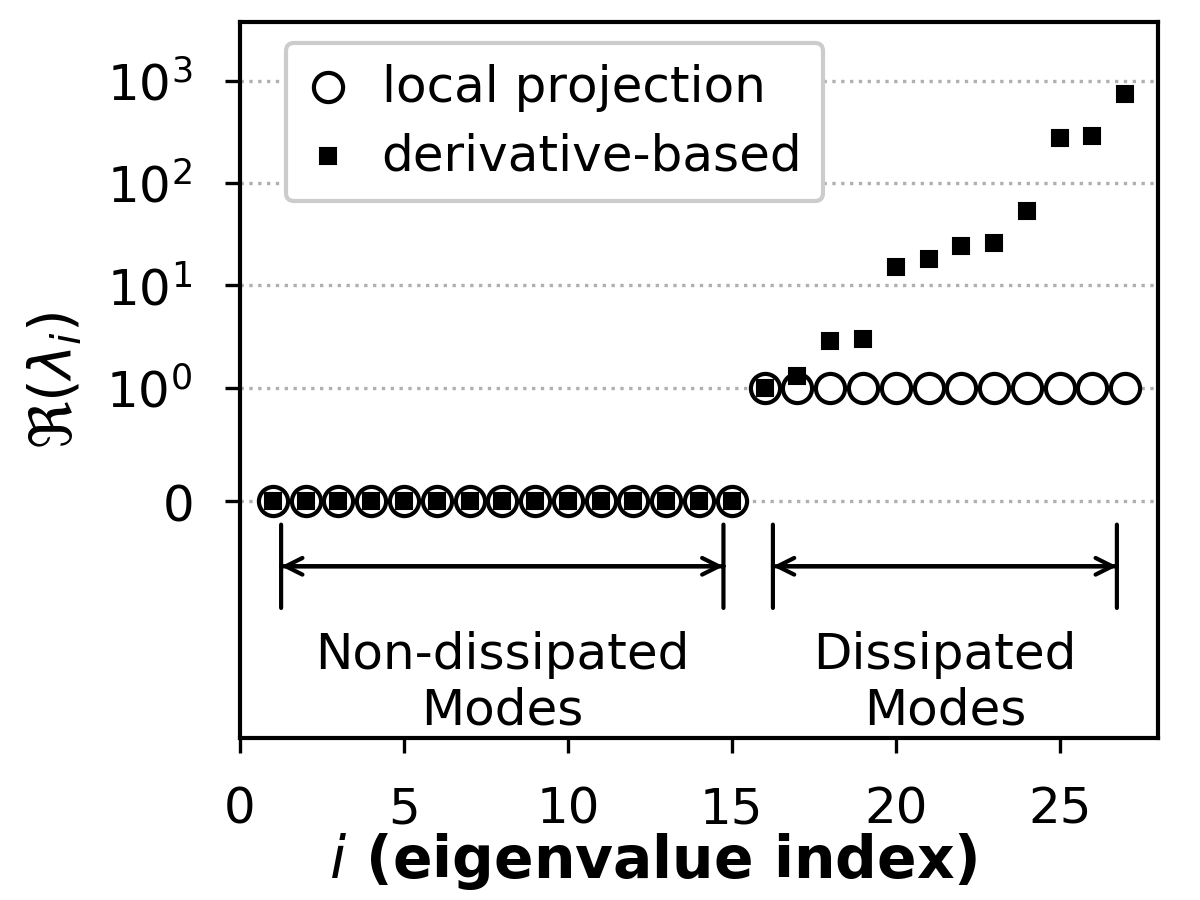}}
  \caption{Eigenvalues of the derivative-based and projection-based stabilization operators for the degrees $p=1$ and $p=4$ diagonal $\mat{E}$ SBP operators defined in Section~\ref{sec:constructSBP}.  Note that the $y$-axis scale is linear between $0$ and $1$ and logarithmic for values above $1$. \label{fig:diss_spect}}
\end{figure}

This distinction between the operators becomes especially significant as $p$ increases.  Table~\ref{tab:eigratio} lists the damping ratio between the largest, $\lambda_{\max}$, and the smallest, $\lambda_{\min}$, non-zero eigenvalues for degrees $p=0$ to $p=4$.  By the time we reach $p=4$ the eigenvalue ratio for $\Disk$ is almost three orders of magnitude.

The problem with a large damping ratio is that it requires a tradeoff between adequately damping \emph{all} the high-frequency modes and maintaining a relatively small spectral radius or condition number.  Consequently, if we use derivative-based dissipation with multi-dimensional SBP discretizations, we must choose between optimal $p+1$ convergence rates from adequate damping, on the one hand, and fast simulations from large time steps, on the other.

This conundrum is completely avoided with projection-based dissipation.  With $\Mlps$ all the high-frequency modes are equally damped, so a C-SBP discretization can achieve an optimal convergence rate while maintaining an attractive spectral radius.  We will illustrate this with the numerical results in Section~\ref{sec:results}.

\begin{table}[t]
\begin{center}
\caption[]{Ratio between the largest and smallest non-zero eigenvalues in the derivative-based and projection-based dissipation operators for the triangular-element SBP operators described in Section~\ref{sec:constructSBP}. \label{tab:eigratio}}
\begin{tabular}{llrrrrr}
\rule{0ex}{4ex}%
 & & \multicolumn{5}{c}{\textbf{SBP operator degree}} \\[1ex]\cline{3-7}
\rule{0ex}{2.5ex} & & $p=0$ & $p=1$ & $p=2$ & $p=3$ & $p=4$ \\[0.5ex]\hline
\rule{0ex}{2.5ex} $(\lambda_{\max} /\lambda_{\min})$ &
derivative & 3.00 & 16.65 & 13.28 & 91.43 & 743.93 \\
& projection & 1.00 & 1.00 & 1.00 &  1.00 & 1.00 \\\hline
\end{tabular}
\end{center}
\end{table}

\section{Stabilization of continuous SBP discretizations: Euler equations}\label{sec:euler}

In the preceding section, I described two complementary ideas regarding the stabilization of C-SBP discretizations: i) sacrificing optimal approximation order by using more nodes than necessary on a given element; and ii) constructing well-conditioned (element-local) dissipation operators based on local-projection stabilization.

Hopefully I was able to convince you that this approach to stabilization has merit, at least in the context of the constant-coefficient advection equation.  In this section, I will generalize projection-based dissipation to the Euler equations and use it to construct high-order, entropy-stable C-SBP discretizations. 

\subsection{Review of entropy conservative SBP discretizations}

As discussed in the introduction, there has been considerable work developing entropy-conservative/stable SBP discretizations in the context of discontinuous solution spaces.  I will draw heavily on this literature for the baseline entropy-conservative C-SBP discretization in order to avoid a lengthy digression into the existing theory.  Furthermore, since C-SBP schemes can use any weakly-imposed boundary conditions developed for discontinuous SBP discretizations, I will ignore boundary conditions and focus on periodic problems.

\subsubsection{The Euler PDE and entropy conservation} 

The two dimensional Euler equations on a square periodic domain, $\Omega = [0,1]^2$, are given by
\begin{equation}\label{eq:2d_euler}
  \begin{alignedat}{2}
    \frac{\partial \bfnc{U}}{\partial t} + \frac{\partial \bfnc{F}_x}{\partial x} + \frac{\partial \bfnc{F}_y}{\partial y} &= \bm{0},
    &\qquad &\forall (x,y) \in \Omega, \\
    \bfnc{U}(0,y,t) &= \bfnc{U}(1,y,t), &\qquad &\forall t \geq 0, y \in [0,1], \\
    \bfnc{U}(x,0,t) &= \bfnc{U}(x,1,t), &\qquad &\forall t \geq 0, x \in [0,1], \\
    \bfnc{U}(x,y,0) &= \bfnc{U}_{0}(x,y), &\qquad &\forall (x,y) \in \Omega,
  \end{alignedat}
\end{equation}
where $\bfnc{U} = [ \rho, \rho u, \rho v, e ]^{T}$ denotes the conservative
variables, and the Euler fluxes are
\begin{gather*}
\bfnc{F}_x(\bfnc{U}) = \begin{bmatrix} \rho u \\ \rho u^{2} + p \\ \rho u v \\ (e + p)u \end{bmatrix},
\qquad\text{and}\qquad
  \bfnc{F}_{y} = \begin{bmatrix} \rho v \\ \rho v u \\ \rho v^2 + p \\ (e + p)v \end{bmatrix}.
\end{gather*}
The pressure is defined by $p = (\gamma -1)[e - \frac{\rho}{2}(u^2 + v^2)]$, where the heat capacity ratio is assumed to be $\gamma = 1.4$.

Unlike the constant-coefficient linear advection equation, we cannot show that the $L^2$ norm of the conservative variables is bounded by contracting \eqref{eq:2d_euler} with $\bfnc{U}$; however, Dafermos~\cite{Dafermos2010hyperbolic} has shown that we can bound the solution indirectly, provided the entropy is bounded and the density and pressure remain positive.

A bound on the entropy follows from the Euler equations, since the PDE~\eqref{eq:2d_euler} implies the entropy is conserved for smooth flows (see~\cite{Tadmor2003entropy} and the references therein):
\begin{equation}\label{eq:ent}
\frac{d}{dt} \int_{\Omega} \fnc{S} \, d\Omega - \int_{\partial \Omega} \rho \fnc{S} U_n \, d\Gamma = 0,
\end{equation}
where $\fnc{S} = -\rho s /(\gamma - 1)$ is the (mathematical) entropy, $s = \ln(p/\rho^{\gamma})$ is the thermodynamic entropy, and $U_n = u n_x + v n_y$ is the normal component of the velocity.  Thus, for periodic domains like $\Omega$, the surface integral vanishes and the integral of entropy is constant with respect to time.  If the flow has shocks, then we want a weak solution such that the left-hand side of \eqref{eq:ent} is less than or equal to zero.  Thus, for both continuous and discontinuous solutions, the integral of entropy remains bounded, thereby bounding\footnote{Again, the bound requires the assumption that the density and pressure remain positive.} $\bfnc{U}$ .

\subsubsection{Entropy-conservative SBP discretization}

Remarkably, SBP operators can be used to construct high-order spatial discretizations of \eqref{eq:2d_euler} that also mimic~\eqref{eq:ent}.  Indeed, an entropy-conservative SBP discretization of \eqref{eq:2d_euler} is given by~\cite{Fisher2012thesis,Fisher2013discretely,Fisher2013high,Carpenter2014entropy,Parsani2016entropy,Chen2017entropy,Crean2018entropy}
\begin{equation}\label{eq:SBPstrong}
  \frac{d \uh}{d t}
  + \left[\barDx \circ \Fx(\uh) \right] \bm{1}
  + \left[\barDy \circ \Fy(\uh) \right] \bm{1}
  = \bm{0}.
\end{equation}
I have introduced some new notation in \eqref{eq:SBPstrong} the needs to be defined.  First, the difference operators are $\barDx \equiv \Dx \otimes \Id_4$ and $\barDy \equiv \Dy \otimes \Id_4$, where $\Dx$ and $\Dy$ are (global) SBP operators, $\Id_4$ is the $4\times 4$ identity matrix, and $\otimes$ denotes the Kronecker product; $\barDx$ and $\barDy$ simply allow us to reuse the scalar difference operators for systems of PDEs.

I also introduced the quantities $\Fx(\uh)$ and $\Fy(\uh)$ in \eqref{eq:SBPstrong}.  These are solution-dependent matrices that hold two-point, entropy-conservative flux functions; the flux functions themselves are discussed later.  The matrix $\Fx(\uh)$ is defined by
\begin{equation*}
  \Fx(\uh) \equiv 2
  \begin{bmatrix}
    \mydiag{\left[\flux{x}(\unode{1},\unode{1})\right]} &
    \hdots &
    \mydiag{\left[\flux{x}(\unode{1},\unode{n})\right]} \\
    \vdots & \ddots & \vdots \\
    \mydiag{\left[\flux{x}(\unode{n},\unode{1})\right]} &
    \hdots &
    \mydiag{\left[\flux{x}(\unode{n},\unode{n})\right]}
  \end{bmatrix},
\end{equation*}
where $\flux{x}(\unode{i},\unode{j})$ is the entropy-conservative flux in the $x$ direction based on the discrete solution at the (global) nodes $i$ and $j$, and the operator $\mydiag(\flux{x})$ indicates the diagonal matrix whose diagonal entries are given by the vector $\bm{\flux{x}}$.  The matrix $\Fy$ is defined similarly.  Finally, the operator $\circ$ denotes the Hadamard (entrywise) matrix product, defined by $(\mat{A} \circ \mat{B})_{ij} = \mat{A}_{ij} \mat{B}_{ij}$, and $\bm{1} \in \mathbb{R}^{4n}$ is a vector of ones.

\begin{remark}
We have defined $\Fx(\uh)$ as a dense block matrix where the flux $\flux{x}$ is evaluated between every pair of nodes in the mesh; however, in practice, if nodes $i$ and $j$ are not in the stencil of $\Dx$, that is, if $(\Dx)_{ij} = 0$, then $\flux{x}(\unode{i},\unode{j})$ is not evaluated and the corresponding block in $\Fx$ is zero.
\end{remark}

To obtain semi-discrete conservation of entropy and, thus, nonlinear stability, \eqref{eq:SBPstrong} must use an entropy-conservative flux function~\cite{Tadmor1987entropy}.  Although Tadmor's original entropy-conservative flux function was too expensive to use in practice, several inexpensive entropy-conservative fluxes have been proposed in the last decade~\cite{Ismail2009affordable,Chandrashekar2015kinetic}.  The availability of such fluxes has been critical to enabling high-order entropy-stable SBP discretizations.

To formally show that the discretization \eqref{eq:SBPstrong} mimics \eqref{eq:ent}, I need to introduce the entropy variables 
\begin{equation*}
  \bfnc{W}(\bfnc{U}) \equiv \frac{\partial \fnc{S}}{\partial \bfnc{U}}
  = \begin{bmatrix} \ds
    \frac{\gamma-s}{\gamma-1} - \frac{1}{2}\frac{\rho}{p}(u^2 + v^2), &
    \ds \frac{\rho u}{p}, & \ds \frac{\rho v}{p}, & \ds -\frac{\rho}{p} \end{bmatrix}^T
\end{equation*}
Furthermore, let the entries in $\wh = \wh(\uh)$ hold the entropy variables evaluated at the nodes.  

\begin{theorem}\label{thm:ent_cons}
Consider the initial-boundary-value problem \eqref{eq:2d_euler} and its discretization \eqref{eq:SBPstrong}.  Assume that the element-level SBP operators are such that the matrices $\mat{E}_{x,\kappa}$ and $\mat{E}_{y,\kappa}$ cancel at element interfaces and, consequently,
\begin{align*}
\Qx &= \sum_{\kappa=1}^{K} \Rsk^T \mat{Q}_{x,\kappa} \Rsk = \sum_{\kappa=1}^{K} \Rsk^T \mat{S}_{x,\kappa} \Rsk = \Sx \\
 \text{and}\qquad
 \Qy &= \sum_{\kappa=1}^{K} \Rsk^T \mat{Q}_{y,\kappa} \Rsk = \sum_{\kappa=1}^{K} \Rsk^T \mat{S}_{y,\kappa} \Rsk = \Sx.
\end{align*}
In addition, assume that each node at the element level is defined by only one global degree of freedom; consequently, each row in the restriction operators $\Rsk$ consists of a single, transposed standard basis vector, \eg, $\left[ \Rsk\right]_{i,:}  = \bm{e}_{j}^T$.  Then the discretization \eqref{eq:SBPstrong} conserves total entropy in the following sense:
\begin{equation*}
\frac{d}{dt} \left(\bm{1}^{T} \H \bm{s}_h \right) = 0.
\end{equation*}
\end{theorem}

\begin{proof}
Left multiply \eqref{eq:SBPstrong} by $\wh^T \barH$:
\begin{gather}
\wh^T \barH \frac{d \uh}{d t} 
  + \wh^T \barH \left[\barDx \circ \Fx(\uh) \right] \bm{1}
  + \wh^T \barH \left[\barDy \circ \Fy(\uh) \right] \bm{1} = 0, \notag \\
  \Rightarrow\qquad
  \frac{d}{dt} \left(\bm{1}^{T} \barH \bm{s}_h \right) = 
  -\wh^T \left[\barSx \circ \Fx(\uh) \right] \bm{1} 
  -\wh^T \left[\barSy \circ \Fy(\uh) \right] \bm{1}. \label{eq:rhs}
\end{gather}
To simplify the temporal term above, I used the differential relation $d\fnc{S} = \bfnc{W}^T d \bfnc{U}$ and the fact that $\barH$ is a diagonal matrix.  I simplified the spatial difference terms by using $\H \Dx = \Qx$, $\H \Dy = \Qy$, and the stated assumption that $\Qx = \Sx$ and $\Qy = \Sy$.

Next, I decompose the global matrices $\barSx$ and $\barSy$ into their constituent, element-based operators in order to express the Hadamard products at the element level.  For instance, in the case of the $x$-component product,
\begin{align*}
\barSx \circ \Fx(\uh)
&= \left[ \sum_{\kappa=1}^{K} \barRsk^T \barSxk  \barRsk \right] \circ \Fx(\uh)  \\
&= \sum_{\kappa=1}^{K} \barRsk^T \left[ \barSxk \circ \Fx(\uk) \right] \barRsk,
\end{align*}
where I used the assumed structure of the restriction operators $\Rsk$.  Using this expression, and an analogous one for $\barSy \circ \Fx(\uh)$, the right-hand-side of \eqref{eq:rhs} becomes
\begin{equation*}
\frac{d}{dt} \left(\bm{1}^{T} \barH \bm{s}_h \right) = 
- \sum_{\kappa=1}^{K} \wk^T \left[ \barSxk \circ \Fx(\uk) + \barSyk \circ \Fy(\uk) \right] \bm{1}_{\kappa},
\end{equation*}
where $\wk = \barRsk \wh$ and $\bm{1}_{\kappa} = \barRsk \bm{1}$.

Lemma 2 from \cite{Crean2018entropy} states that
\begin{align*}
\wk^T \left[ \barSxk \circ \Fx(\uk) \right] \bm{1}_{\kappa} &= -\bm{1}_{\kappa}^T \mat{E}_{x,\kappa} \bm{\psi}_{x,\kappa} \\
\text{and}\qquad
\wk^T \left[ \barSyk \circ \Fy(\uk) \right] \bm{1}_{\kappa} &= -\bm{1}_{\kappa}^T \mat{E}_{y,\kappa} \bm{\psi}_{y,\kappa},
\end{align*}
where $\bm{\psi}_{x,\kappa}$ and $\bm{\psi}_{y,\kappa}$ are the potential fluxes, $\psi_x = \rho u$ and $\psi_y = \rho v$, evaluated at the nodes of element $\kappa$.  Using this lemma and the stated assumption that the symmetric matrices  $\mat{E}_{x,\kappa}$ and $\mat{E}_{y,\kappa}$ cancel at element interfaces, I arrive at
\begin{equation*}
\frac{d}{dt} \left(\bm{1}^{T} \barH \bm{s}_h \right) = 
\sum_{\kappa=1}^{K} \bm{1}_{\kappa}^T \left( \mat{E}_{x,\kappa} \bm{\psi}_{x,\kappa}  + \mat{E}_{y,\kappa} \bm{\psi}_{y,\kappa} \right)
= 0,
\end{equation*}
which is the desired result.
\end{proof}

\begin{remark}
The product $\bm{1}^T \H \bm{s}_{h}$, which Theorem~\ref{thm:ent_cons} tells us is conserved, is a high-order approximation to $\int_{\Omega} \fnc{S} \, d\Omega$.
\end{remark}
 
\subsection{Entropy-stable continuous SBP discretization}
 
Exact semi-discrete conservation of entropy is an attractive property for a discretization to possess for smooth, periodic flows.  Unfortunately, this conservation property reflects the skew symmetry of the underlying discretization, so we should expect dispersion and aliasing errors to pollute the discrete solution of \eqref{eq:SBPstrong} and lead to sub-optimal convergence rates.  Therefore, in this section, we present an entropy-stable discretization --- that is, a semi-discretization that satisfies a non-increasing integral of entropy --- by generalizing the projection-based dissipation.

The idea is simple: apply the projection-based dissipation to the entropy variables.  Thus, the discretization of the Euler equations becomes
\begin{multline}\label{eq:SBPstrong_stab}
  \frac{d \uh}{d t}
  + \left[\barDx \circ \Fx(\uh) \right] \bm{1}
  + \left[\barDy \circ \Fy(\uh) \right] \bm{1} \\
  = - \barH^{-1} \left[\sum_{\kappa=1}^{K} \barRsk^T \left(\barPrjk \right)^T \barHk \mat{A}_{\kappa} \barPrjk \barRsk \right] \wh,
\end{multline}
As with $\barDx$ and $\barDy$, the over-bars appearing on the operators on the right-hand-side of \eqref{eq:SBPstrong_stab} denote Kronecker products between scalar operators and the $4\times 4$ identity; for example, $\barH = \H \otimes \Id_4$.

\begin{theorem}\label{thm:ent_stab}
The discretization \eqref{eq:SBPstrong_stab} is entropy stable if the product $\barHk \mat{A}_{\kappa}$ is symmetric positive semi-definite.
\end{theorem}

\begin{proof}
Assuming that the $\mat{A}_{\kappa}$ are symmetric positive definite and commute with $\barHk$, entropy stability follows easily from the entropy-conservative property of \eqref{eq:SBPstrong}; left multiplying \eqref{eq:SBPstrong_stab} by $\wh^T \barH$ we have
\begin{align*}
\wh^T \barH \frac{d \uh}{d t}
  &+ \wh^T \barH \left[\barDx \circ \Fx(\uh) \right] \bm{1}
  + \wh^T \barH \left[\barDy \circ \Fy(\uh) \right] \bm{1} \\
  &= - \wh^T \left[\sum_{\kappa=1}^{K} \barRsk^T \left(\barPrjk \right)^T \barHk \mat{A}_{\kappa} \barPrjk \barRsk \right] \wh  \\
  \Rightarrow\qquad
  \bm{1}^T \H \frac{d \bm{s}_{h}}{d t} &= - \sum_{\kappa=1}^{K} \underbrace{\wk^T \left(\barPrjk \right)^T \barHk \mat{A}_{\kappa} \barPrjk \wk}_{\text{non-negative}} \\
  \Rightarrow\qquad
  \bm{1}^T \H \frac{d \bm{s}_{h}}{d t} &\leq 0,
\end{align*}
where $\wk = \barRsk \wh$.  This result shows that $\bm{1}^T \H \bm{s}_{h}$ is non-increasing, so \eqref{eq:SBPstrong_stab} is entropy stable.
\end{proof}

Many reasonable choices of $\mat{A}_{\kappa}$ are possible that satisfy the assumptions of Theorem~\ref{thm:ent_stab}.  In this work I adopt a form for $\mat{A}_{\kappa}$ similar to the matrix used to scale the penalty terms in \cite{Crean2018entropy}; this choice of $\mat{A}_\kappa$ is dimensionally consistent.  Specifically, I assume $\mat{A}_{\kappa}$ is a block diagonal matrix in which the $4\times 4$ block corresponding to node $i$ is given by
\begin{equation*}
  \left[\mat{A}_{\kappa}\right]_{i} = \left[\frac{1}{2}(\sigma_\xi + \sigma_\eta) \frac{\partial \bfnc{U}}{\partial \bfnc{W}} \right]_{i},
\end{equation*}
where $\partial \bfnc{U}/\partial \bfnc{W}$ is the inverse of the Hessian $\partial^2 \fnc{S} /\partial \bfnc{U}^2$ and, therefore, symmetric positive definite~\cite{Tadmor2003entropy}.  The scalars $\sigma_\xi$ and $\sigma_\eta$ are the spectral radii of the flux Jacobians in reference space:
\begin{align*}
\sigma_\xi &= \fnc{J} \left(| \nabla_x \xi \cdot (u,v) | + a \| \nabla_x \xi \| \right), \\
\text{and}\qquad
\sigma_\eta &= \fnc{J} \left(| \nabla_x \eta \cdot (u,v) | + a \| \nabla_x \eta \| \right),
\end{align*}
where $\fnc{J}$ is the mapping Jacobian and $a = \sqrt{\gamma p/\rho}$ is the speed of sound.

\section{Construction of SBP and LPS operators on triangles}\label{sec:constructSBP}

Multidimensional SBP operators are not unique, so this section provides the information necessary to construct the particular operators used in this work.  The operators are designed for triangular elements, and they pair well with element-local LPS-based stabilization, because both require more nodes than necessary for a total degree polynomial basis.

For convenience, I include the relevant SBP operator definition below~\cite{multiSBP}.

\begin{definition}[Diagonal norm SBP operator]\label{def:sbp}
Consider a bounded, connected domain $\Omega_{\kappa}$ and node set $\nodes=\left\{(\xi_{i},\eta_{i})\right\}_{i=1}^{\nk}$ with $\nk$ nodes.  The matrix $\Dxk$ is a degree $p$, diagonal-norm SBP approximation to $\partial/\partial \xi$ on $\nodes$ if the following three conditions are met.
\begin{enumerate}
\item $\Dxk \bm{p}$ is equal to $\partial \fnc{P}/\partial \xi$ at the nodes $\nodes$, for all polynomials $\fnc{P} \in
    \poly{p}(\Omega_{\kappa})$, where $\poly{p}(\Omega_{\kappa})$ denotes the space of polynomials of total degree $p$ on $\Omega_{\kappa}$. \label{sbp:accuracy}
  \item $\Dxk = \Hk^{-1}\Qx$, where $\Hk$ is a positive-definite and diagonal matrix. \label{sbp:H}
  \item $\Qxk = \Sxk + \frac{1}{2}\Exk$, where $\Sxk^{T}=-\Sxk$, $\Exk^{T}=\Exk$, and $\Exk$ satisfies \label{sbp:Ex}
    \begin{equation*}
      \bm{p}^{T}\Exk\bm{q} = \int_{\partial\Omega_{\kappa}} \fnc{P} \, \fnc{Q}\; n_{\xi}\,
      d \Gamma,
    \end{equation*}
    for all polynomials $\fnc{P}, \fnc{Q} \in \poly{r}(\Omega_{\kappa})$, where
    $r \ge p$. In the above integral, $n_{\xi}$ is the $\xi$ component of
    $\bm{n}=\left[n_{\xi},n_{\eta}\right]^{T}$, the outward pointing unit normal on
    $\partial \Omega_{\kappa}$.
\end{enumerate}
An analogous definition holds for $\mat{D}_{\eta,\kappa}$.  
\end{definition}

\subsection{Construction of the SBP operator $\Dxk$}

\subsubsection{Determination of the node locations and norm matrix}

A remarkable property of a degree $p$ diagonal-norm SBP operator is that the matrix $\Hk$ and the node locations $\nodes$ define a cubature rule that is at least $2p-1$ exact; this was shown for classical one-dimensional SBP operators in \cite{Hicken2013quad} and for the multidimensional case in \cite{multiSBP}.  Therefore, one approach to the construction of SBP operators on triangles is to first select or build a suitable cubature rule --- $\Hk$ and the node locations --- and then find $\Qxk$.  This was the approach taken in references \cite{multiSBP} and \cite{Fernandez2017simultaneous}, and it is the one we will follow here.

\ignore{
\begin{remark}
  See \cite{X} for a coupled approach to finding a multidimensional SBP operator, in which $\Hk$, $\Qxk$, and the nodes are determined simultaneously using an optimization approach.
\end{remark}
}

In order to select the cubature rules, I first set out some requirements for a target degree $p$ SBP operator.
\begin{itemize}
\item The cubature weights must be positive; this is a requirement of the SBP operator definition and is necessary for stability.
\item The cubature rule should be exact for total degree $2p$ polynomials.  This is one degree higher than necessary for SBP operators~\cite{multiSBP}, but our numerical experience with discontinuous SBP discretizations~\cite{Fernandez2017simultaneous} suggests that $2p$ exactness is necessary to achieve order $h^{p+1}$ convergence rates on more challenging problems.  This requirement is also consistent with discontinuous-Galerkin theory~\cite{Cockburn1990runge}, and it simplifies the construction of the LPS operator, as we will see below.
\item There should be a node at each vertex.  This requirement helps reduce the global number of degrees of freedom for C-SBP discretizations, since a vertex degree-of-freedom is shared by many elements.
\item There should be $p+2$ nodes along each face located at the LGL quadrature points.  This requirement ensures that $\Exk$ is a diagonal matrix, which, as discussed in Section~\ref{sec:euler}, facilitates entropy conservation.  This requirement could be relaxed to requiring only that the contributions to $\Exk$ on a given face involve only nodes lying on that face; this relaxed condition is satisfied by the operators used in \cite{multiSBP}, but those operators do not meet the second requirement above.
\end{itemize}

Cubature rules satisfying the above requirements are either available in the literature~\cite{Cools1999monomial,Liu1998exact} or can be constructed directly; for details on their construction see~\cite{multiSBP} and the references therein.  Figure~\ref{fig:nodes} shows the nodal locations for the $p=0$ through to $p=4$ operators. The corresponding norm matrices are found by simply inserting the cubature weights along the diagonal of $\Hk$ in an order consistent with the nodes.

I now turn to the construction of $\Qxk$ or, more precisely, its symmetric and skew-symmetric parts, $\frac{1}{2}\Exk$ and $\Sxk$, respectively.  

\begin{figure}[tbp]
  \renewcommand{\thesubfigure}{}
    \subfigure[$p=0$ \label{fig:nodes_p0}]{%
    \includegraphics[width=0.19\textwidth]{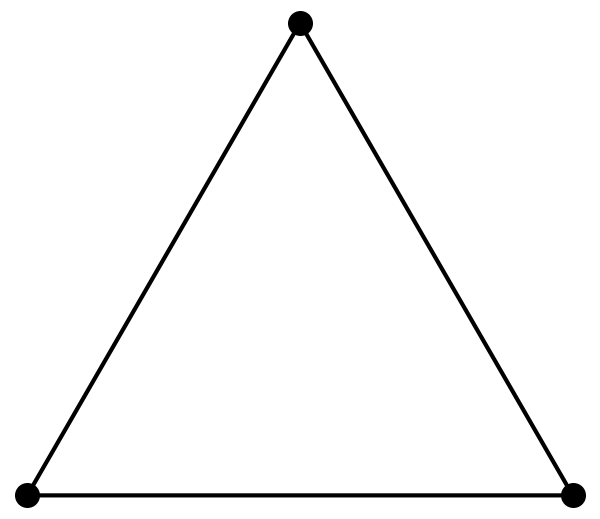}}
  \subfigure[$p=1$ \label{fig:nodes_p1}]{%
    \includegraphics[width=0.19\textwidth]{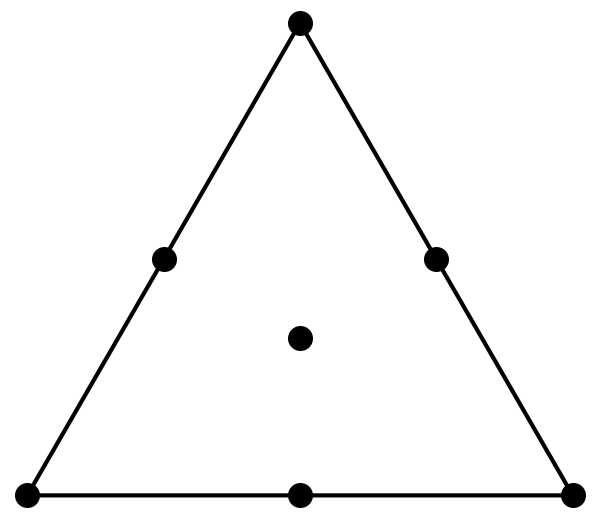}}
  \subfigure[$p=2$ \label{fig:nodes_p2}]{%
        \includegraphics[width=0.19\textwidth]{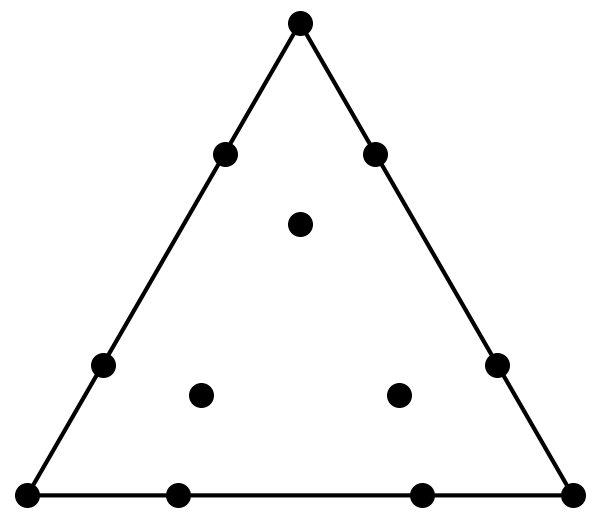}}
  \subfigure[$p=3$ \label{fig:nodes_p3}]{%
        \includegraphics[width=0.19\textwidth]{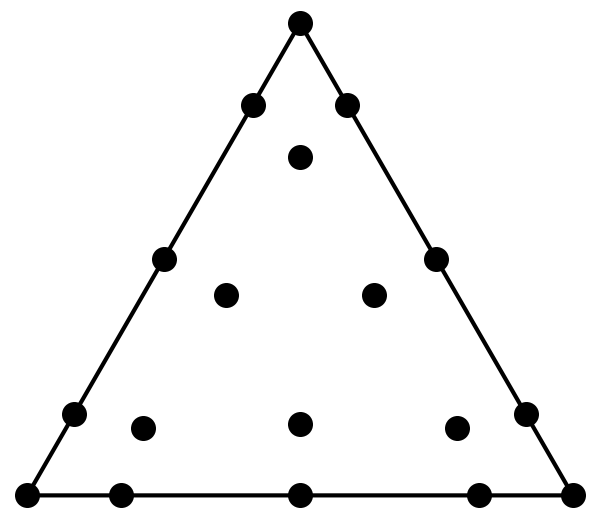}}
  \subfigure[$p=4$ \label{fig:nodes_p4}]{%
        \includegraphics[width=0.19\textwidth]{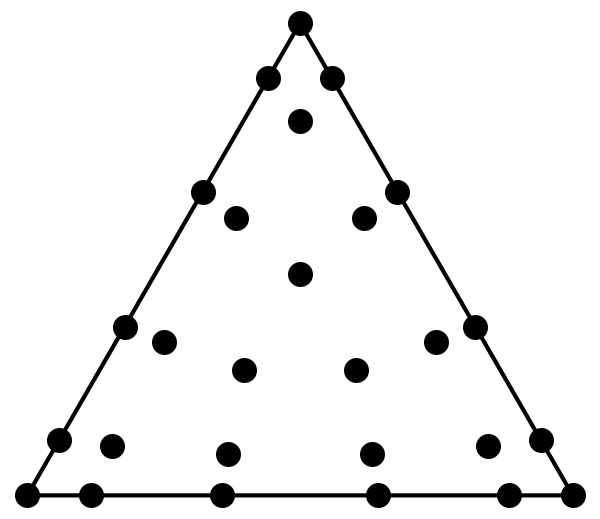}}
  \caption{Node distributions of the SBP operators used for the CSBP discretizations.\label{fig:nodes}}
\end{figure}

\subsubsection{Computation of the face operator $\Exk$}

To describe the construction of $\Exk$, consider the $p=1$ element on the equilateral triangle with the node ordering shown in Figure~\ref{fig:indices}.  For each face $\gamma \in \{1,2,3\}$ on the triangle, we define a restriction operator $\Rgk$ that maps the solution on the volume nodes to the face nodes.  This operator is analogous to the restriction operator $\Rsk$ that maps global degrees of freedom to the degrees of freedom on element $\kappa$.  For example, the operator mapping from the volume nodes to face 3 is given by 
\begin{equation*}
\mat{R}_{3,\kappa} =
\begin{bmatrix}
0 & 0 & 1 & 0 & 0 & 0 & 0 \\
1 & 0 & 0 & 0 & 0 & 0 & 0 \\
0 & 0 & 0 & 0 & 0 & 1 & 0
\end{bmatrix}.
\end{equation*}

\begin{figure}[tbp]
    \subfigure[$p=1$ node indices \label{fig:p1indices}]{%
    \includegraphics[width=0.24\textwidth]{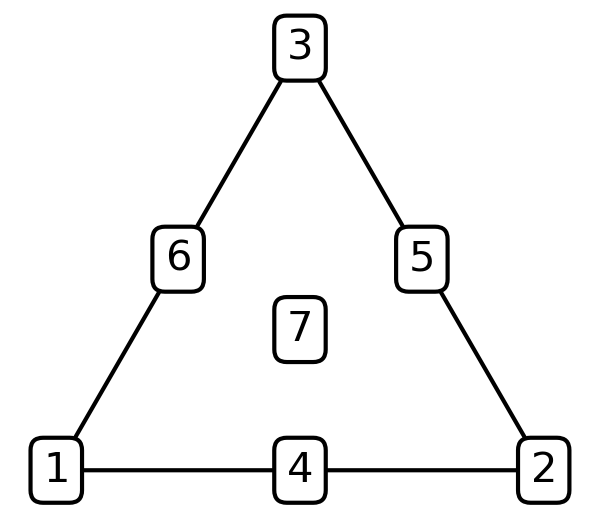}}
  \subfigure[face 1 indices \label{fig:p1indices_f1}]{%
    \includegraphics[width=0.24\textwidth]{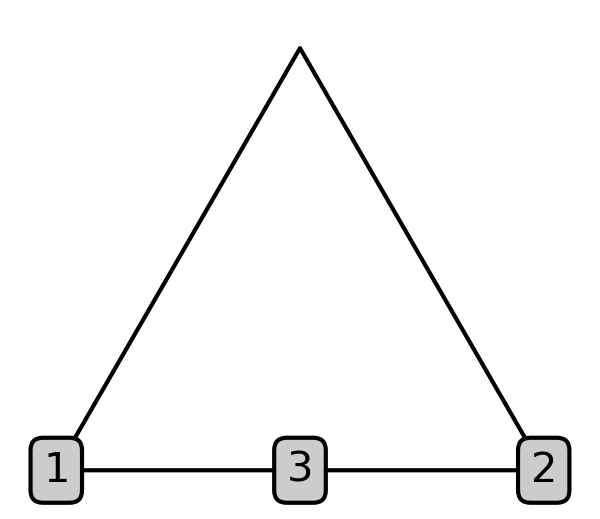}}
  \subfigure[face 2 indices \label{fig:p1indices_f2}]{%
        \includegraphics[width=0.24\textwidth]{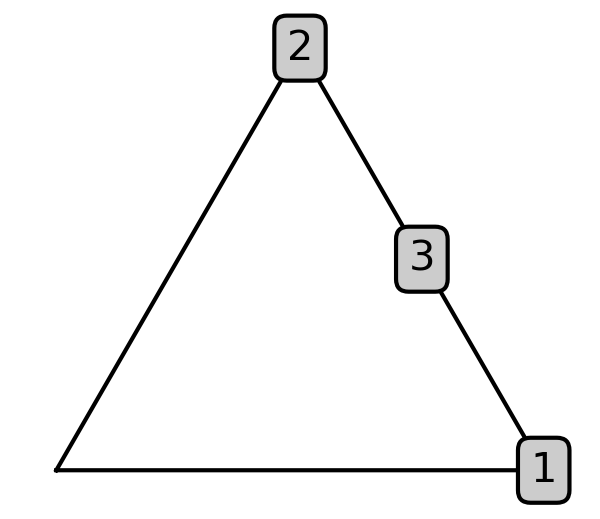}}
  \subfigure[face 3 indices \label{fig:p1indices_f3}]{%
        \includegraphics[width=0.24\textwidth]{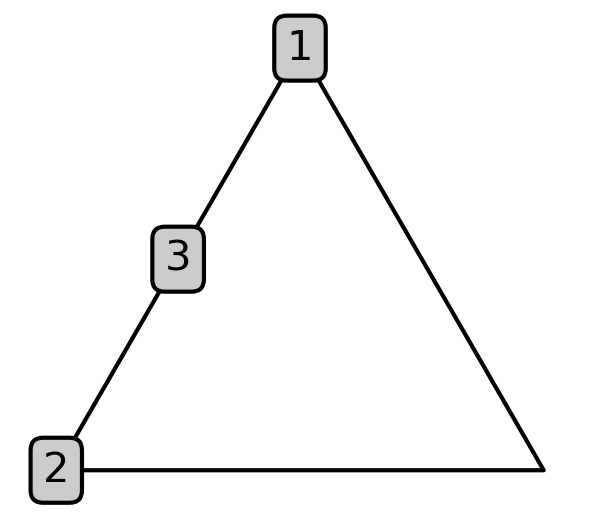}}
  \caption{Volume node indices for $p=1$ and the face-based subindices.\label{fig:indices}}
\end{figure}

All three faces use the same LGL-based quadrature rule --- recall that the face nodes are required to coincide with the LGL nodes in the face reference space. The quadrature weights will be held in the diagonal matrix $\Hg \in \mathbb{R}^{(p+2)\times (p+2)}$.  Considering the $p=1$ case, this face-quadrature matrix is given
\begin{equation*}
\Hg = \frac{1}{6} \begin{bmatrix} 1 & 0 & 0 \\ 0 & 1 & 0 \\ 0 & 0 & 4 \end{bmatrix}.
\end{equation*}
Note that the ordering in $\Hg$ for $p=1$ corresponds to the face-node ordering shown in Figures~\ref{fig:p1indices_f1}--\ref{fig:p1indices_f3}, in which the face-midpoint node is ordered last.

Using the $\Rgk$ and $\Hg$, we define the symmetric part of $\Qxk$ as follows \cite{Fernandez2017simultaneous}:
\begin{equation*}
\Exk = \sum_{\gamma=1}^{3} n_{\xi,\gamma} \Rgk^T \Hg \Rgk,
\end{equation*}
where $n_{\xi,\gamma}$ is the $\xi$ component of the outward normal vector on face $\gamma$; thus, for the equilateral triangle, $n_{\xi,1} = 0$, $n_{\xi,2} = \sqrt{3}/2$, and $n_{\xi,3} = \sqrt{3}/2$.  It is easy to show that $\Exk$, as constructed above, is a diagonal matrix.  Furthermore, it satisfies the accuracy condition \ref{sbp:Ex} in Definition~\ref{def:sbp} with $r=p$.  To see this, let $\bm{p}$ and $\bm{q}$ denote two arbitrary polynomials, $\fnc{P}$ and $\fnc{Q} \in \poly{p}(\Omega_{\kappa})$, evaluated at the nodes $\nodes$.  Then, 
\begin{equation*}
\bm{p}^{T}\Exk\bm{q} = \sum_{\gamma=1}^{3} n_{\xi,\gamma} \bm{p}^T \Rgk^T \Hg \Rgk \bm{q} 
= \sum_{\gamma=1}^{3} n_{\xi,\gamma} \int_{\Gamma_\gamma} \fnc{P} \, \fnc{Q} \, d \Gamma
=  \int_{\partial\Omega_{\kappa}} \fnc{P} \, \fnc{Q}\; n_{\xi} \, d \Gamma,
\end{equation*}
since the face quadrature is exact for the product $\fnc{P} \, \fnc{Q} \in \poly{2p}(\Gamma_{\gamma})$.

\subsubsection{Computation of the skew-symmetric matrix $\Sxk$}

Finally, we use $\Sxk$ to satisfy the accuracy conditions \ref{sbp:accuracy} in Definition~\ref{def:sbp}.  Let $\mat{L} \in \mathbb{R}^{\nk \times n_p}$ hold basis functions for the (total) degree $p$ polynomial space $\poly{p}(\Omega_{\kappa})$ evaluated at the nodes $\nodes$, where $n_p = (p+1)(p+2)/2$ is the number of basis functions in two dimensions.  In addition, let $\mat{L}' \in \mathbb{R}^{\nk \times n_p}$ denote the $\xi$-direction partial derivatives of the basis functions, also evaluated at the nodes.  Then the accuracy condition can be written concisely as (recall that $\Dxk = \Hk^{-1}(\Sxk + \frac{1}{2} \Exk)$)
\begin{equation}\label{eq:accuracy}
\Dxk \mat{L} = \mat{L}' \qquad \Rightarrow \qquad \Sxk \mat{L} = \Hk \mat{L}' - \frac{1}{2}\Exk \mat{L}.
\end{equation}

There are $\nk n_p$ equations in \eqref{eq:accuracy} and $\nk(\nk-1)/2$ unknowns in the skew-symmetric matrix $\Sxk$; thus, at first glance, it is not obvious that \eqref{eq:accuracy} is solvable.  For example, if we consider a triangular element for $p=2$, then there are $n_p = 6$ basis functions and it would appear we need $\nk \geq 11$ to have a sufficient number of unknowns in $\Sxk$.  However, the number of unique equations is greatly reduced by the quadrature accuracy built into $\Hk$ and $\Exk$; that is, we are saved by the so-called compatibility conditions~\cite{Fernandez2014generalized,multiSBP}.  To see this, left multiply \eqref{eq:accuracy}  by $\mat{L}^T$ and add the transpose of the result.  This produces the \emph{symmetric} $n_p \times n_p$ matrix equation
\begin{equation}\label{eq:compatability}
\underbrace{\mat{L}^T \Sxk \mat{L} + \mat{L}^T \Sxk^T \mat{L}}_{\ds = \mat{0}} = \underbrace{\mat{L}^T \Hk \mat{L}' + (\mat{L}')^T \Hk \mat{L} - \mat{L}^{T} \Exk \mat{L}}_{\ds = \mat{0}}.
\end{equation}
On the left, we used the skew-symmetry of $\Sxk$ to conclude that $\mat{L}^T (\Sxk + \Sxk^T) \mat{L} = \mat{0}$, while on the right we used the fact that $\Hk$ and $\Exk$ are degree $2p$ exact cubatures\footnote{In this work, $\Hk$ is $2p$ exact, but, as was shown in \cite{multiSBP}, one only needs $2p-1$ exactness for the operators to exist.}:
\begin{align*}
\left[ \mat{L}^T \Hk \mat{L}' + (\mat{L}')^T \Hk \mat{L} \right]_{ij} &= \int_{\Omega_\kappa} \left(\fnc{L}_{i} \frac{\partial \fnc{L}_{j}}{\partial \xi}  + \fnc{L}_{j} \frac{\partial \fnc{L}_{i}}{\partial \xi} \right) \,d\Omega \\
& = \int_{\partial \Omega_\kappa} \fnc{L}_i \fnc{L}_j \; n_{\xi} \, d\Gamma 
= \left[ \mat{L}^{T} \Exk \mat{L} \right]_{ij},
\end{align*}
where we have used $\fnc{L}_i$ and $\fnc{L}_{j}$ to denote the $i$th and $j$th basis functions, respectively.

Consequently, due to the symmetry of the matrix equation \eqref{eq:compatability}, there are $n_p(n_p + 1)/2$ equations automatically satisfied in \eqref{eq:accuracy}.  Thus, as the inequalities below demonstrate, the number of unknowns is always greater than or equal to the number of equations provided the number of nodes is greater than or equal to the number of basis functions, $\nk \geq n_p$:
\begin{alignat*}{3}
& & (\nk - n_p)\left[(\nk - n_p) - 1\right] &\geq 0, & \qquad &\forall\; \nk \geq n_p \\
&\Rightarrow \qquad &
\underbrace{\frac{\nk (\nk - 1)}{2}}_{\text{num. unknowns}} &\geq 
\underbrace{\nk n_p - \frac{n_p(n_p+1)}{2}}_{\text{num. equations}}, & \qquad &\forall\; \nk \geq n_p
\end{alignat*}

The bottomline is that $\Sxk$ is underdetermined by \eqref{eq:accuracy} whenever $\nk > n_p + 1$, so we need to introduce additional conditions.  In \cite{multiSBP} and \cite{Fernandez2017simultaneous}, the authors addressed this underdetermined problem by minimizing the Frobenius norm of $\Sxk$ under the constraint \eqref{eq:accuracy}.  In the present work, I instead minimize the Frobenius norm of $\Hk^{-1} \Sxk$, which is motivated by the desire to bound the spectral radius of this element-local derivative operator.

\begin{remark}
Recall that the $\Exk$ terms cancel along the element interfaces, so $\Hk^{-1} \Sxk$ is more relevant than $\Hk^{-1} \Qxk$ for the global C-SBP operators $\Dx$ and $\Dy$.
\end{remark}

In summary, the skew-symmetric matrix $\Sxk$ is determined by the following optimization problem:
\begin{equation*}
\min_{\Sxk}  \| \Hk^{-1} \Sxk \|_{F}^2, \qquad \text{subject to} \quad \Sxk \mat{L} = \Hk \mat{L}' - \frac{1}{2}\Exk \mat{L}.
\end{equation*}
In practice, I solve this convex, quadratic optimization problem by converting it to an equivalent 2-norm formulation, so I can apply the standard Moore-Penrose pseudo inverse.  I store the unknown entries corresponding to the lower-triangular part of $\Sxk$ in the one-dimensional vector $\bm{s} \in \mathbb{R}^{\nk(\nk-1)/2}$, where the matrix and vector entries are related by
\begin{equation*}
s_{m(i,j)} = (\Sxk)_{ij},\qquad \text{with}\quad m(i,j) \equiv j + (i-1)(i-2)/2,
\end{equation*}
for all $i=2,3,\ldots,\nk$ and $j = 1,2,\ldots,i$.  Using the vector of unknowns $\bm{s}$, the objective function becomes
\begin{align*}
\| \Hk^{-1} \Sxk \|_{F}^2 &= \sum_{i=1}^{\nk} \sum_{j=1}^{\nk} \frac{1}{(\Hk)_{ii}^2} (\Sxk)_{ij}^2 \\
&= \sum_{i=2}^{\nk} \sum_{j=1}^{i} \left(\frac{1}{(\Hk)_{ii}^2} + \frac{1}{(\Hk)_{jj}^2}\right) (\Sxk)_{ij}^2 \\
&= \bm{s}^{T} \mat{W} \bm{s},
\end{align*}
where $\mat{W}$ is a diagonal weighting matrix given by $(\mat{W})_{m,m} = 1/(\Hk)_{ii}^2  + 1/(\Hk)_{jj}^2$.  Let $\mat{A} \bm{s} = \bm{b}$ denote the vector form of the matrix equation $\Sxk \mat{L} = \Hk \mat{L}' - \frac{1}{2}\Exk \mat{L}$, and define the scaled solution $\tilde{\bm{s}} = \mat{W}^{\frac{1}{2}} \bm{s}$.  Then the optimization problem is equivalent to
\begin{equation*}
\min_{\tilde{\bm{s}}} \| \tilde{\bm{s}} \|_{2}^{2}, \qquad \text{subject to} \quad \mat{A} \mat{W}^{-\frac{1}{2}} \tilde{\bm{s}} = \bm{b}.
\end{equation*}
Once I find the minimum 2-norm solution to $\mat{A} \mat{W}^{-\frac{1}{2}} \tilde{\bm{s}} = \bm{b}$, I can recover $\bm{s} = \mat{W}^{-\frac{1}{2}} \tilde{\bm{s}}$ and, hence, $\Sxk$.

\subsection{Alternative constructions for the LPS operators}\label{sec:construct_lps}

I conclude this section by generalizing the construction and application of the local-projection-stabilization operator.  I describe three approaches: one suitable for element-based schemes that have a $2p$ exact norm; one for element-based schemes with $2p-1$ exact norms; and one based on reconstruction, which is suitable for traditional finite-difference schemes.

\subsubsection{LPS operator for $2p$ exact norms}

Recall the continuous $L^2$ projection operator applied to some $\fnc{U} \in L^2(\Omega_\kappa)$: find $\tilde{\fnc{U}} \in \poly{p}(\Omega_\kappa)$ such that  
\begin{equation}\label{eq:L2proj_tri}
\int_{\Omega_\kappa} \tilde{\vfnc}(\ufnc - \tilde{\ufnc}) \, d\Omega = 0, \qquad \forall\, \tilde{\vfnc} \in \poly{p}(\Omega_\kappa).
\end{equation}
As in Section~\ref{sec:2d_stab}, let $\{\fnc{L}_{i}\}_{i=1}^{n_p}$ denote an orthonormal basis for $\poly{p}(\Omega_\kappa)$, and let $\mat{L}_{ij} = \fnc{L}_{j}(\xi_i,\eta_i)$ be the $\nk \times n_p$ matrix holding the values of this basis at the nodes $\Xi_\kappa$.

%Assume that one of the basis functions in $\{\fnc{L}_{i}\}_{i=1}^{n_p}$ is the normalized constant function.

When the norm matrix $\Hk$ is exact for degree $2p$ polynomials, the $L^2$ projection operator can be discretized as described in Section~\ref{sec:2d_stab} --- that is, by $\tilde{\bm{u}}_\kappa = \mat{L} \mat{L}^T \Hk \uk$ --- and the ``high-frequency'' modes in $\uk \in \mathbb{R}^{\nk}$ are given by  
\begin{equation*}
\Prjk \uk = \left(\mat{I} - \mat{L}\mat{L}^T \Hk\right)\uk,
\end{equation*}
where $\mat{I} \in \mathbb{R}^{\nk\times \nk}$ is the identity matrix.  As mentioned previously in Section~\ref{sec:2d_stab}, this is the construction for $\Prjk$ that I use in this work, since the chosen $\Hk$ define $2p$ exact cubatures.

\subsubsection{LPS operator for $2p-1$ exact norms}

If the norm matrix is only $2p-1$ exact, it can still be used to discretize \eqref{eq:L2proj_tri}.  In this case one obtains the approximate $L^2$ projection 
\begin{equation*}
\tilde{\bm{u}}_{\kappa} = \mat{L} (\mat{L}^T \Hk \mat{L})^{-1} \mat{L} \Hk \uk,
\end{equation*} 
since $\mat{L}^T \Hk \mat{L}$, while nonsingular, is no longer equal to the identity matrix.  Despite not being exact for all polynomials in $\poly{p}(\Omega_\kappa)$, this projection operator remains exact for constant functions; consequently, the corresponding projection-based stabilization is conservative.

\subsubsection{LPS operator based on reconstruction}

More generally, the projection operator can use some form of reconstruction, such as polynomial regression.  I will illustrate this using a traditional finite-difference scheme on a uniform mesh in one-dimension, as shown in Figure~\ref{fig:Prj_recon}.  

\begin{figure}[tbp]
  \begin{center}
    \includegraphics[width=0.9\textwidth]{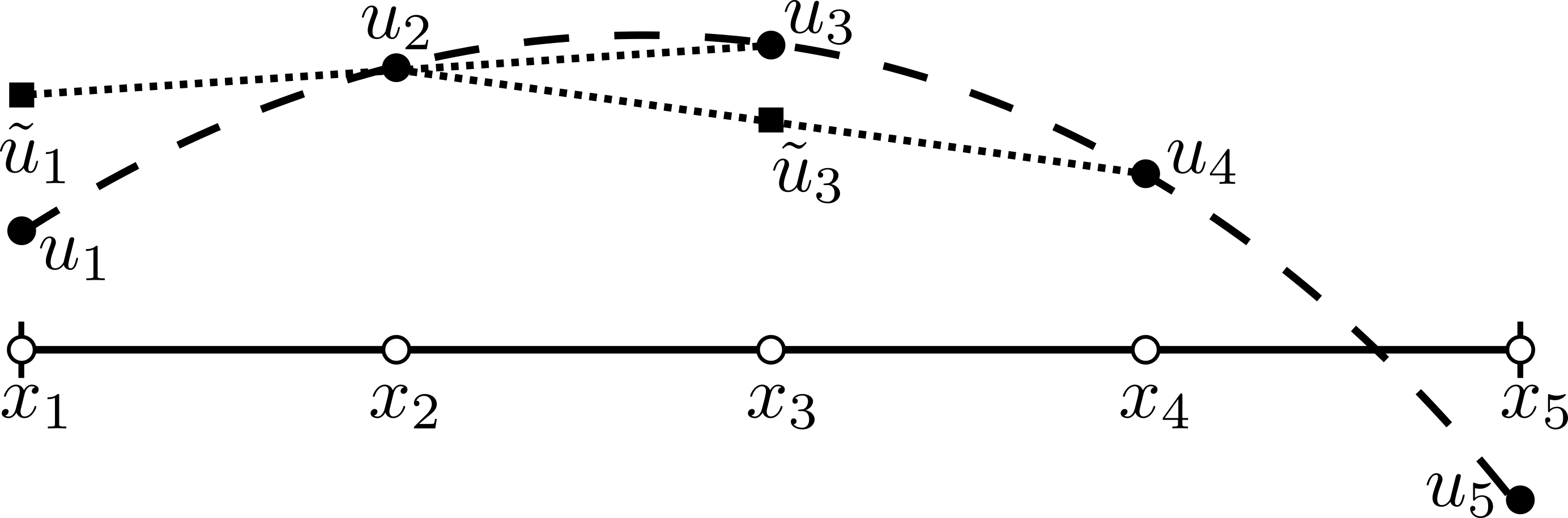}
    \caption[]{Projection operator for finite-difference methods.\label{fig:Prj_recon}}
  \end{center}
\end{figure}

Suppose we want to approximate projection onto locally linear functions.  One such projection can be defined using a simple average at interior nodes; for node $i$ we would have $\tilde{u}_i = (u_{i+1} + u_{i-1})/2$.  At boundary nodes, we can use extrapolation from the interior; for example, at node $i=1$ we could use $\tilde{u}_1 = 2 u_2 - u_3$.  We can then use the difference $u_i - \tilde{u}_i$ to define $\Prjk$, which, as above, can be used to isolate ``high-frequency'' modes.  For the example in Figure~\ref{fig:Prj_recon}, using 5 nodes and a second-order SBP operator, we obtain
\begin{equation*}
\Prjk = \underbrace{\begin{bmatrix}
1 & 0 & 0 & 0 & 0 \\
0 & 1 & 0 & 0 & 0 \\
0 & 0 & 1 & 0 & 0 \\
0 & 0 & 0 & 1 & 0 \\
0 & 0 & 0 & 0 & 1 
\end{bmatrix}}_{\ds \text{extracts}\; u_i}
- 
\underbrace{\frac{1}{2}
\begin{bmatrix}
0 & 4 & -2 & 0 & 0 \\
1 & 0 & \phm1 & 0 & 0 \\
0 & 1 & \phm0 & 1 & 0 \\
0 & 0 & \phm1 & 0 & 1 \\
0 & 0 & -2 & 4 & 0 
\end{bmatrix}}_{\ds \text{defines}\; \tilde{u}_{i}}
= 
\frac{1}{2}
\begin{bmatrix}
\phm2 & -4 & \phm2 & \phm0 & \phm0 \\
-1 & \phm2 & -1 & \phm0 & \phm0 \\
\phm0 & -1 & \phm2 & -1 & \phm0 \\
\phm0 & \phm0 & -1 & \phm2 & -1 \\
\phm0 & \phm0 & \phm2 & -4 & \phm2 
\end{bmatrix}
\end{equation*}

Using $\Prjk$, the norm matrix $\Hk$, and the scaling matrix $\mat{A}_{\kappa}$, we can define the LPS operator exactly as we did for the element-based discretizations: $\Mlps = \Prjk^T \Hk \mat{A}_\kappa \Prjk$.  Continuing our second-order-accurate example on a 5-node grid with $\Hk = \frac{h}{2} \mydiag(1,2,2,2,1)$, and assuming $\mat{A}_\kappa = \mat{I}$ for simplicity, the LPS operator is 
\begin{equation*}
\Mlps = \Prjk^T \Hk \mat{A}_\kappa \Prjk = 
\frac{h}{8}
\begin{bmatrix}
\phm6 & -12 & \phm6 & \phm0 & \phm0 \\
-12 & \phm26 & -16 & \phm2 & \phm0 \\
\phm6 & -16 & \phm20 & -16 & \phm6 \\
\phm0 & \phm2 & -16 & \phm26 & -12 \\
\phm0 & \phm0 & \phm6 & -12 & \phm6 
\end{bmatrix}
\end{equation*}

Readers familiar with traditional finite-difference methods may recognize that the interior scheme of $\Mlps$ becomes identical to divided-difference dissipation~\cite{Jameson1981numerical,Pulliam1986artificial,Mattsson2004stable} for nodes sufficiently far from the boundary.  Indeed, LPS provides a new perspective for divided-difference dissipation that may prove useful for constructing stable dissipation operators for finite-difference methods, particularly near the boundary.  
 
\ignore{
is constructed much like it was for the one-dimensional LGL case (when $q > p+1$).  First, consider the continuous $L^2$ projection operator applied to some $\fnc{U} \in L^2(\Omega_\kappa)$: find $\tilde{\fnc{U}} \in \poly{p}(\Omega_\kappa)$ such that  
\begin{equation}\label{eq:L2proj_tri}
\int_{\Omega_\kappa} \tilde{\vfnc}(\ufnc - \tilde{\ufnc}) \, d\Omega = 0, \qquad \forall\, \tilde{\vfnc} \in \mathbb{P}^{p}(\Omega_\kappa).
\end{equation}
Let $\{\fnc{L}_{i}\}_{i=1}^{n_p}$ denote an \emph{orthonormal} basis for $\poly{p}(\Omega_\kappa)$, and let $\mat{L}_{ij} = \fnc{L}_{j}(\xi_i,\eta_i)$ be the $\nk \times n_p$ matrix holding the values of this basis at the nodes $\Xi_\kappa$.  Then, as in the one-dimensional case, we can represent $\tilde{\ufnc}$ evaluated at the nodes as $\fnc{L} \bm{y}$ for some $\bm{y} \in \mathbb{R}^{n_p}$.  Finally, choosing the $\fnc{L}_{i}$ as the test functions $\tilde{\vfnc}$ and using $\Hk$ to perform the discrete integration we find that \eqref{eq:L2proj_tri} has the discretization
\begin{equation*}
\mat{L}^T\Hk \left(\uk - \mat{L} \bm{y}\right) = 0 \qquad \Rightarrow \qquad
\bm{y} = \mat{L}^T \Hk \uk,
\end{equation*}
where, as in the one-dimensional case, we have used $\mat{L}^T \Hk \mat{L} = \mat{I} \in \mathbb{R}^{(p+1)\times(p+1)}$, which follows from the orthonormality of the basis polynomials and the accuracy of $\Hk$.
}

\section{Results}\label{sec:results}

I use the following results to verify the accuracy and stability of the local-projection stabilization for C-SBP discretizations.  I also use the results to draw some comparisons with the more common D-SBP discretizations.

\subsection{Linear advection}

I use the constant-coefficient linear advection equation to study the accuracy and efficiency of the C-SBP discretization compared with the D-SBP method from \cite{Fernandez2017simultaneous}.  I will also use this PDE to investigate the spectra of the two discretizations.  

Consider the two-dimensional, constant-coefficient advection equation on a square periodic domain $\Omega = [0,1]^2$:
\begin{equation}\label{eq:2d_advec}
  \begin{alignedat}{2}
    \frac{\partial \ufnc}{\partial t} + \lambda_x \frac{\partial \ufnc}{\partial x} + \lambda_y \frac{\partial \ufnc}{\partial y} &= 0,
    &\qquad &\forall (x,y) \in \Omega, \\
    \ufnc(0,y,t) &= \ufnc(1,y,t), &\qquad &\forall t \geq 0, y \in [0,1], \\
    \ufnc(x,0,t) &= \ufnc(x,1,t), &\qquad &\forall t \geq 0, x \in [0,1]. \\
  \end{alignedat}
\end{equation}
where $(\lambda_x,\lambda_y) = (1,1)$ is the advection velocity.  The initial condition is the same bell-shaped function used in \cite{multiSBP} and \cite{Fernandez2017simultaneous}:
\begin{equation*}
\ufnc(x,y,0) = \begin{cases} 
1 - (4\rho^2 - 1)^5, & \text{if}\; \rho < \frac{1}{2} \\
1, & \text{otherwise,}
\end{cases}
\end{equation*}
where $\rho^2 = (x-1/2)^2 + (y-1/2)^2$.  The initial condition is plotted in Figure~\ref{fig:IC}.

I divide the domain $\Omega$ into a triangular mesh using the recursive kernel-based method described in \cite{Fernandez2017simultaneous}.  I consider four mesh levels; the level-2 mesh, which is the second coarsest mesh, is illustrated in Figure~\ref{fig:mesh_lev2}.  In order to move from one level to the next, each element in the mesh is divided, in reference space, according to a canonical subdivision, \ie, the kernel mesh; for additional details, please see \cite{Fernandez2017simultaneous}.  My motivation for using this kernel-based set of meshes is to avoid meshes in which the element size varies smoothly from element to element.  Such smooth meshes might unfairly bias the results in favor of the C-SBP discretization.

\begin{figure}[t]
  \subfigure[second coarsest mesh \label{fig:mesh_lev2}]{%
    \includegraphics[height=0.49\textwidth]{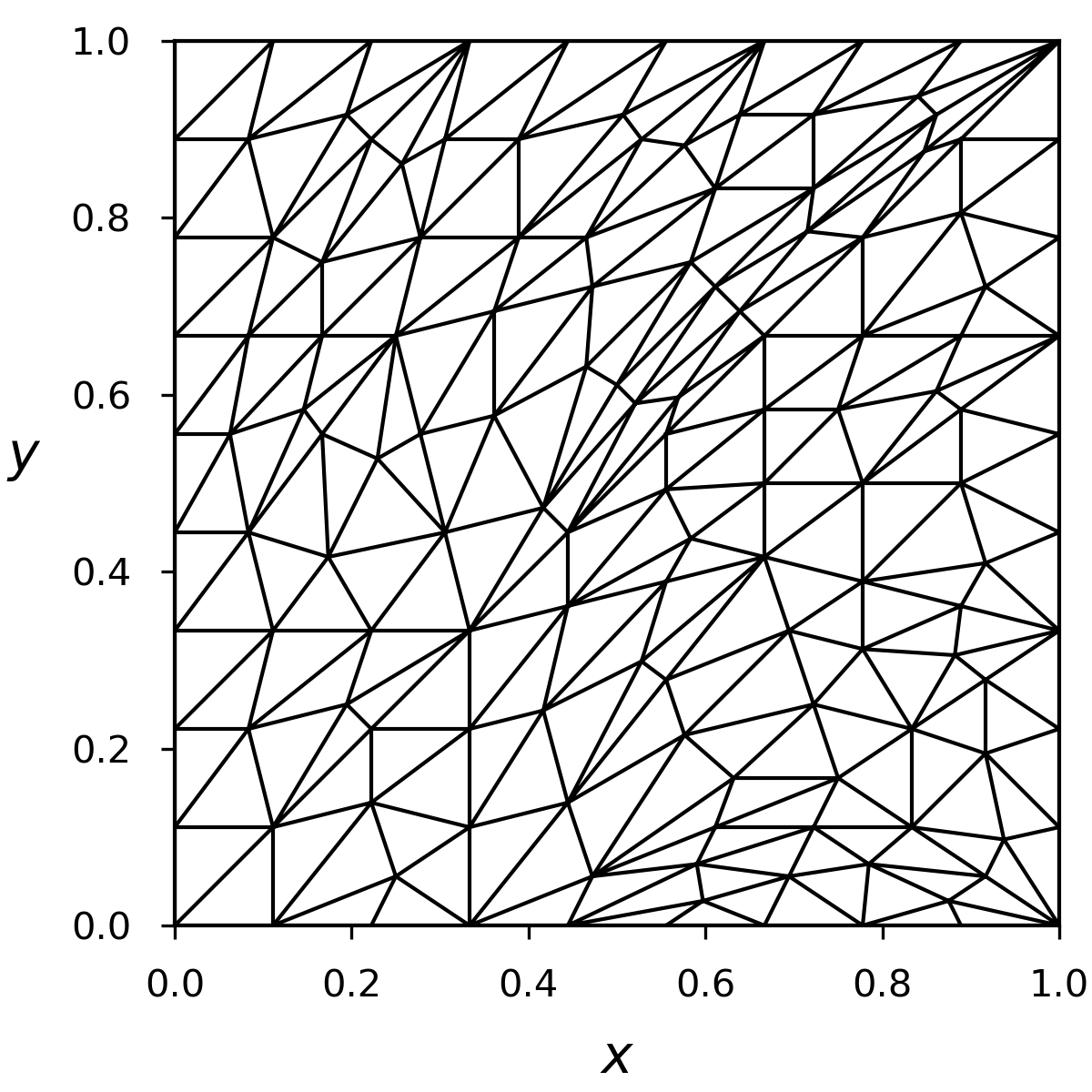}}\hfill
  \subfigure[Initial condition \label{fig:IC}]{%
        \includegraphics[height=0.49\textwidth]{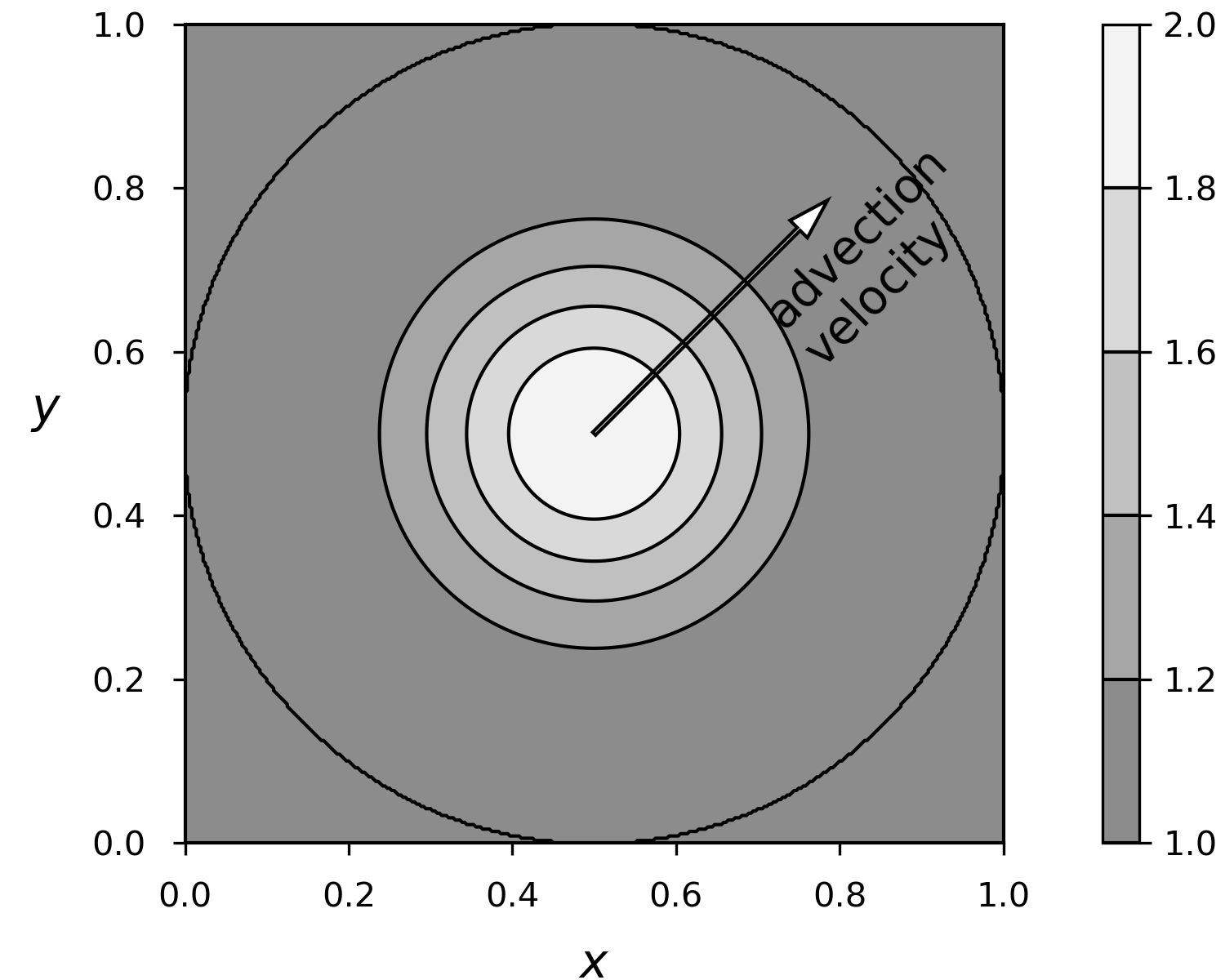}}
  \caption{Example mesh and initial condition for the linear-advection problem.}
\end{figure}

For a given mesh level, let $\Tmesh$ denote the set of element subdomains:
\begin{equation*}
  \Tmesh \equiv \left\{ \Omega_\kappa \right\}_{\kappa=1}^{K}.
\end{equation*}
Each triangle in $\Tmesh$ is the image of the reference triangle, $\Omega_{\xi} = \{ (\xi,\eta) | \xi \geq -1, \eta \geq -1, \eta \leq -\xi \}$, under an appropriate mapping.  For the linear-advection experiments, each triangle can be obtained from $\Omega_\xi$ using an affine mapping; I will discuss curvilinear elements in the context of the Euler equations.

The D-SBP discretization is the same as the one described in \cite{Fernandez2017simultaneous} and uses the $p=1$ and $p=2$ SBP-$\Omega$ operators introduced in that same paper.  The $p=3$ and $p=4$ operators from  \cite{Fernandez2017simultaneous} \emph{were not used in the present studies}, since they have only $2p-1$ exact cubatures.  Instead, I created new $p=3$ and $p=4$ operators that have $2p$ exact cubatures using the procedure described in \cite{Fernandez2017simultaneous}.  Consequently, both the D-SBP and C-SBP discretizations use operators based on $2p$ exact cubatures.

Although the C-SBP discretization has been described throughout this paper, I need to clarify two details for this numerical experiment.  First, based on the affine-mapping assumption, the global first-derivative SBP operators are given by
\begin{gather}
  \Dx = \H^{-1} \Qx \qquad\text{and}\qquad
  \Dy = \H^{-1} \Qy \notag \\
  \intertext{where}  
  \Qx \equiv \sum_{\kappa=1}^{K} \Rs_{\kappa}^T
  \left[ \phantom{-}\left(\partial_\eta y \right)_{\kappa} \Qxi - \left( \partial_\xi y \right)_{\kappa} \Qeta \right]
  \Rs_{\kappa}, \\
  \Qy \equiv \sum_{\kappa=1}^{K} \Rs_{\kappa}^T
  \left[ -\left(\partial_\eta x \right)_{\kappa} \Qxi + \left( \partial_\xi x \right)_{\kappa} \Qeta \right]
  \Rs_{\kappa}, \\
  \intertext{and}
  \H  \equiv \sum_{\kappa=1}^{K} \Rs_{\kappa}^T  J_{\kappa} \Hk \Rs_{\kappa}.
  \label{eq:Qx_and_H_2d}
\end{gather}
where $J_{\kappa} = \left[(\partial_\xi x)(\partial_\eta y) - (\partial_\xi y)(\partial_\eta x) \right]_{\kappa}$ is the determinant of the mapping Jacobian on element $\kappa$.  The Jacobian terms --- $\partial_\xi x$, $\partial_\eta x$, $\partial_\xi y$, and $\partial_\eta y$ --- and the determinant are constant over each element, so they can be computed either analytically or using the local SBP operators.

The second point that needs clarifying is the scaling function that appears in LPS.  For dimensional consistency, I use the magnitude of the advection velocity in reference space:
\begin{equation*}
\fnc{A}(\xi,\eta) = \sqrt{\lambda_\xi^2 + \lambda_\eta^2},
\end{equation*}
where
\begin{align*}
\lambda_\xi &= \fnc{J} \lambda_x \partial_x \xi + \fnc{J} \lambda_y \partial_y \xi = \phm\lambda_x \partial_\eta y - \lambda_y \partial_\eta x, \\
\lambda_\eta &= \fnc{J} \lambda_x \partial_x \eta + \fnc{J} \lambda_y \partial_y \eta = -\lambda_x \partial_\xi y + \lambda_y \partial_\xi x.
\end{align*}
Again, since the Jacobian is constant over each element, the scalar $\fnc{A}(\xi,\eta)$ can either be computed analytically or using the SBP operators.

Finally, the C-SBP and D-SBP semi-discretizations are both advanced in time using the classical fourth-order Runge-Kutta method.  I chose the time-step size for each discretization using the results of the eigenvalue spectra study, presented next.

\subsubsection{Investigation of operator spectrum}

The eigenvalues of the spatial operators are plotted in Figure~\ref{fig:eigs}.  These spectra are from discretizations on the level 2 mesh shown in Figure~\ref{fig:mesh_lev2}.

I have scaled each set of eigenvalues by the corresponding spectral radius, so the axes can use the same scale.  To indicate the relative size of the spectral-radius scaling, each figure includes the effective time step, $\Delta t$, that the discretization would need to take relative to the C-SBP $p=1$ scheme.   For example, if $\rho_1^{\textsf{C}}$ is the spectral radius of the C-SBP $p=1$ scheme, and $\rho_3^{D}$ is the spectral radius of the D-SBP $p=3$ scheme, than the effective time step is
\begin{equation*}
\Delta t = \frac{\rho_1^{\textsf{C}}}{\rho_3^{\textsf{D}}} \Delta t_{\text{ref}} \approx 0.26 \Delta t_{\text{ref}}.
\end{equation*}
For $p=1$, $p=2$, and $p=3$ the C-SBP discretizations have larger effective time steps.  For $p=4$ the D-SBP scheme has a slightly larger time step of $0.22\Delta t_{\text{ref}}$ versus $0.2\Delta t_{\text{ref}}$.

Qualitatively, the spectra of the C-SBP and D-SBP are quite distinct.  As $p$ increases, the C-SBP spectra cluster closer and closer to the imaginary axis.  In contrast, the D-SBP spectra remain clustered within a disk-shaped region.  This suggests that the C-SBP discretization may approximate the infinite-dimensional spectrum, which is pure imaginary, more efficiently than the D-SBP discretization.

\begin{figure}[tbp]
  \renewcommand{\thesubfigure}{}
  \subfigure[$p=1$ (C-SBP) \label{fig:eigCG_p1}]{%
    \includegraphics[width=0.24\textwidth]{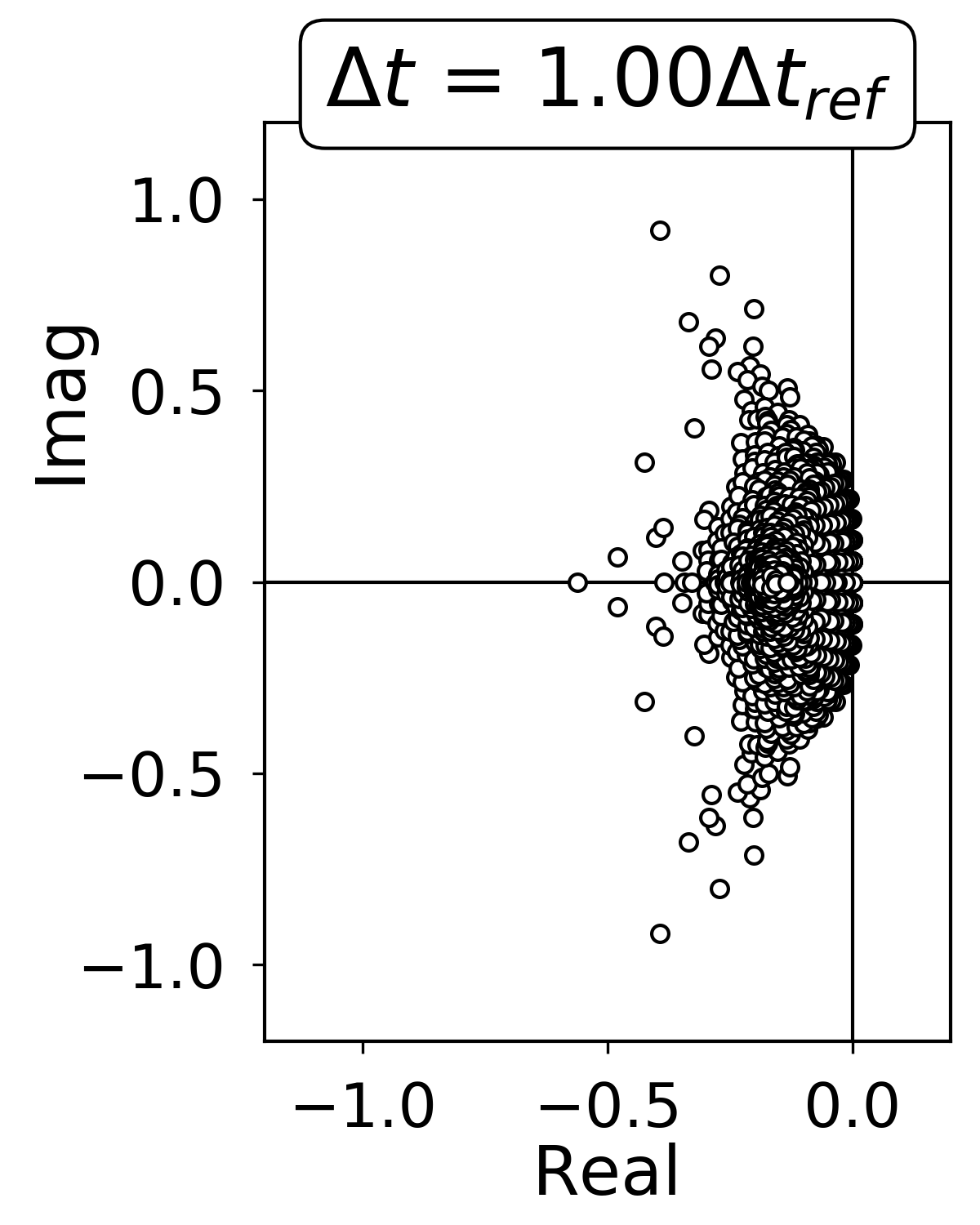}}
  \subfigure[$p=2$ (C-SBP) \label{fig:eigCG_p2}]{%
        \includegraphics[width=0.24\textwidth]{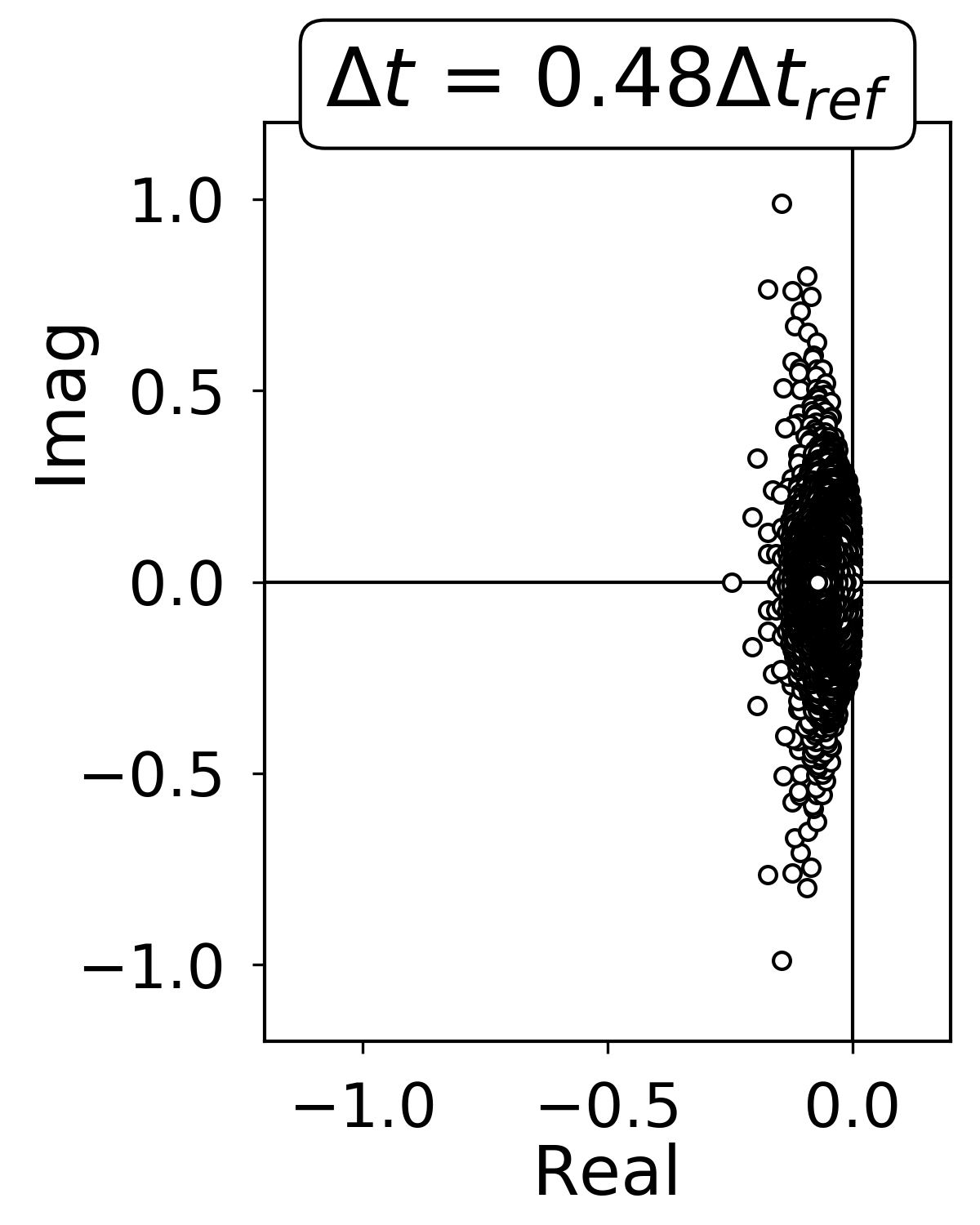}}
  \subfigure[$p=3$ (C-SBP) \label{fig:eigCG_p3}]{%
        \includegraphics[width=0.24\textwidth]{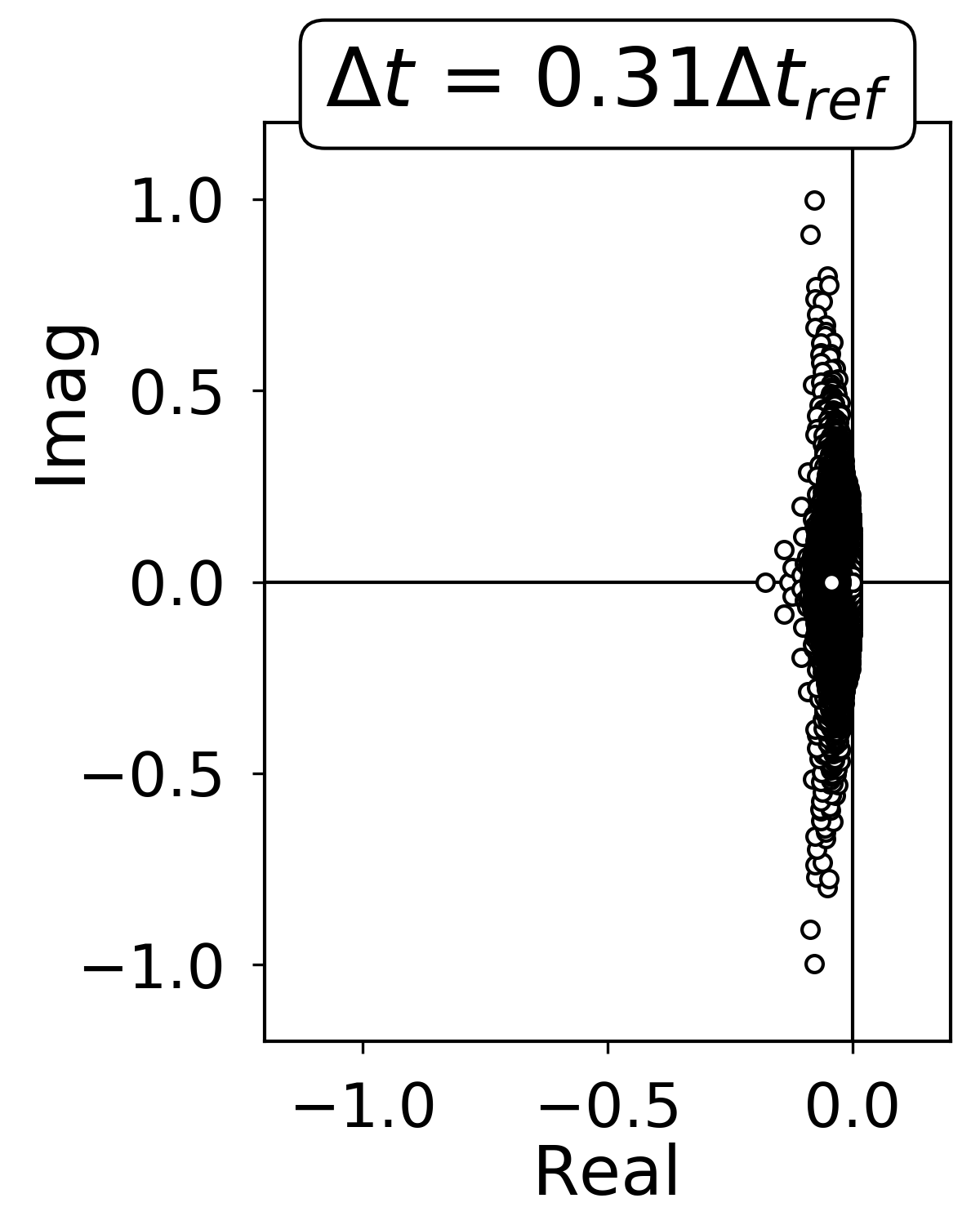}}
  \subfigure[$p=4$ (C-SBP) \label{fig:eigCG_p4}]{%
        \includegraphics[width=0.24\textwidth]{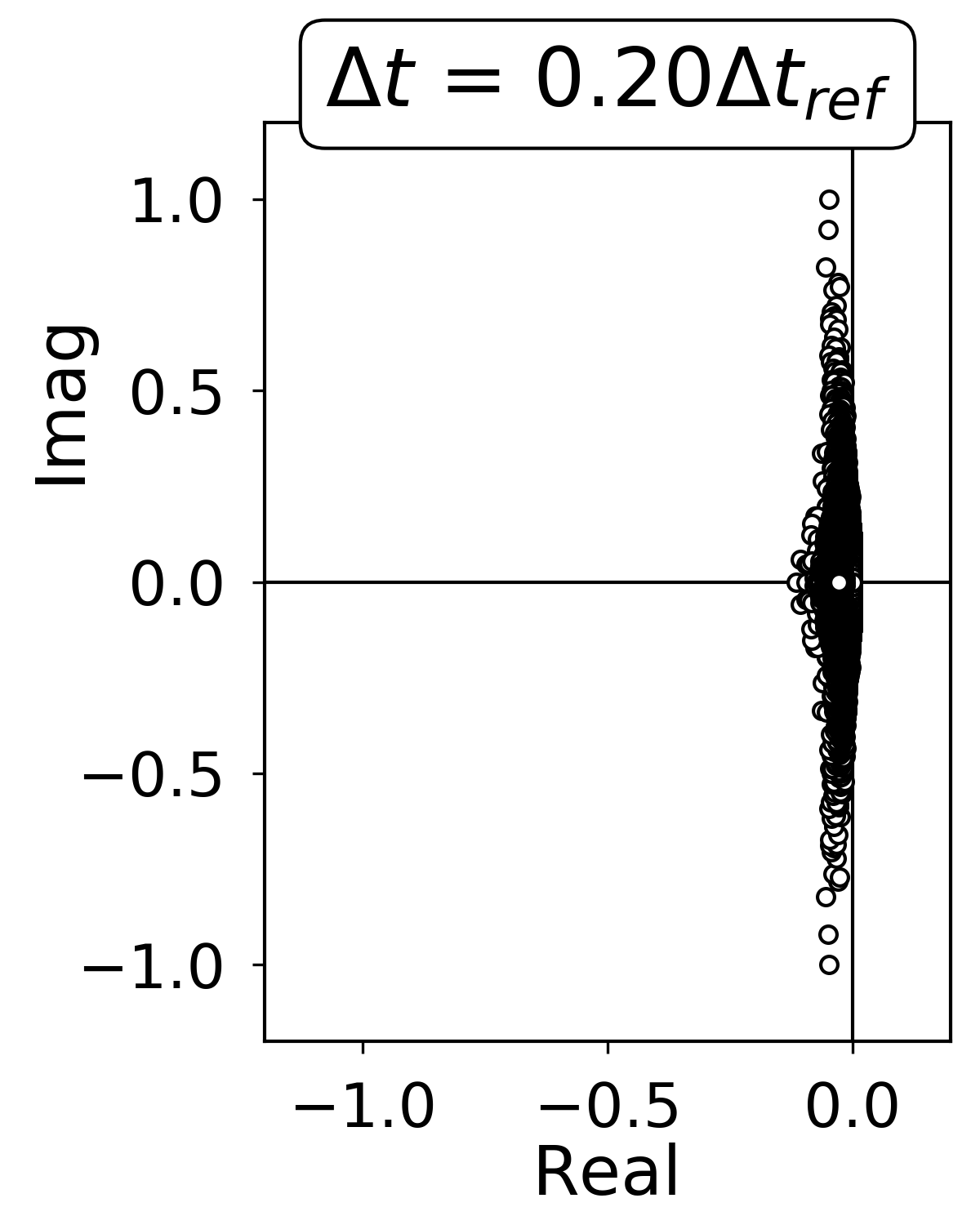}}\\
  \subfigure[$p=1$ (D-SBP) \label{fig:eigDG_p1}]{%
    \includegraphics[width=0.24\textwidth]{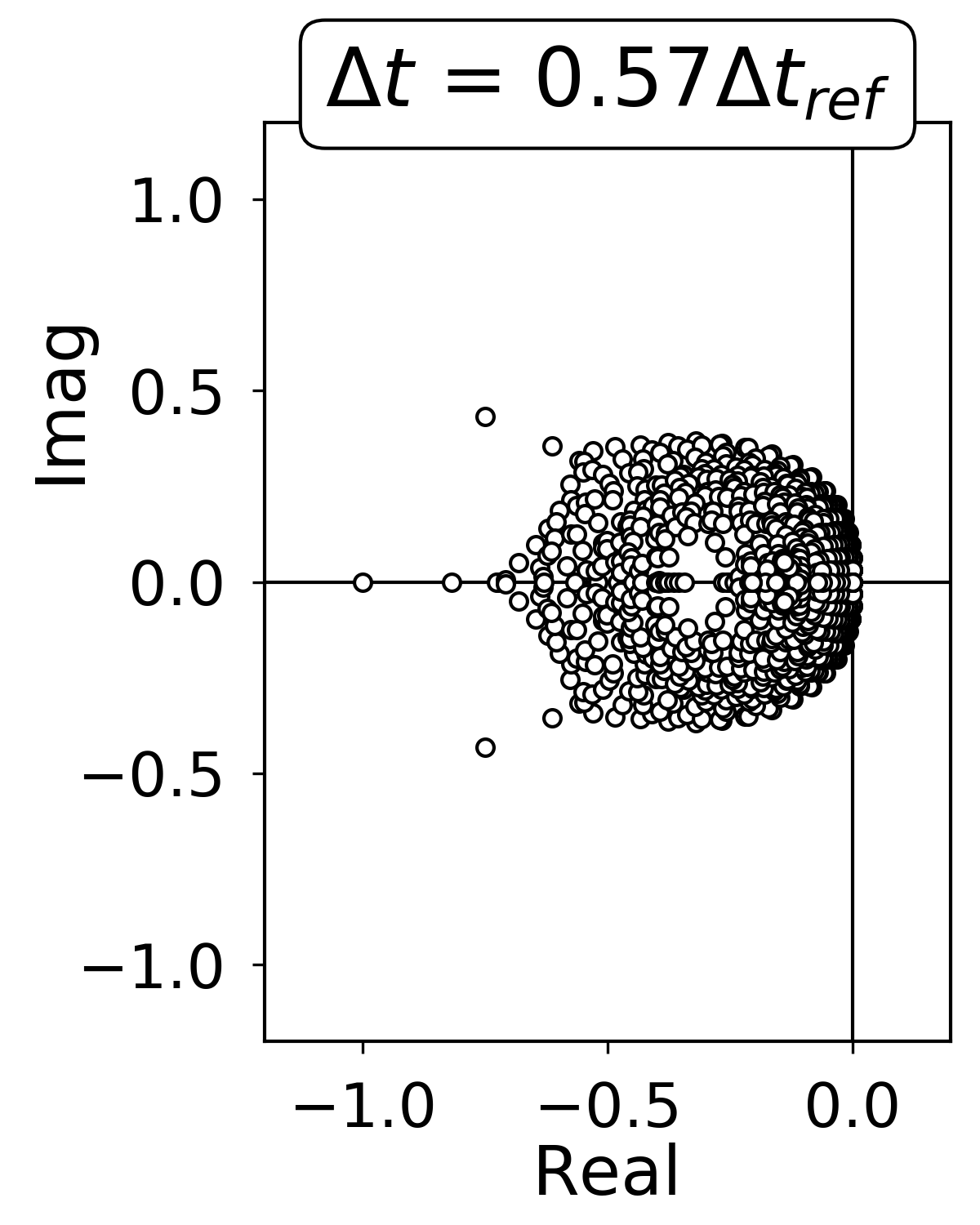}}
  \subfigure[$p=2$ (D-SBP) \label{fig:eigDG_p2}]{%
        \includegraphics[width=0.24\textwidth]{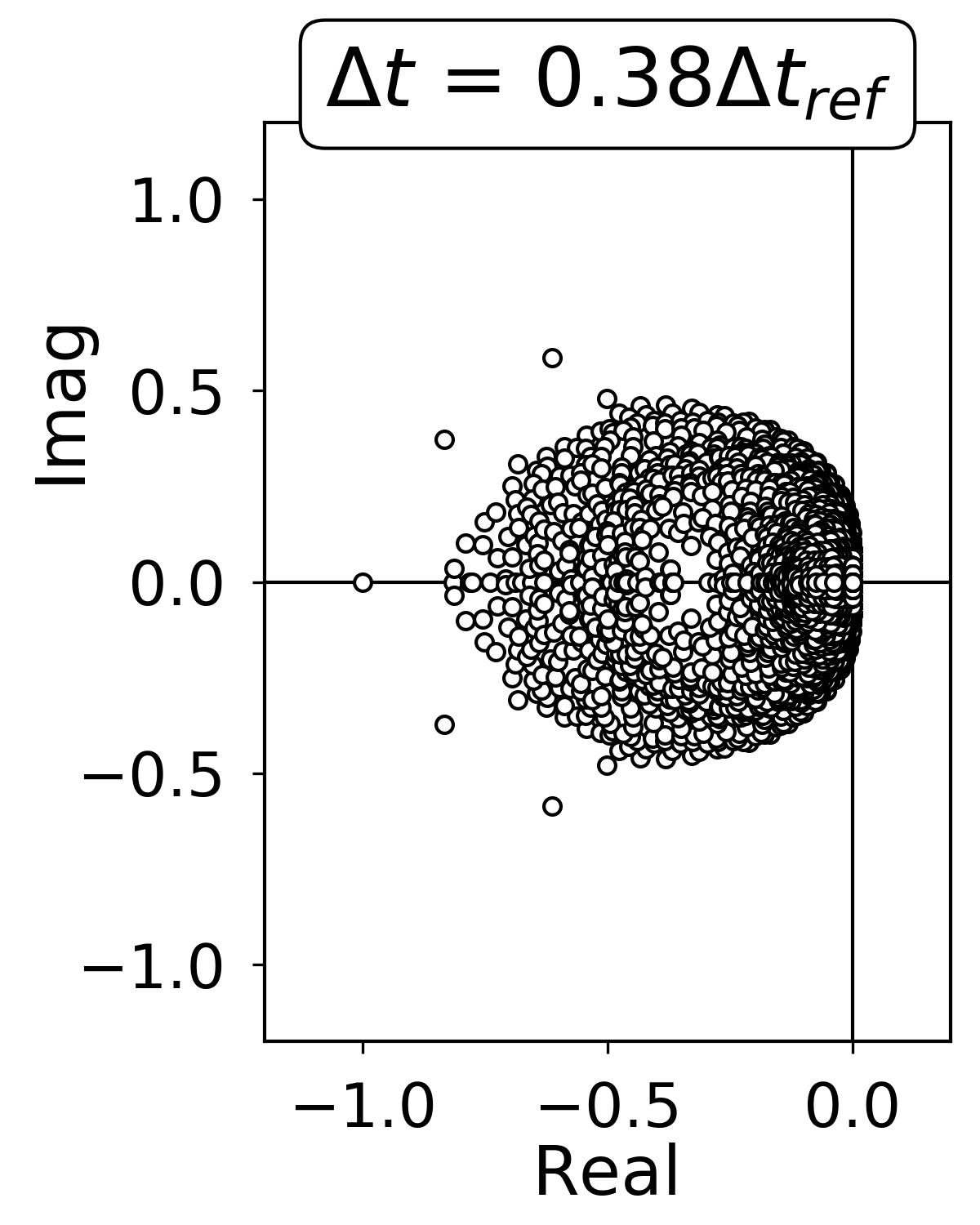}}
  \subfigure[$p=3$ (D-SBP) \label{fig:eigDG_p3}]{%
        \includegraphics[width=0.24\textwidth]{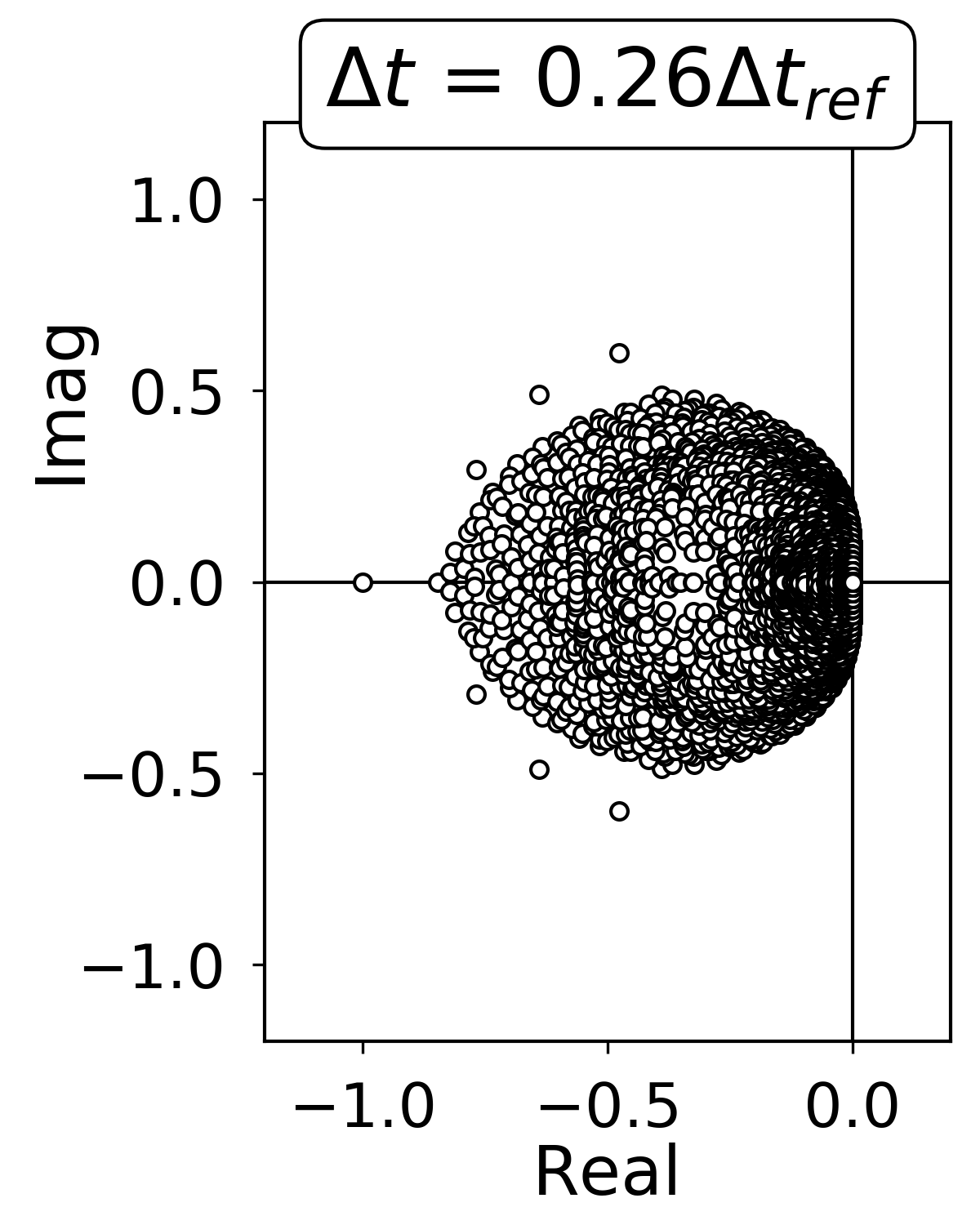}}
  \subfigure[$p=4$ (D-SBP) \label{fig:eigDG_p4}]{%
        \includegraphics[width=0.24\textwidth]{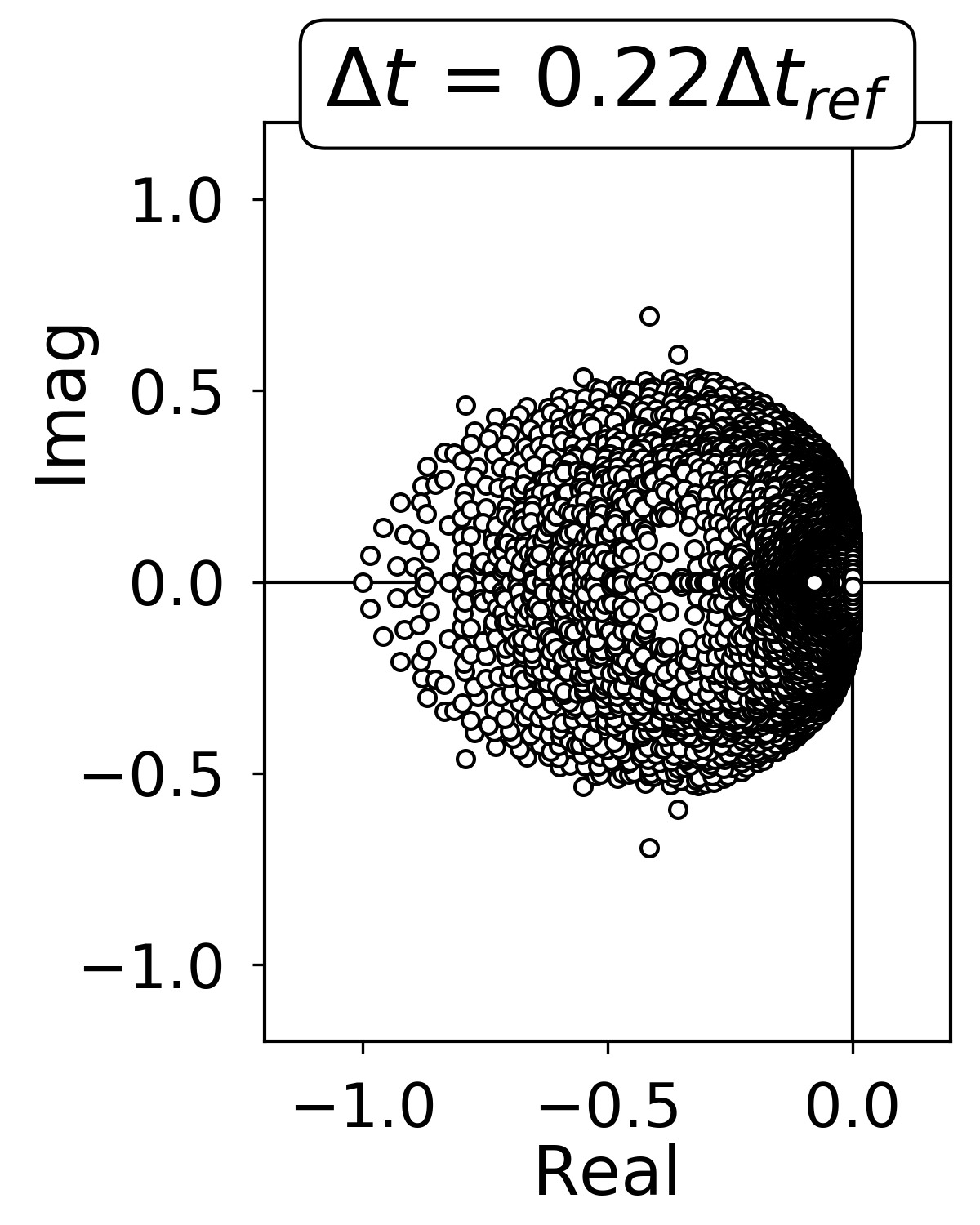}}\\
  \caption{Eigenvalue distributions of the C-SBP (upper row) and D-SBP (lower row) spatial discretizations for the linear advection problem.\label{fig:eigs}}
\end{figure}

\subsubsection{$L^2$ accuracy study}

I used all four mesh levels to conduct a mesh convergence study for the linear-advection equation.  The objective of this study was to verify that the C-SBP discretization achieves optimal, or near optimal, rates of convergence in the $L^2$ norm.  Furthermore, I wanted to investigate the relative error between the C-SBP and D-SBP schemes on the same mesh.

The equations were marched forward in time using the classical fourth-order Runge-Kutta method.  The spectra in the previous section were used to identify the maximum stable time step for each discretization.  Although the spectra were gathered for only one mesh, I observed that the ratio between the reference spectral radius, $\rho_1^{\textsf{C}}$, and the other spectral radii remained roughly constant across mesh levels.  Thus, for a given mesh level and discretization, the maximum stable time step can be determined from the ratios in Figure~\ref{fig:eigs} and the maximum reference time step for the C-SBP $p=1$ scheme.  For example, if $\Delta t_{\text{ref}} = 0.1$, then $\Delta t = 0.026$ for the D-SBP $p=3$ scheme.

The spectral radius was not available for the C-SBP $p=1$ discretization on the finest mesh; therefore, in order to estimate the maximum reference time step, \ie $\Delta t_{\text{ref}}$, I fit the following model for the spectral radius:
\begin{equation*}
\rho_{1}^{\textsf{lev}} = \frac{\sqrt{\lambda_x^2 + \lambda_y^2}}{h_{\text{ref}}(\textsf{lev})} a b^{\textsf{lev}}
= \sqrt{2} a (3 b)^{\textsf{lev}},
\end{equation*}
where $\textsf{lev}$ is the mesh level, and $h_{\text{ref}}(\textsf{lev}) = (1/3)^{\textsf{lev}}$ is a nominal element size for the kernel-based mesh: each edge is split into thirds during refinement, hence the factor of $1/3$.  I used the spectral radii from the $\textsf{lev}=2$ and $\textsf{lev}=3$ meshes to fit the above model and found that $a=2.0743$ and $b=2.0758$.  The resulting fit predicts $\rho_{1}^{1} = 17.614$ for the $p=1$ discretization on the level 1 mesh, which is a 6\% error from the true spectral radius of $18.268$.

I advanced the solution in time one period, to $t=1$, which brings the bell-shaped solution back to its initial position.  The $L^2$ error between the numerical solution at $t=1$ and the initial condition was then evaluated.  I used the SBP norm matrices $\Hk$ and Jacobian determinant $J_\kappa$ to approximate the integrals in the $L^2$ norm of the error.

Figures~\ref{fig:accuracy_CSBP} and Figures~\ref{fig:accuracy_DSBP} plot the $L^2$ error versus the nominal element size $h_{\text{ref}}(\textsf{lev}) = (1/3)^{\textsf{lev}}$ for the C-SBP and D-SBP discretizations.  For a given $p$ and mesh size, the errors for the two discretizations are comparable.  The asymptotic convergence rates are estimated using the error on the finest two grid levels and are displayed beneath the rate triangles in Figure~\ref{fig:accuracy_adv}.

\begin{figure}[t]
  \subfigure[C-SBP \label{fig:accuracy_CSBP}]{%
    \includegraphics[width=0.46\textwidth]{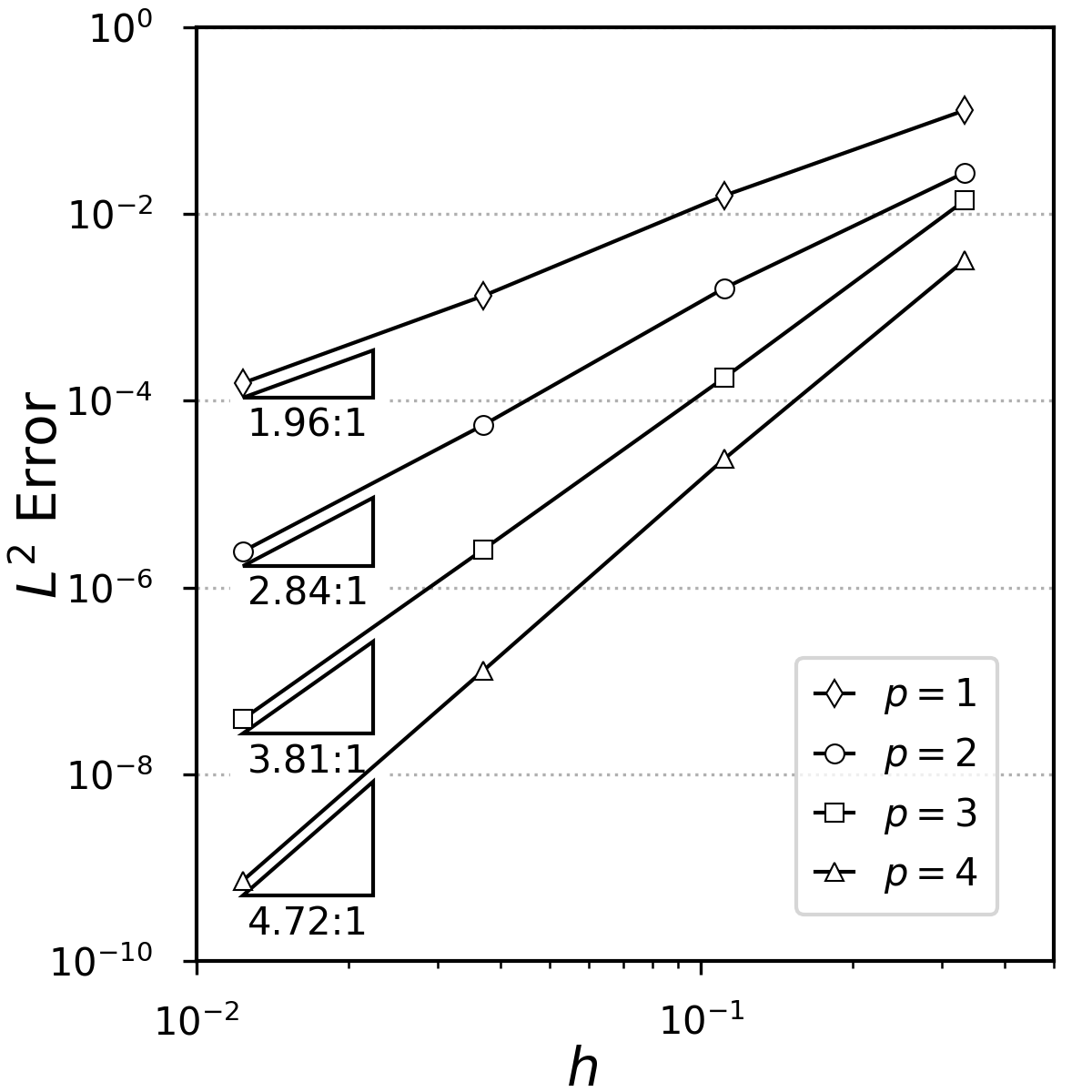}}\hfill
  \subfigure[D-SBP \label{fig:accuracy_DSBP}]{%
        \includegraphics[width=0.46\textwidth]{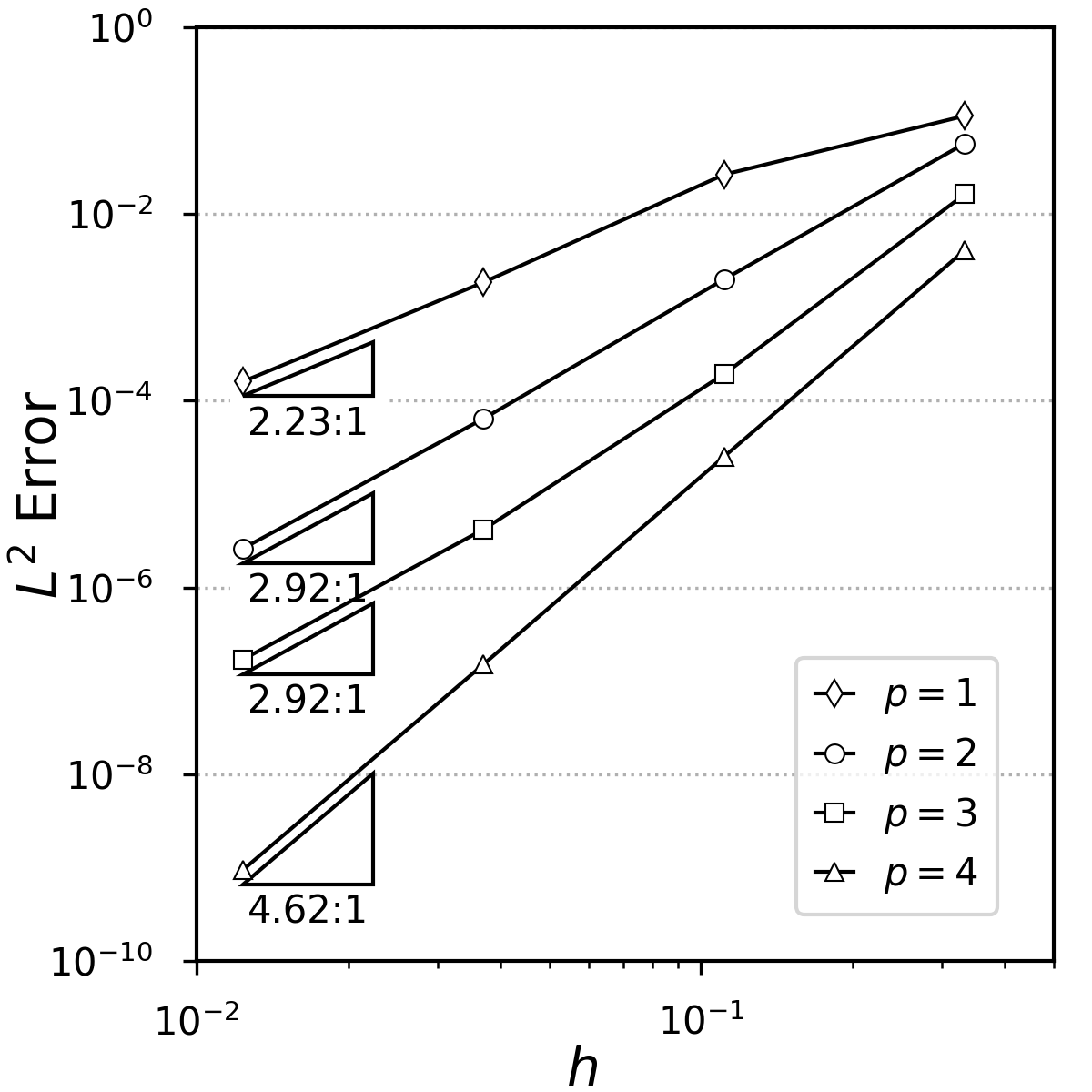}}
  \caption{$L^2$ solution error after one period versus nominal element size. \label{fig:accuracy_adv}}
\end{figure}

\subsubsection{Efficiency study}

To assess the relative efficiency of the C-SBP and D-SBP discretizations, I have plotted the $L^2$ solution error versus normalized CPU time in Figure~\ref{fig:efficiency}.  The times are normalized by the C-SBP $p=1$ discretization on the coarsest mesh.  The Julia code that I wrote to gather these results used pre-allocated work arrays to avoid expensive memory allocation, and was run on Julia version 0.6.2 with array-bound checking turned off.

The results in Figure~\ref{fig:efficiency} indicate that the C-SBP discretizations are more efficient than the D-SBP discretizations up to polynomial degree $p=4$, at which point the two schemes are comparable.  Admittedly, this is only one simple test case, and it is likely that improvements could be made to both schemes.  Furthermore, these results do not shed much light on the relative efficiency of the discretizations in the context of implicit time-marching schemes, viscous terms, and three-dimensional problems.  Nevertheless, the results in Figure~\ref{fig:efficiency} suggest that, at the very least, C-SBP schemes warrant further investigation.

\begin{figure}[t]
  \begin{center}
    \includegraphics[width=\textwidth]{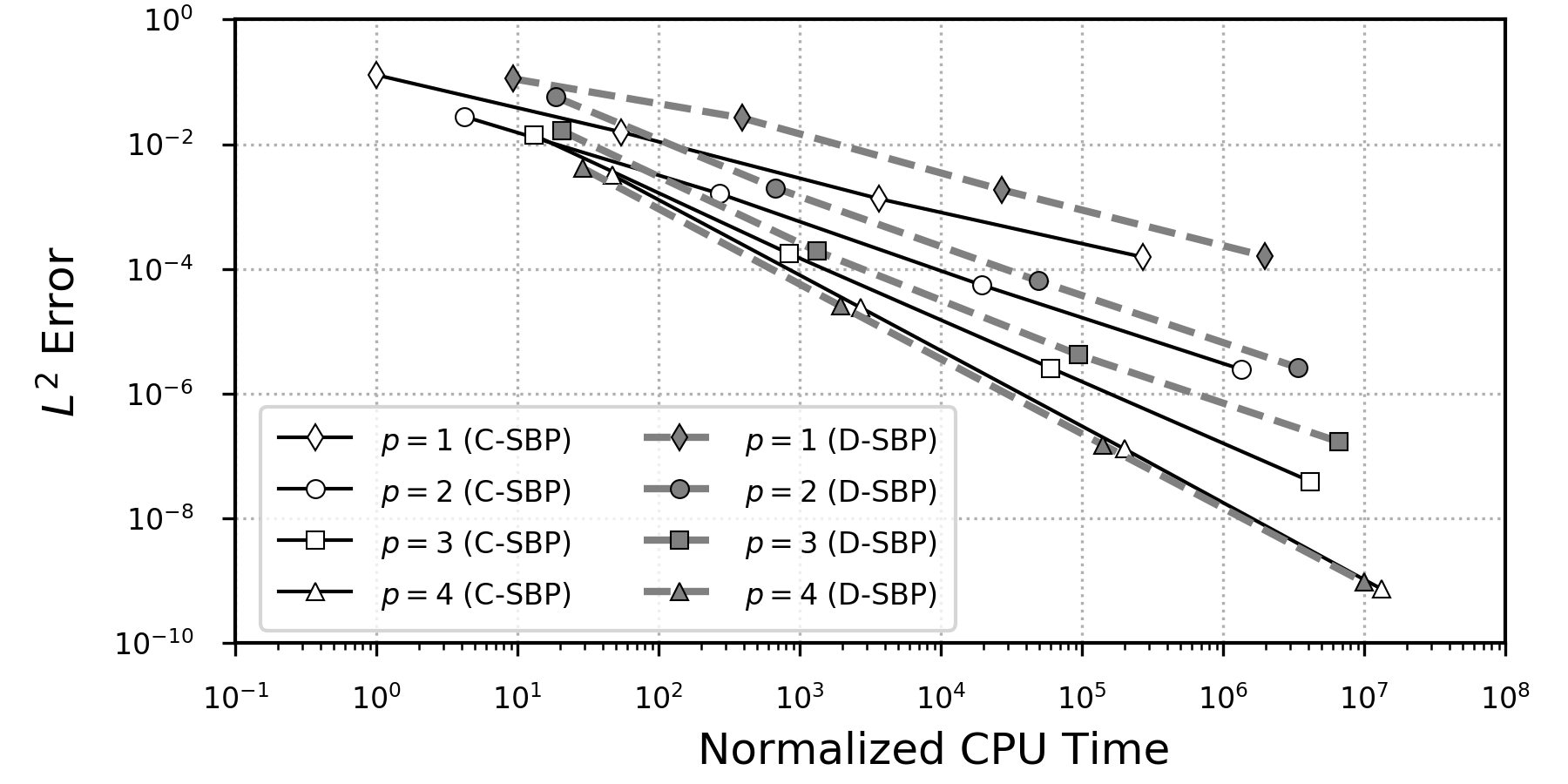}
      \caption{$L^2$ solution error after one period versus normalized CPU time. \label{fig:efficiency}}
  \end{center}
\end{figure}

\subsection{Euler equations}

The remaining numerical experiments are used to assess the C-SBP entropy-stable discretization of the Euler equations described in Section~\ref{sec:euler}.  In particular, I am interested in verifying accuracy in the context of a nonlinear system of equations, as well as verifying entropy conservation and stability.

The following studies use curvilinear elements, in contrast to the affine elements adopted for the linear-advection studies.  To ensure entropy conservation and stability, I evaluated the mapping Jacobian as described in \cite{Crean2018entropy}.  In fact, since I considered exclusively two-dimensional meshes, I used $p+1$ Lagrange elements to define the coordinate transformation and computed the metrics analytically.  I then used equations (23)--(25) from \cite{Crean2018entropy} to define the SBP operators on each element.  The resulting operators satisfy the theoretical requirements for entropy conservation and stability, but they not satisfy the polynomial-exactness condition \ref{sbp:accuracy} in Definition~\ref{def:sbp}; with the exception of constant functions, the operators differentiate polynomials only in an asymptotic sense as the mesh is refined.  This motivates the accuracy study, which I discuss next.

\subsubsection{Accuracy verification using the steady vortex}

The steady isentropic vortex is an exact, smooth solution to the Euler equations, which makes it suitable for verifying the accuracy of the C-SBP entropy-stable discretization.  The vortex flow has circular symmetry about the origin.  Specifically, its streamlines are concentric circles and its density is given by
\begin{equation*}
\rho(r) = \rho_{\textsf{in}} \left[ 1 + \frac{\gamma -1}{2} M_{\textsf{in}}^2 \left( 1 - \frac{r_{\textsf{in}}^2}{r^2}\right) \right]^{\frac{1}{\gamma -1}},
\end{equation*}
where $r$ denotes the radial polar coordinate; $r_{\textsf{in}} = 1$ is a reference radius, and $\rho_{\textsf{in}} = 2$ and $M_{\textsf{in}} = 0.95$ are the density and Mach number at $r_{\textsf{in}}$, respectively.  The remaining conservative variables can be obtained using the isentropic relations.

The domain for the steady-vortex problem is the quarter annulus $\Omega = \{ (r,\theta) \;|\; 1 \leq r \leq 3, 0 \leq \theta \leq \pi/2 \}$.  I applied an inviscid ``slip'' boundary condition, $\rho u n_x + \rho v n_y = 0$, along the inner radius, $r=r_{\textsf{in}} = 1$.  On the remaining boundaries, I provided the exact solution to characteristic-type boundary conditions.  Both boundary conditions are implemented in a dual-consistent manner~\cite{Lu2005posteriori,Hartmann2007adjoint,Hicken2014dual}, which, as we shall see, is important for functional accuracy.

\begin{remark}
While the boundary conditions are dual consistent, they are not entropy stable.  Entropy-stable boundary conditions have been proposed~\cite{Svard2013entropy}, but they are not needed for this steady problem.
\end{remark}

I created the meshes for $\Omega$ by uniformly dividing the domain into $N\times N$ quadrilaterals in polar coordinates.  Each quadrilateral was then subdivided into two triangles.  For a degree $p$ SBP discretization, the coordinate mapping for each triangle was represented using a $p+1$ Lagrange basis with $(p+2)(p+3)/2$ uniformly spaced nodes.  The coordinate mapping was then uniquely determined by mapping the Lagrange nodes to physical space.  The resulting curvilinear elements were used to define the locations of the SBP nodes (in physical space) and the mapping Jacobian according to \cite{Crean2018entropy}.

Figure~\ref{fig:vortex_mesh} shows the $p=2$ mesh with $N=4$ edges along each boundary, and Figure~\ref{fig:vortex_density} plots the discrete density of the corresponding C-SBP discretization.  All steady-vortex solutions were obtained using Newton's method combined with a sparse direct solver.

\begin{figure}[t]
  \subfigure[example mesh ($p=2$) \label{fig:vortex_mesh}]{%
    \includegraphics[width=0.49\textwidth]{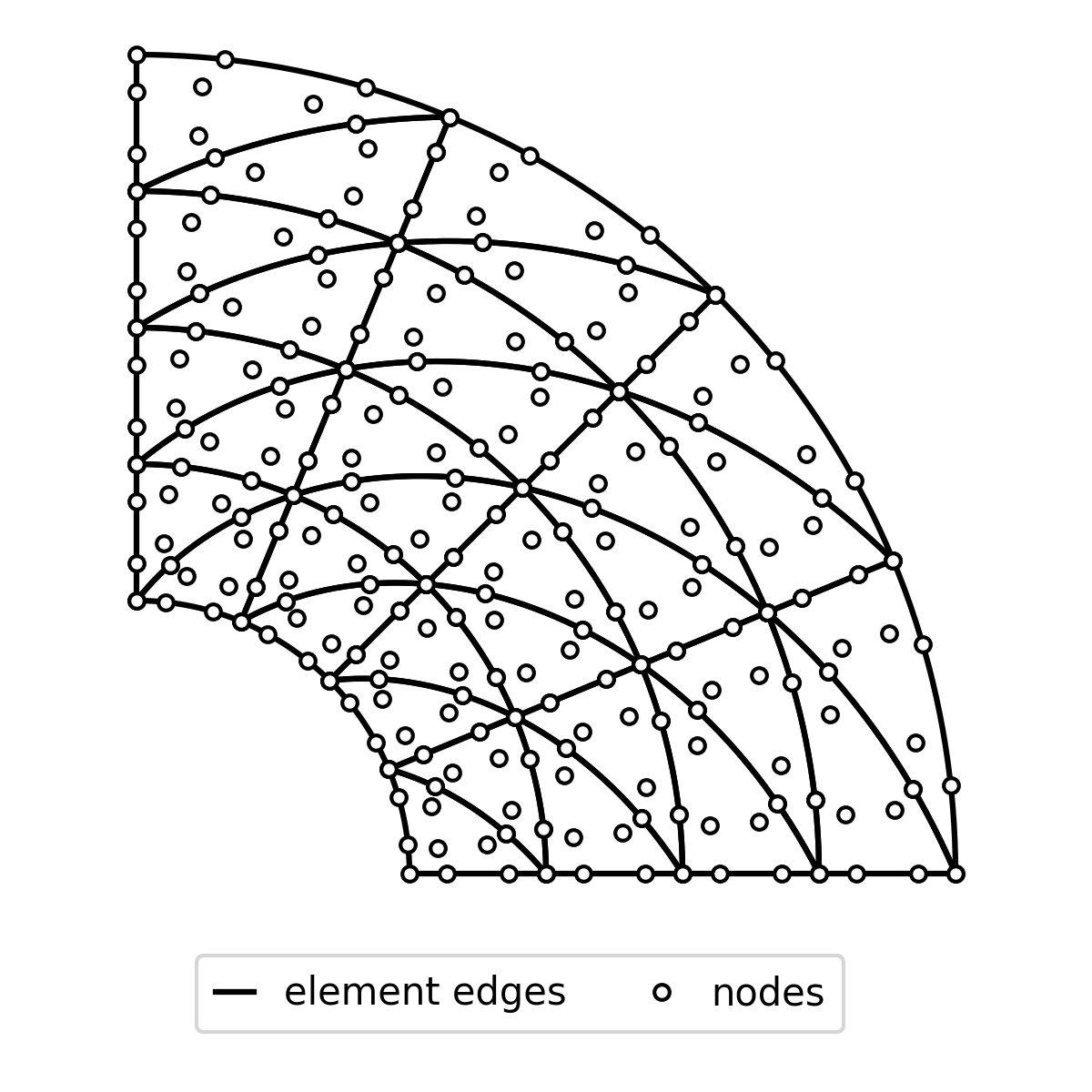}}
  \subfigure[example density ($p=2$) \label{fig:vortex_density}]{%
        \includegraphics[width=0.49\textwidth]{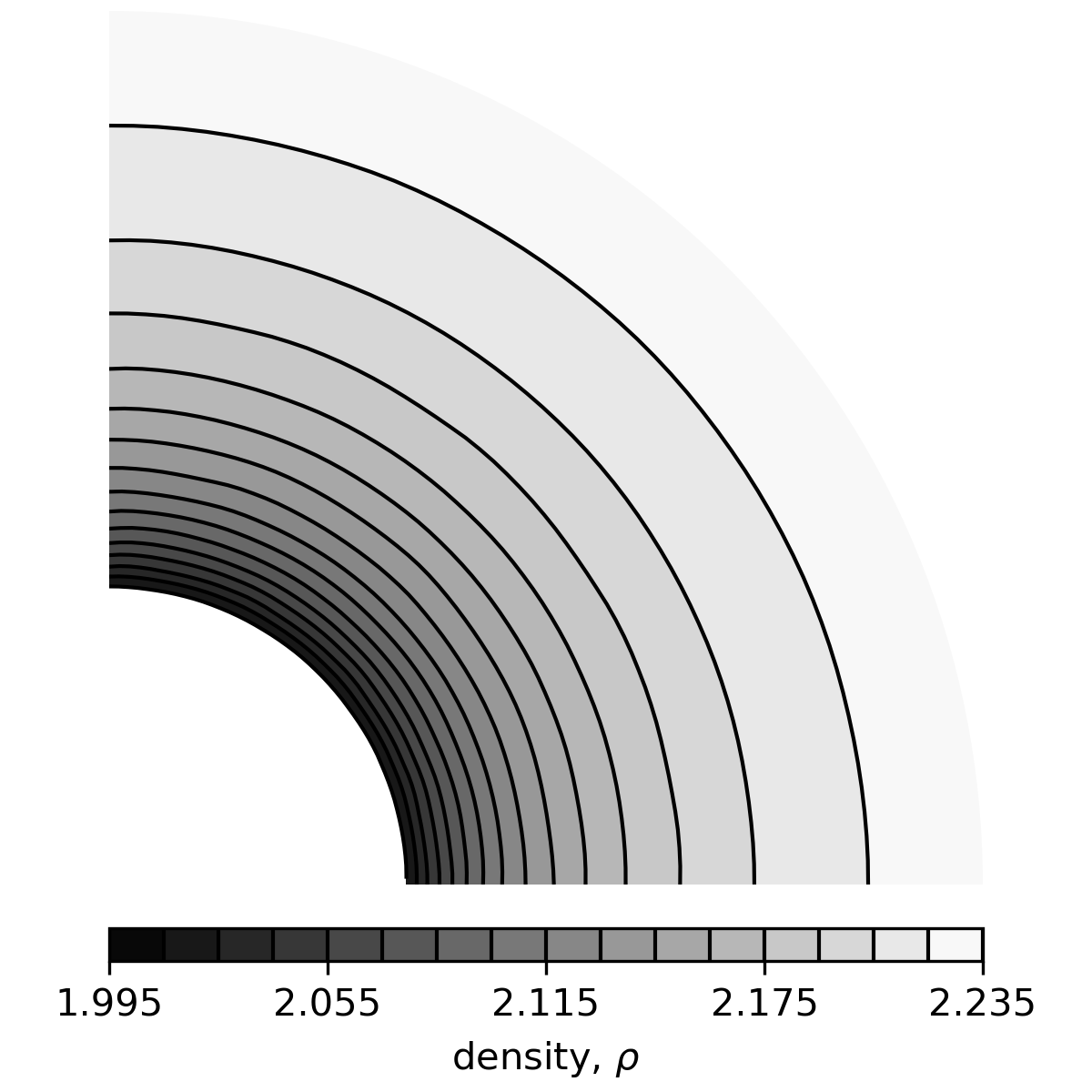}}
  \caption{Example curvilinear mesh and numerical solution for the steady-vortex accuracy study.  \label{fig:vortex}}
\end{figure}

The $L^2$ density error is plotted versus the nominal mesh size, $h = 1/N$, in Figure~\ref{fig:vortex_L2}.  I estimated the asymptotic convergence rates, which are listed under the triangles in Figure~\ref{fig:vortex_L2}, using the errors on the finest two grids.  For this problem the C-SBP schemes have rates close to $p+1$.

I conclude this accuracy study by assessing the drag force on the inner radius $r_{\textsf{in}}=1$.  Boundary functionals, such as drag, are important in many CFD applications, and a discretization is often judged based on how accurate it can estimate such outputs.  Figure~\ref{fig:vortex_drag} plots the drag error versus mesh size for the C-SBP discretizations.  For all polynomial degrees under consideration, we see that the drag is superconvergent.  For the $p=1$ and $p=2$ discretizations, the rate is approximately $2p+1$.  The asymptotic rate is less clear for the two higher-order schemes, because the drag values on the finest grid(s) are impacted by round-off errors.

\begin{figure}[t]
  \subfigure[$L^2$ density error versus $h$ \label{fig:vortex_L2}]{%
    \includegraphics[width=0.49\textwidth]{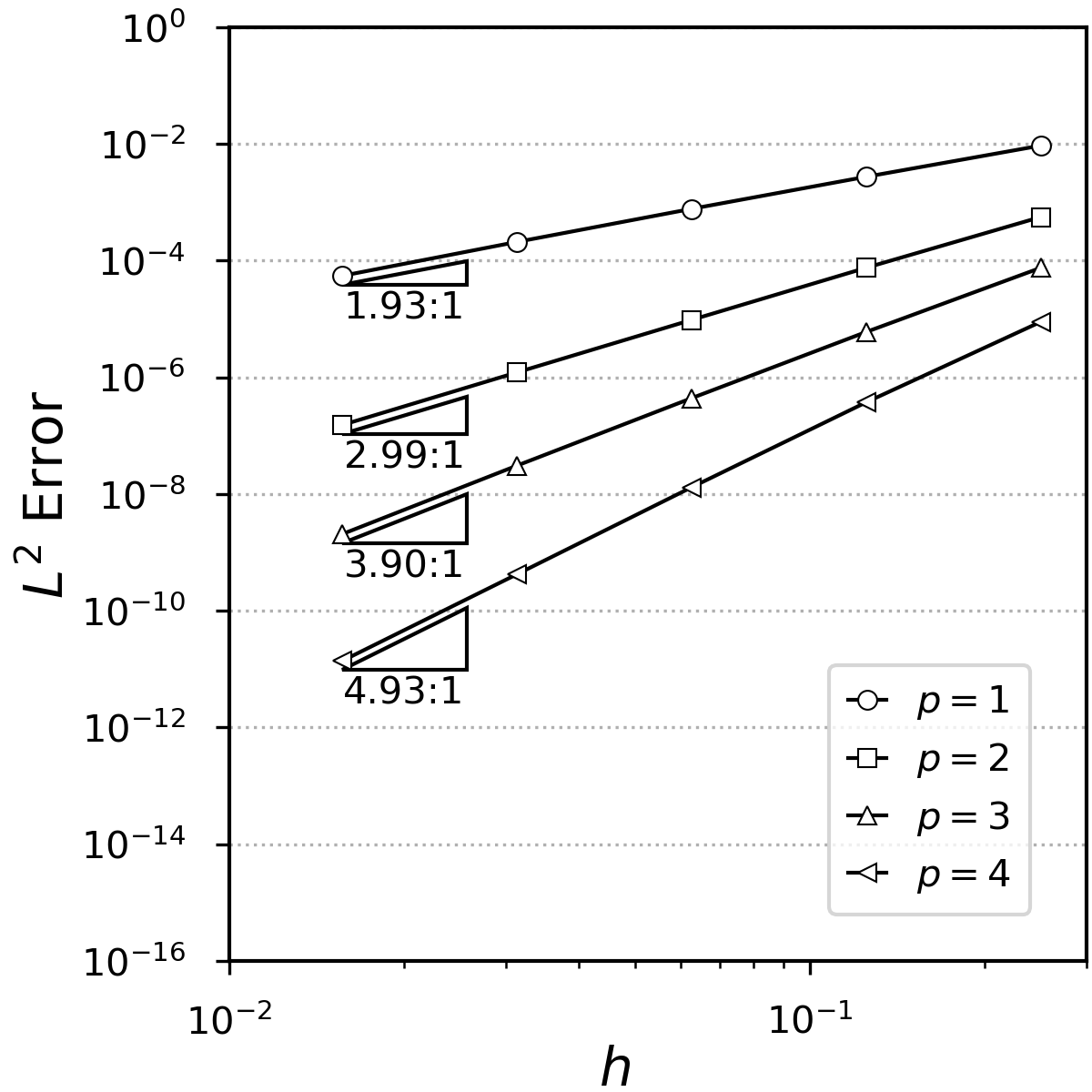}}\hfill
  \subfigure[drag error versus $h$ \label{fig:vortex_drag}]{%
        \includegraphics[width=0.49\textwidth]{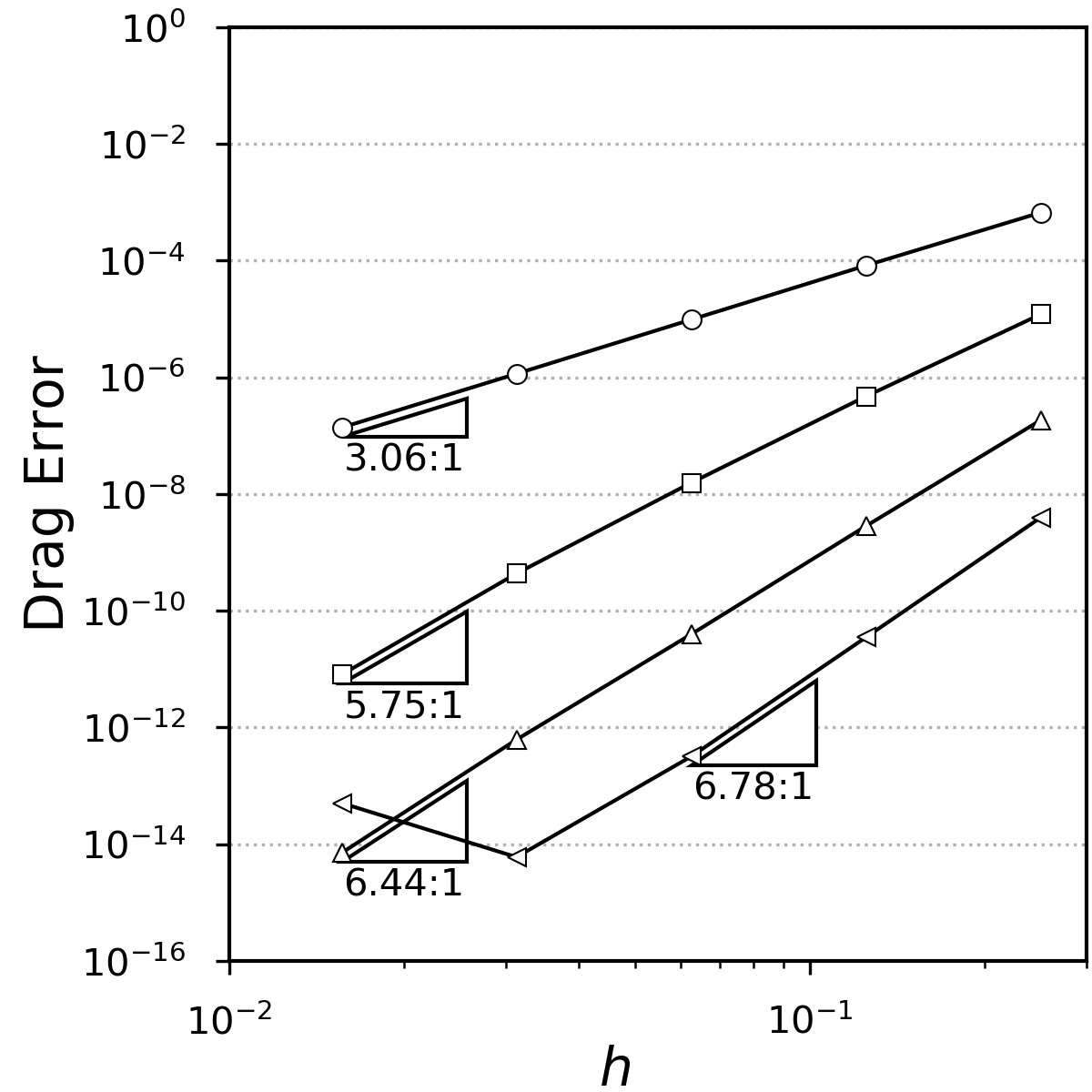}}
  \caption{  \label{fig:vortex_error}}
\end{figure}

\subsubsection{Entropy-conservation and -stability verification}

For the final set of results, I solved the discretized Euler equations on a periodic domain with a discontinuous initial condition.  The objective was to verify the entropy-conservation and entropy-stability properties of the C-SBP discretization on a non-trivial flow.  I did not consider accuracy, nor did I consider ``monotonicity'' preservation; these remain important issues in the context of discontinuous flows and will constitute future work.

The domain was the unit square, $\Omega = [0,1]^2$, with periodic boundary conditions.  The mesh generation process began by creating a $6\times 6$ uniform quadrilateral mesh in a reference space, $0\leq \xi, \eta \leq 1$.  Next, each quad was subdivided into a triangle.  As with the steady vortex case, the coordinate transformation was defined by assigning a $p+1$ Lagrange element to each triangle, and then mapping the Lagrange nodes to physical space based on the transformation
\begin{equation*}
x = \xi + \frac{1}{20} \sin(3\pi\xi)\sin(3\pi\eta),
\qquad
y = \eta - \frac{1}{20} \sin(3\pi\xi)\sin(3\pi\eta).
\end{equation*}
Figure~\ref{fig:unsteady_mesh} illustrates the $p=1$ mesh.  The meshes for the other degrees had similar shaped elements, but they obviously had more nodes per element.

The initial condition was similar to the one used in \cite{Fernandez2018staggered} and was defined by
\begin{equation*}
\bfnc{U}^T = \begin{bmatrix} \rho, & \rho u, & \rho v,& e \end{bmatrix} 
= \begin{cases}
\begin{bmatrix} 1.1, & 0, & 0, & 5.1 \end{bmatrix}, & \text{if}\; \frac{1}{3} \leq x, y  \leq \frac{2}{3}, \\[2ex]
\begin{bmatrix} 1.0, & 0, & 0, & 5.0 \end{bmatrix}, & \text{otherwise.}
\end{cases}
\end{equation*}
The analytical initial condition is shown in Figure~\ref{fig:IC_density}.  I advanced the solution in time from $t=0$ to $t=10$ non-dimensional units using the implicit midpoint rule, rather than the RK4 scheme adopted for the linear-advection study.  I found that, despite its slower rate of convergence, the midpoint rule produced a smaller entropy-conservation error for the problem and time steps that were considered.

\begin{figure}[t]
  \subfigure[example mesh ($p=1$) \label{fig:unsteady_mesh}]{%
    \includegraphics[width=0.49\textwidth]{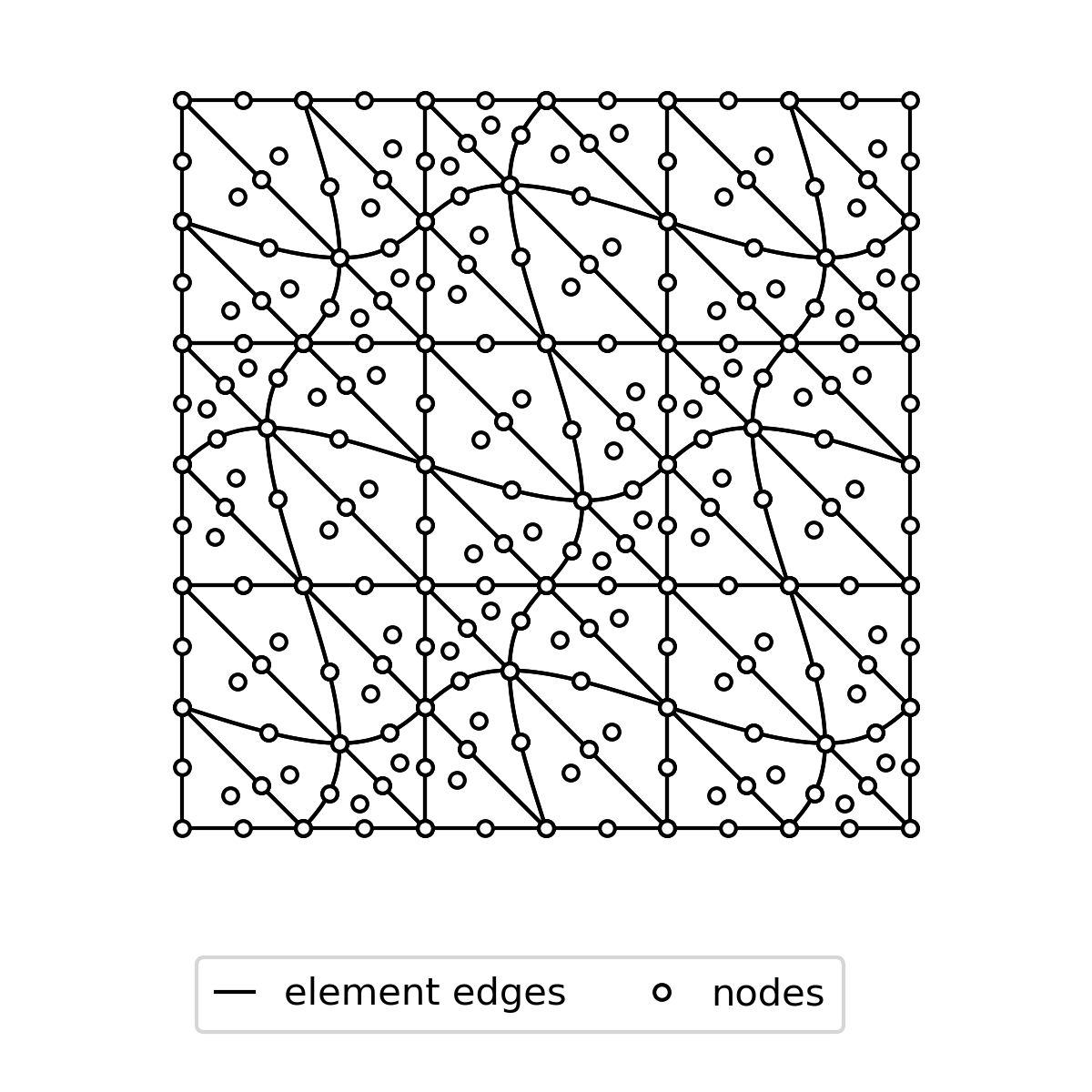}}
  \subfigure[initial density \label{fig:IC_density}]{%
        \includegraphics[width=0.49\textwidth]{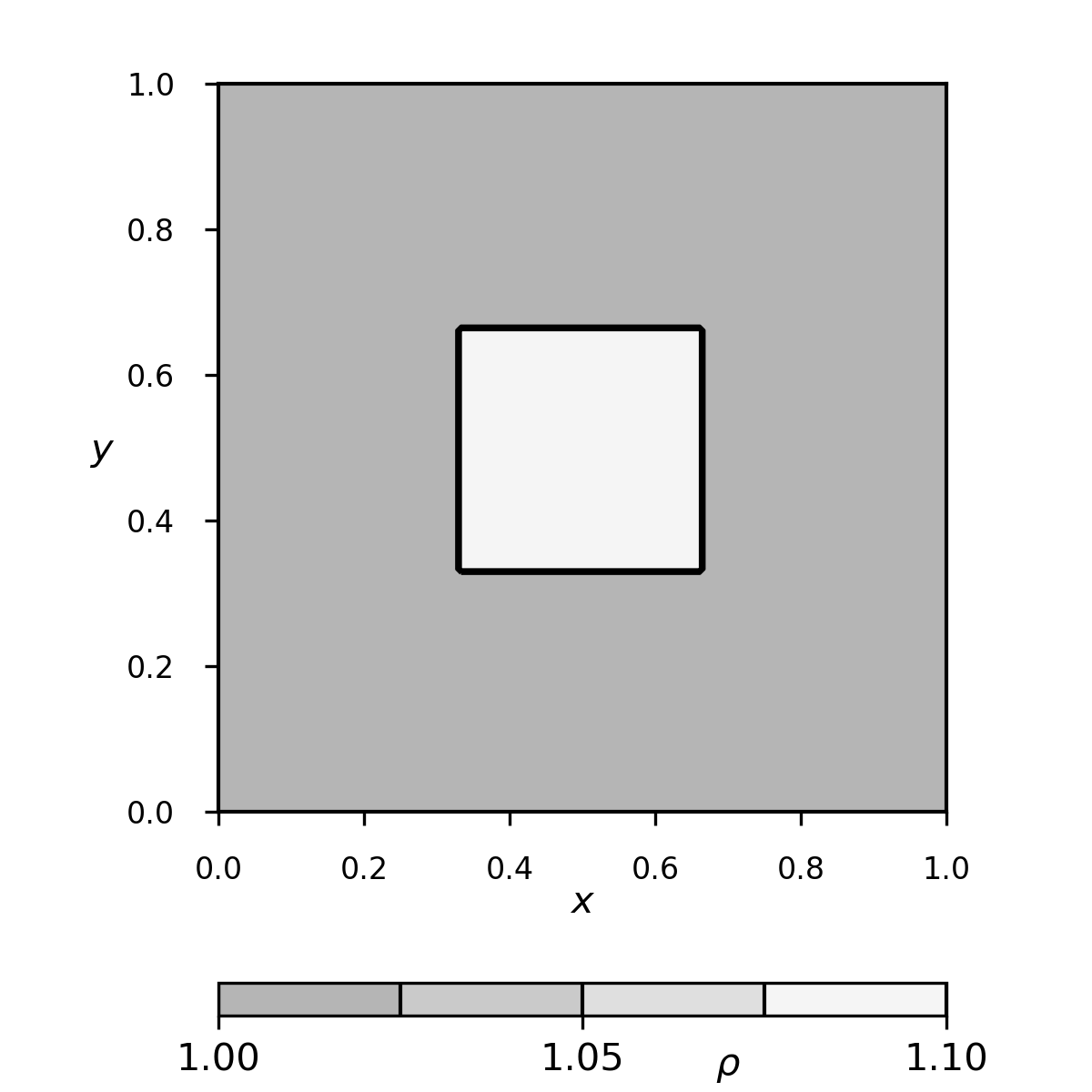}}
  \caption{Example curvilinear mesh and initial condition for the entropy-conservation and entropy-stability studies.}
\end{figure}

The set of plots in Figure~\ref{fig:entropy_cons} show the change in total entropy versus time for the entropy-conservative discretization \eqref{eq:SBPstrong}; there is one plot for each degree $p$ under consideration.  If $\bm{s}_h^{(k)}$ denotes the nodal values of the (mathematical) entropy at time step $k$, then the change in entropy is defined by
\begin{equation*}
\Delta \bm{s}_h^{(k)} \equiv \bm{1}^T \H (\bm{s}_{h}^{(k)} - \bm{s}_{h}^{(k-1)}).
\end{equation*}
Most time discretizations are not entropy conservative, so $\Delta \bm{s}_h^{(k)}$ will be non-zero even though the semi-discrete scheme \eqref{eq:SBPstrong} is entropy conservative.  This is reflected in Figure~\ref{fig:entropy_cons}, which shows that the change in entropy is indeed non-zero.  To verify that this entropy-conservation error is due to the temporal discretization, I ran the simulations using a CFL of 0.1 and 0.01.  With this change in time step size, one would expect the entropy-conservation error to decrease by two orders of magnitude, since the implicit midpoint rule is second-order accuracy.  This is confirmed by the results in Figure~\ref{fig:entropy_cons}; the magnitude in the entropy fluctuations is reduced by two orders of magnitude between CFL=0.1 and CFL=0.01, as expected.

\begin{figure}[t]
  \subfigure[ $p=1$ \label{fig:entropy_cons_p1}]{%
    \includegraphics[width=0.49\textwidth]{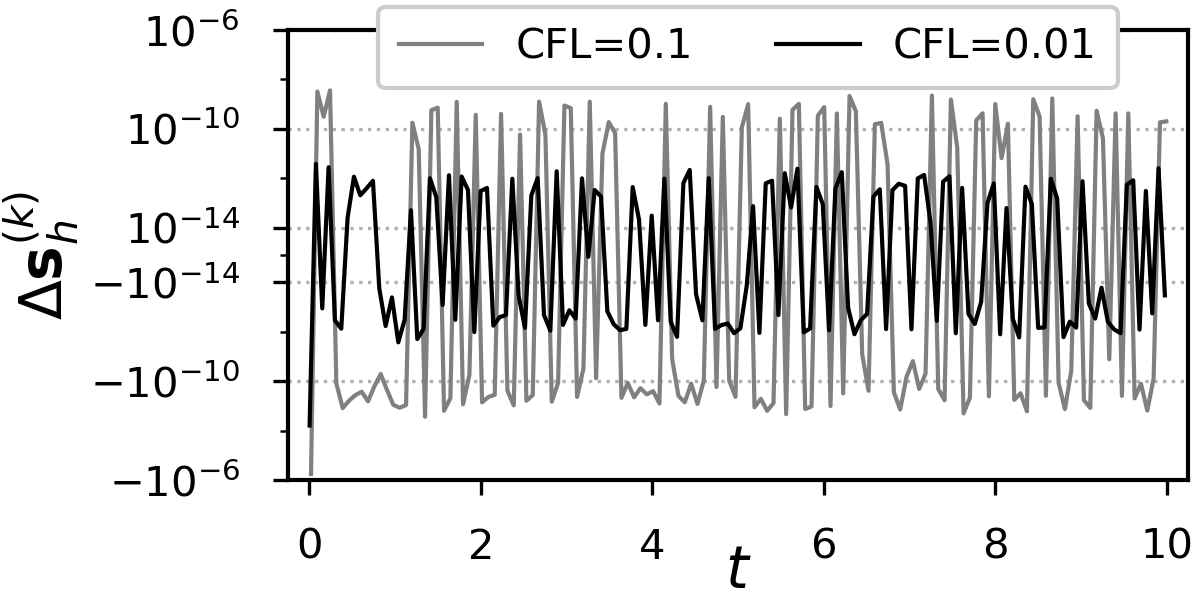}}
  \subfigure[ $p=2$ \label{fig:entropy_cons_p2}]{%
        \includegraphics[width=0.49\textwidth]{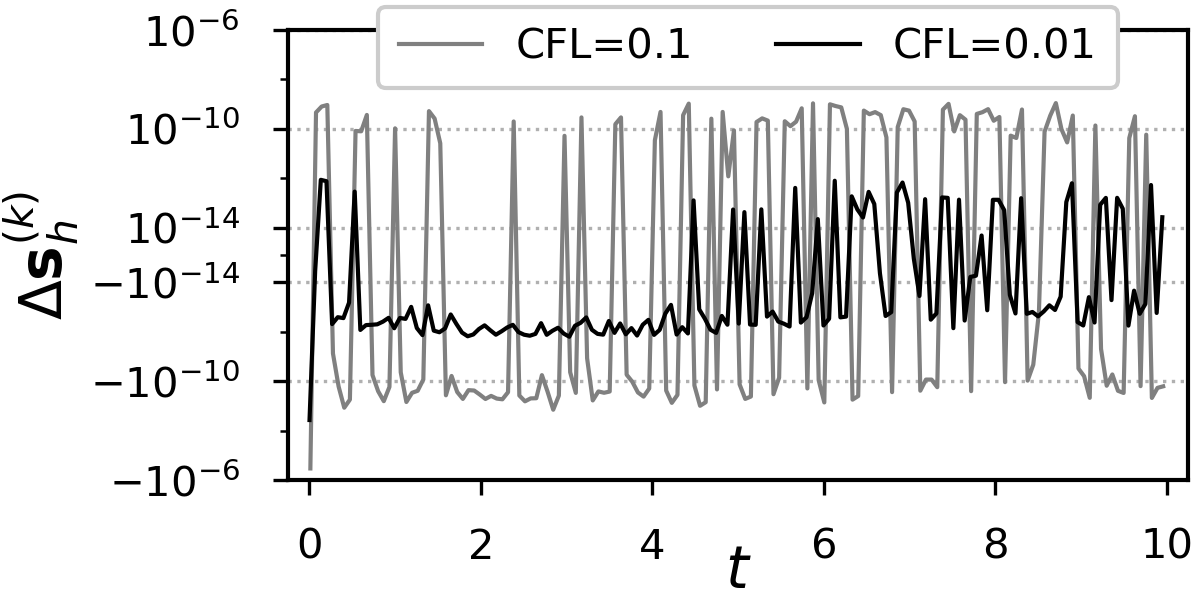}}
  \subfigure[ $p=3$ \label{fig:entropy_cons_p3}]{%
    \includegraphics[width=0.49\textwidth]{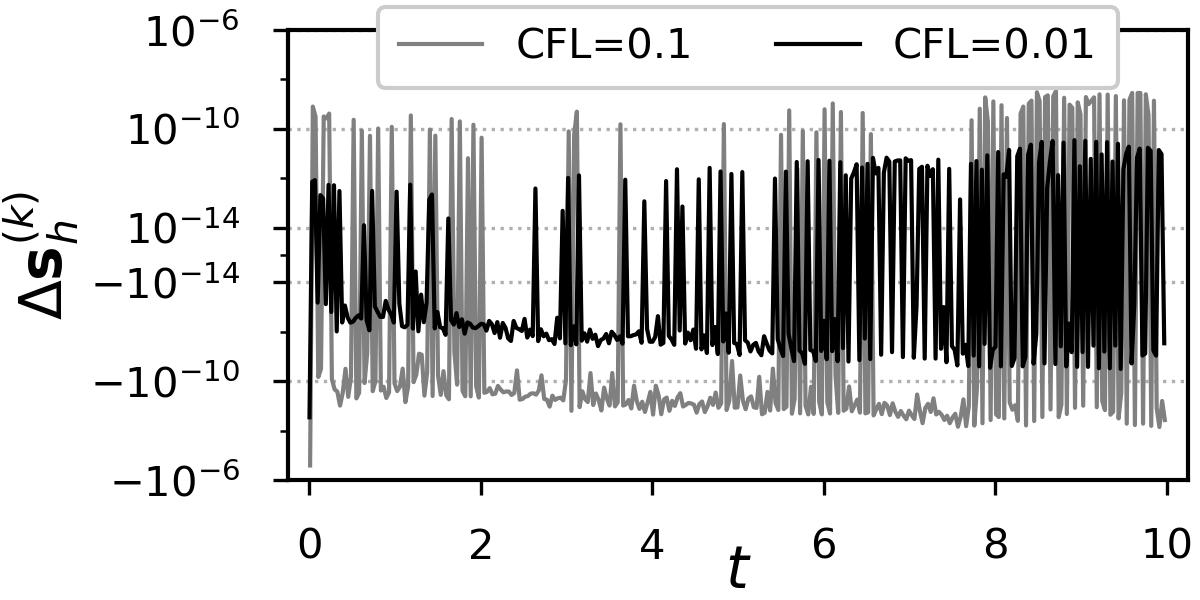}}
  \subfigure[ $p=4$ \label{fig:entropy_cons_p4}]{%
        \includegraphics[width=0.49\textwidth]{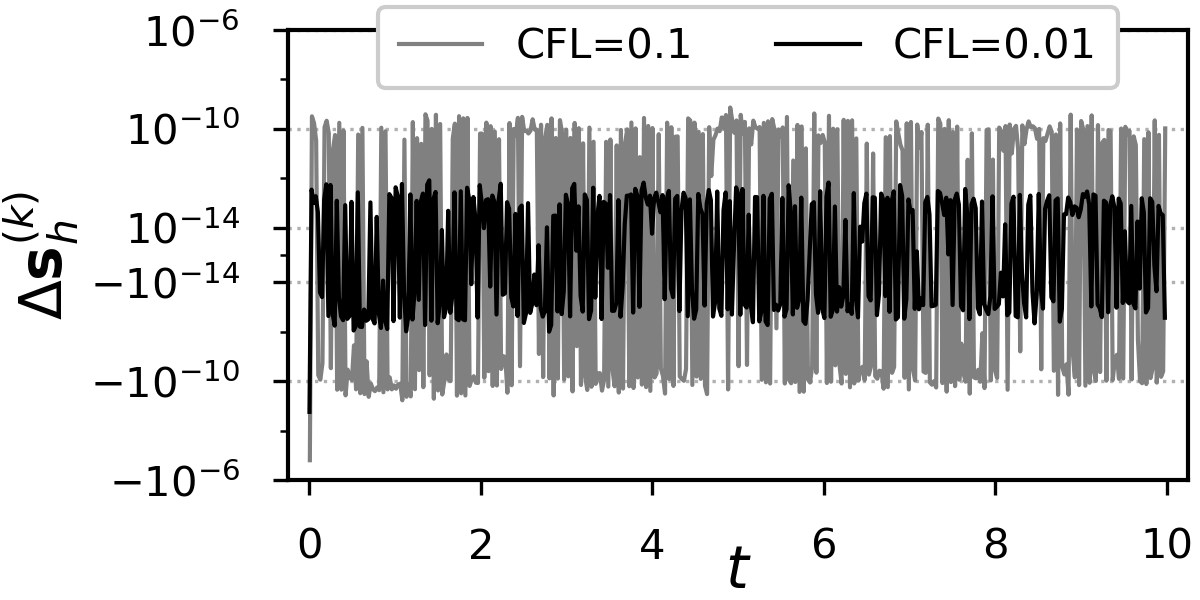}}
  \caption{Change in entropy between steps for the entropy-conservative discretizations using different CFL numbers. \label{fig:entropy_cons}}
\end{figure}

When the entropy-stable LPS terms are included, Theorem \ref{thm:ent_stab} tells us that the change in entropy should always be negative.  Again, this is only guaranteed for the semi-discrete scheme, so it is possible that a particular time discretization may lead to entropy growth.  However, this is not the case for the present simulations using the implicit midpoint rule, as Figure~\ref{fig:entropy_stab} demonstrates.  This figure plots the change in entropy from one time step to the next, and it shows that this change is always negative, so entropy is non-increasing.

\begin{figure}[t]
  \begin{center}
      \includegraphics[width=\textwidth]{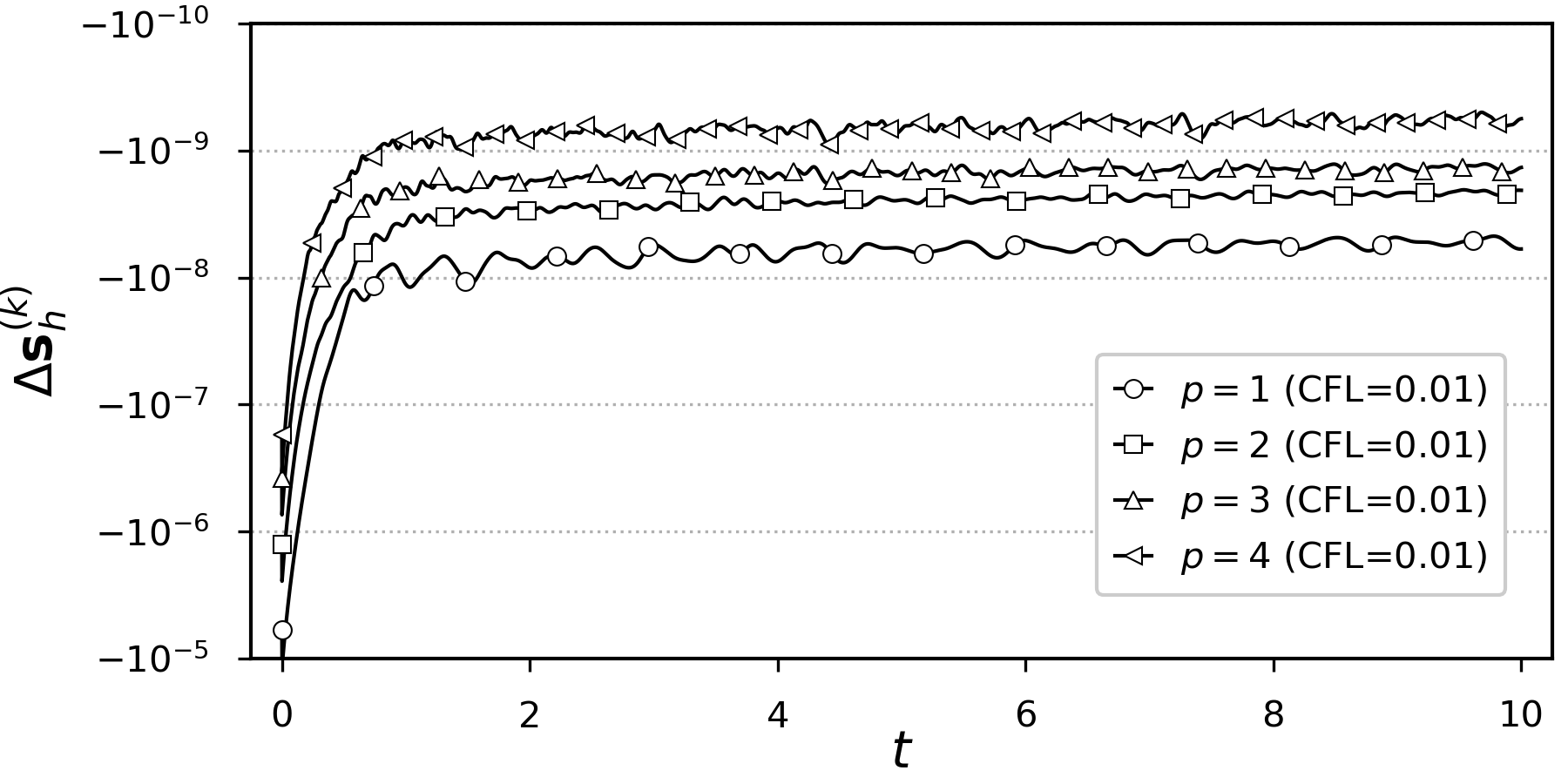}
  \caption{Change in entropy between steps for the entropy-stable C-SBP discretizations. \label{fig:entropy_stab}}
  \end{center}
\end{figure}

\section{Summary and Discussion}\label{sec:conclude}

Summation-by-parts (SBP) operators have received renewed interest in recent years, because they can be used to implement entropy-stable, high-order CFD methods.  The focus of this renewed interest has been discontinuous-Galerkin-type discretizations, 
and limited, if any, attention has been paid to continuous-Galerkin analogies, despite their efficiency for low to moderate orders of accuracy.  To address this gap, I have presented a high-order, entropy-stable C-SBP discretization that uses a continuous representation of the solution.

The baseline C-SBP discretization is neutrally stable, so an important goal of this work was to develop a stabilization that is simultaneously entropy stable, well conditioned, and element local.  To meet these requirements, I advocated the use of additional nodal degrees of freedom to enable local-projection stabilization (LPS) at the element level.  While this solution is not optimal from the perspective of approximation theory, it is well suited to diagonal-norm SBP operators, which typically require more nodes than necessary for a degree $p$ polynomial basis.

For completeness, I reviewed the construction of the SBP operators used in this work.  The operators, which are designed for triangular elements, have $2p$ exact norms, vertex nodes, and diagonal boundary operators.  I also described three methods of constructing LPS operators; one suitable for degree $2p$ exact cubatures, one suitable for $2p-1$ exact cubatures, and one suitable for finite-difference discretizations more generally.
 
I verified the C-SBP discretizations using the linear advection equation and the Euler equations.  The discretizations exhibited near optimal, $p+1$ rates of convergence, as well as superconvergent functionals when implemented in a dual-consistent manner.  The results also established that the baseline C-SBP discretization is entropy conservative, and that LPS is entropy stable when applied to the entropy variables directly.

The results suggest that C-SBP discretizations are competitive with D-SBP discretizations for moderate degree $p \leq 4$ discretizations, where the interface penalties in D-SBP methods represent a significant fraction of the total computational expense.  Thus, I suspect that the C-SBP discretization will be most attractive when the discretization error is dominated by errors in some phenomenological model(s), such as a turbulence closure.  At the very least, the continuous approach warrants a closer examination than it has in the past.

\section{Acknowledgments}

My sincerest thanks to my students --- Anthony Ashley, Tucker Babcock, Garo Bedonian, Luiz Cagliari, Jared Crean, Sharanjeet Kaur, Kinshuk Panda, Ge Yan, and Jianfeng Yan --- for their feedback on an early draft of this paper.

All the results in this paper were obtained using software written in Julia~\cite{Bezanson2017julia}.  The plots were generated using Matplotlib~\cite{Hunter2007matplotlib}, with help from the Numpy~\cite{Oliphant2006guide,Van2011numpy} and Scipy~\cite{Jones2001Scipy} libraries.

\bibliographystyle{spmpsci}
\bibliography{references}

\end{document}